\documentclass[12pt]{article}
\usepackage{amssymb}
\usepackage{amsthm}
\usepackage{amsmath}
\usepackage{bbm}
\usepackage{amscd}
\usepackage{stmaryrd}
\usepackage{array}
\usepackage{url}
\usepackage{hyperref}
\usepackage[latin2]{inputenc}
\usepackage{t1enc}
\usepackage{enumerate}
\usepackage[mathscr]{eucal}
\usepackage[british]{babel}
\usepackage[dvips]{graphicx}
\usepackage[dvips]{epsfig}
\usepackage{epsf}
\usepackage{color}
\usepackage{indentfirst}
\usepackage[all]{xy}

\newcommand{\Ker}{\mathrm{Ker}}
\newcommand{\Zp}{\mathbb{Z}_p}
\newcommand{\Qp}{\mathbb{Q}_p}

\newcommand{\val}{\mathrm{val}}
\newcommand{\diag}{\mathrm{diag}}
\newcommand{\genlength}{{\mathrm{length}_{gen}}}

\newcommand{\Tr}{\mathrm{Tr}}
\newcommand{\Ext}{\mathrm{Ext}}
\newcommand{\Hom}{\mathrm{Hom}}

\newcommand{\Ind}{\mathrm{Ind}}
\newcommand{\Fil}{\mathrm{Fil}}
\newcommand{\GL}{\mathrm{GL}}
\newcommand{\Gal}{\mathrm{Gal}}
\newcommand{\res}{\mathrm{res}}
\newcommand{\id}{\mathrm{id}}

\newcommand{\rk}{\mathrm{rk}}
\newcommand{\pr}{\mathrm{pr}}

\newcommand{\Coker}{\mathrm{Coker}}

\newcommand{\gr}{\mathrm{gr}^{\cdot}}
\newcommand{\bg}{(\hspace{-0.06cm}(}
\newcommand{\jg}{)\hspace{-0.06cm})}

\newcommand{\bs}{[\hspace{-0.04cm}[}
\newcommand{\js}{]\hspace{-0.04cm}]}
\newtheorem{thm}{Theorem}[section]
\newtheorem{thma}{Theorem}

\newtheorem{pro}[thm]{Proposition}
\newtheorem{lem}[thm]{Lemma}
\newtheorem{cor}[thm]{Corollary}

\newtheorem{que}{Question}

\theoremstyle{definition}
\newtheorem*{rem}{Remark}
\newtheorem*{rems}{Remarks}

\begin{document}
\date{\today}
\title{Multivariable $(\varphi,\Gamma)$-modules and smooth $o$-torsion representations}
\author{Gergely Z\'abr\'adi \footnote{I would like to thank the MTA Alfr\'ed R\'enyi Institute of Mathematics for its hospitality where parts of this work was written. This research was supported by a Hungarian OTKA Research grant K-100291 and by the J\'anos Bolyai Scholarship of the Hungarian Academy of Sciences.}}
\maketitle

\begin{abstract}
Let $G$ be a $\mathbb{Q}_p$-split reductive group with connected centre and Borel subgroup $B=TN$. We construct a right exact functor $D^\vee_\Delta$ from the category of smooth modulo $p^n$ representations of $B$ to the category of projective limits of finitely generated \'etale $(\varphi,\Gamma)$-modules over a multivariable (indexed by the set of simple roots) commutative Laurent-series ring. These correspond to representations of a direct power of $\Gal(\overline{\mathbb{Q}_p}/\mathbb{Q}_p)$ via an equivalence of categories. Parabolic induction from a subgroup $P=L_PN_P$ gives rise to a basechange from a Laurent-series ring in those variables with corresponding simple roots contained in the Levi component $L_P$. $D^\vee_\Delta$ is exact and yields finitely generated objects on the category $SP_A$ of finite length representations with subquotients of principal series as Jordan-H\"older factors. Lifting the functor $D^\vee_\Delta$ to all (noncommuting) variables indexed by the positive roots allows us to construct a $G$-equivariant sheaf $\mathfrak{Y}_{\pi,\Delta}$ on $G/B$ and a $G$-equivariant continuous map from the Pontryagin dual $\pi^\vee$ of a smooth representation $\pi$ of $G$ to the global sections $\mathfrak{Y}_{\pi,\Delta}(G/B)$. We deduce that $D^\vee_\Delta$ is fully faithful on the full subcategory of $SP_A$ with Jordan-H\"older factors isomorphic to irreducible principal series.
\end{abstract}

\tableofcontents

\section{Introduction}

\subsection{Background and motivation}

By now the $p$-adic Langlands correspondence for $\GL_2(\mathbb{Q}_p)$ is very well understood through the work of Berger \cite{Be}, Breuil \cite{B03a,B03b,BE}, Colmez \cite{Mira}, \cite{C}, Emerton \cite{E}, Kisin \cite{K1}, Pa\v{s}k\=unas \cite{P} (see \cite{B1} for an overview). The starting point of Colmez's work is Fontaine's \cite{F} theorem that the category of modulo $p^h$ Galois representations of $\mathbb{Q}_p$ is equivalent to the category of \'etale $(\varphi,\Gamma)$-modules over $\mathbb{Z}/p^h\bg X\jg $. One of Colmez's breakthroughs was that he managed to relate smooth modulo $p^h$ representations (therefore also continuous $p$-adic representations by letting $h\to\infty$ and inverting $p$) of $\GL_2(\mathbb{Q}_p)$ to $(\varphi,\Gamma)$-modules, too. The so-called ``Montr\'eal-functor'' associates to a smooth mod $p^h$ representation $\pi$ of $\GL_2(\mathbb{Q}_p)$ (first restricting $\pi$ to a Borel subgroup $B_2(\mathbb{Q}_p)$) an \'etale $(\varphi,\Gamma)$-module over $\mathbb{Z}/p^h\bg X\jg$. By Pa\v{s}k\={u}nas's work \cite{P} this induces a bijection for certain $p$-adic Banach space representations of $\GL_2(\mathbb{Q}_p)$.

There have been attempts, for instance by Schneider and Vigneras \cite{SVig}, to generalize Colmez's functor to other $\mathbb{Q}_p$-split reductive groups $G$. More recently, Breuil \cite{B} (in a slightly more general setting allowing finite extensions of $\mathbb{Q}_p$, too) introduced a functor $D^\vee_\xi=D^\vee_{\xi,\ell}$ from smooth $\mathbb{Z}/p^h$-representations of $G$ to projective limits of \'etale $(\varphi,\Gamma)$-modules. The construction depends on the choice of a cocharacter $\xi\colon\mathbf{G}_m\to T$ (with the property that the composition of $\xi$ with all simple roots $\alpha\in\Delta$ is an isomorphism of $\mathbf{G}_m$) and on a Whittaker type functional $\ell$ from the unipotent radical $N$ of a Borel subgroup $B=TN$ to $\mathbb{Q}_p$. In \cite{B} (and also in \cite{SVig}) $\ell$ is assumed to be \emph{generic}, ie.\ $\ell$ induces an isomorphism $N_{\alpha}\to\mathbb{Q}_p$ for the root subgroups $N_\alpha$ of all simple roots $\alpha\in \Delta$ with respect to $B$. The action of $\varphi$ (resp.\ of $\Gamma\cong\Zp^\times$) on $D^\vee_\xi(\pi)$ for a smooth mod $p^h$ representation $\pi$ of $G$ comes from the (inverse of the) action of $\xi(p)$ (resp.\ of $\xi(\Zp^\times)$) on $\pi$. The functor $D^\vee_{\xi,\ell}$ has very promising properties: it is right exact and compatible with tensor products and with parabolic induction. Moreover, $D^\vee_{\xi,\ell}$ is exact and produces finitely generated objects on the category $SP_A$ of finite length representations with all Jordan-H\"older factors appearing as a subquotient of principal series representations (ie.\ of $\Ind_B^G\chi$ for some character $\chi$ of $T$). Finally, $D^\vee_{\xi,\ell}$ is compatible with the conjectures in \cite{BH} made from a global point of view. The assumption on the genericity of $\ell$ is needed crucially for some of these properties, in particular for the exactness on $SP_A$ and for the compatibility with \cite{BH}. However, if $\ell$ is a generic Whittaker functional then the functor $D^\vee_{\xi,\ell}$ loses a lot of information, one cannot possibly recover the representation $\pi$ from the attached $(\varphi,\Gamma)$-module $D^\vee_{\xi,\ell}(\pi)$ (by the methods developed in \cite{SVZ} or otherwise). This has also been predicted by the work of Breuil and Pa\v{s}k\={u}nas \cite{BP}: when one moves beyond $\GL_2(\mathbb{Q}_p)$ then there are much more representations on the automorphic side than on the Galois side. So if we would like to have a bijection for some large class of representations on the reductive group side, we need to put additional data on our Galois-representations. One candidate is that we could perhaps equip the Galois representation with an additional character of the torus $T/\xi(\mathbb{Q}_p^\times)$ extending the action of $\varphi$ and $\Gamma$. The heuristics for this is that even in the case of $\GL_2(\mathbb{Q}_p)$ a central character appears naturally on the attached $(\varphi,\Gamma)$-module. However, if $\ell$ is generic then the action of $\varphi$ and $\Gamma$ on $D^\vee_{\xi,\ell}(\pi)$ cannot be extended to the dominant submonoid $T_+\subset T$ since in this case the kernel $H_{gen}=\Ker(\ell\colon N\to \mathbb{Q}_p)$ is not invariant under the conjugation action of any larger subgroup of $T$ than the product of the image of $\xi$ and the centre. On the other hand, if we choose $\ell$ to be very far from being generic, ie.\ $\ell=\ell_\alpha$ is the projection onto a root subgroup $N_\alpha$ for some simple root $\alpha\in\Delta$ then we do have an additional action of $T_+$ on $D^\vee_{\xi,\ell}(\pi)$ as shown by the present author and Erd\'elyi \cite{EZ}. Moreover, in op.\ cit.\ a natural transformation $\beta_{G/B,\cdot}$ from the functor $\pi\mapsto\pi^\vee$ (taking Pontryagin duals) to the global sections $\mathfrak{Y}_{\alpha,\pi}(G/B)$ of a $G$-equivariant sheaf $\mathfrak{Y}_{\alpha,\pi}$ on the flag variety $G/B$ associated to the \'etale $T_+$-module $D^\vee_{\xi,\ell}(\pi)$ is constructed for the choice $\ell=\ell_\alpha$. The map $\beta_{G/B,\pi}\colon \pi^\vee\to \mathfrak{Y}_{\alpha,\pi}(G/B)$ is nonzero whenever $D^\vee_{\xi,\ell}(\pi)$ is nonzero. However, as mentioned above, for non-generic $\ell$ the functor $D^\vee_{\xi,\ell}$ does not have so good exactness and compatibility properties. 

The goal of this paper is to combine all the mentioned good properties of the above approaches. In order to do this we are going to use \emph{multivariable} $(\varphi,\Gamma)$-modules in the variables $X_\alpha$ ($\alpha\in\Delta$). More concretely, consider the Laurent series ring $\mathbb{Z}/p^h \bg X_{\alpha}\mid \alpha\in\Delta\jg:=\mathbb{Z}/p^h \bs X_\alpha,\alpha\in\Delta\js[X_\alpha^{-1}\mid\alpha\in\Delta]$ with the conjugation action of the monoid $T_+:=\{t\in T\mid \alpha(t)\in\Zp\text{ for all }\alpha\in\Delta\}$. In an analogous way to \cite{B} we construct a functor $D^\vee_\Delta$ from smooth mod $p^h$-representations of a $\mathbb{Q}_p$-split connected reductive group $G$ with connected centre to the category of projective limits of finitely generated \'etale $T_+$-modules over $\mathbb{Z}/p^h \bg X_{\alpha}\mid \alpha\in\Delta\jg$. Moreover, in \cite{Z16} a pair $\mathbb{D}$ and $\mathbb{V}$ of quasi-inverse equivalences of categories is constructed between the category of continuous mod $p^n$ representations of the $|\Delta|$th direct power of the Galois group $\Gal(\overline{\mathbb{Q}_p}/\mathbb{Q}_p)$ (endowed with a character of $\mathbb{Q}_p^\times$) and multivariable \'etale $T_+$-modules. One can pass to usual $(\varphi,\Gamma)$-modules by identifying the variables $X_\alpha$ with each other---this step corresponds to the restriction of a representation of $\Gal(\overline{\mathbb{Q}_p}/\mathbb{Q}_p)^{|\Delta|}$ to the diagonal embedding of $\Gal(\overline{\mathbb{Q}_p}/\mathbb{Q}_p)$ (Cor.\ 3.10 in \cite{Z16}). When doing so we must forget the action of the monoid $T_+$ just keeping the action of $\varphi^{\mathbb{N}}\Gamma=\xi(\Zp\setminus\{0\})\subset T_+$ (or possibly also the action of the centre of $G$) as the kernel $(X_\alpha-X_\beta\mid \alpha,\beta\in\Delta)$ of this identification is not stable under $T_+$. This assignment is faithful and exact in general, but definitely not full. In all known cases---including objects in the category $SP_A$ and parabolically induced representations from the product of copies of $\GL_2(\mathbb{Q}_p)$ and a torus---the resulting Galois representation will coincide with Breuil's $\mathbb{V}_F\circ D^\vee_{\xi,\ell}(\pi)$ (for generic $\ell$) where $\mathbb{V}_F$ stands for Fontaine's equivalence. Whether or not this is true in general is an open question. If $P=L_PN_P$ is a standard parabolic subgroup with Levi component $L_P$ isomorphic to the product of copies of $\GL_2(\mathbb{Q}_p)$ and a torus, then the value of $\mathbb{V}\circ D^\vee_\Delta$ at parabolically induced representations $\Ind_P^G\pi_P$ is well-described in terms of tensor product of Galois representations for each $\alpha\in \Delta$. It can be shown using the fully faithful property of $\mathbb{D}$ that the resulting multivariable $(\varphi,\Gamma)$-modules $D_\Delta^\vee(\Ind_P^G\pi_P)$ are therefore pairwise non-isomorphic for all the irreducible mod $p$ representations of $G$ arising this way (for varying $P$ of this form).

Apart from all the above mentioned exactness and compatibility properties $D^\vee_\Delta$ has the following additional features: induction from a parabolic subgroup $P=L_PN_P$ corresponds to basechange from $\mathbb{Z}/(p^n)\bg X_\alpha\mid\alpha\in\Delta_P\jg$ to $\mathbb{Z}/p^h \bg X_{\alpha}\mid \alpha\in\Delta\jg$ where $\Delta_P\subseteq \Delta$ consists of those simple roots whose root subgroups are contained in the Levi component $L_P$. On the Galois side this means, in particular, that the copy of the Galois group $\Gal(\overline{\mathbb{Q}_p}/\mathbb{Q}_p)$ corresponding to those simple roots $\alpha\in\Delta$ whose root subgroup is not contained in the Levi $L_P$ acts on $\mathbb{V}\circ D^\vee_\Delta(\Ind_P^G\pi_P)$ via a character. This could hopefully lead to detecting $P$ from the attached $T_+$-module over $\mathbb{Z}/p^h \bg X_{\alpha}\mid \alpha\in\Delta\jg$. Another promising property of $D^\vee_\Delta$ is that we can indeed recover successive extensions $\pi$ of irreducible principal series representations from $D^\vee_\Delta(\pi)$. In other words we show---using the methods of \cite{SVZ} \cite{EZ} realizing $\pi^\vee$ as a $G$-invariant subspace of the global sections of a $G$-equivariant sheaf on $G/B$---that $D^\vee_\Delta$ is fully faithful on the category $SP_A^0$ of these representations. By the aforementioned work of Breuil and Pa\v{s}k\={u}nas \cite{BP} we cannot expect a bijection between smooth $\mathbb{Z}/p^h $-representations of $G$ and mod $p^h$ Galois representations of $\mathbb{Q}_p$. However, this work could be considered as evidence that there might still be a bijection between a large class of smooth $\mathbb{Z}/p^h $-representations of $G$ and certain representations of $\Gal(\overline{\mathbb{Q}_p}/\mathbb{Q}_p)^{|\Delta|}$.

Moreover, Breuil, Herzig, and Schraen \cite{B16p} predict that Breuil's functor $\mathbb{V}_F\circ D^\vee_{\xi,\ell}(\pi)$ on a representation of $\GL_n(\mathbb{Q}_p)$ built out from some mod $p$ Hecke isotopic subspace would give something like an ``internal'' tensor product $\rho \otimes_{\mathbb{F}_p} \wedge^2(\rho) \otimes_{\mathbb{F}_p} \cdots \otimes_{\mathbb{F}_p} \wedge^{n-1}(\rho)$ for some local Galois representation $\rho$ of dimension $n$ furnished by the global theory. Now the functor $\mathbb{V}\circ D^\vee_\Delta$ should give the same, but ``external'' tensor product, instead of internal (ie.\ different copies of $\Gal(\overline{\mathbb{Q}_p}/\mathbb{Q}_p)$ for each term in the tensor product). This could perhaps explain why the individual $\wedge^i(\rho)$ appear in the Shimura cohomology of unitary groups of type $U(i,n-i)$, but not their internal tensor product.

Another motivation is that the Robba versions of multivariable $(\varphi,\Gamma)$-modules seem to play a role \cite{Ber} \cite{Ber2} \cite{Ked} in the case of the $p$-adic Langlands programme for $\GL_2(F)$ for finite extensions $F\neq\mathbb{Q}_p$, too.

\subsection{Notations}

Let $G=\mathbf{G}(\mathbb{Q}_p)$ be the $\mathbb{Q}_p$-points of a $\mathbb{Q}_p$-split connected reductive group $\mathbf{G}$ defined over $\Zp$ with connected centre and a fixed split Borel subgroup $\mathbf{B}=\mathbf{TN}$. Put $B:=\mathbf{B}(\mathbb{Q}_p)$, $T:=\mathbf{T}(\mathbb{Q}_p)$, and $N:=\mathbf{N}(\mathbb{Q}_p)$. We denote by $\Phi^+$ the  set  of  roots of $T$ in $N$, by $\Delta\subset \Phi^+$ the set of simple roots, and by $u_\alpha :\mathbb G_a \to  N_\alpha$, for $\alpha \in \Phi^+$,  a  $\mathbb Q_p$-homomorphism onto the root subgroup $N_\alpha$ of $N$ such that $tu_\alpha (x) t^{-1}=  u_\alpha(\alpha (t) x)$ for $x\in \mathbb Q_p$ and $t\in T(\mathbb Q_p)$, and   $N_0=\prod_{\alpha\in \Phi^+} u_\alpha (\mathbb Z_p)$  is a compact open subgroup of $N(\mathbb Q_p) $. We put $n_\alpha:=u_\alpha(1)$ and $N_{\alpha,0}:=u_{\alpha}(\mathbb{Z}_p)$ for the image of $u_{\alpha}$ on $\mathbb{Z}_p$.
We denote by $T_{+} $ the monoid of dominant elements  $t$ in $T$ such that $ \val_p(\alpha(t))\geq 0$ for all $\alpha \in \Phi^+$,   by $T_0\subset T_+$ the maximal subgroup, and we put $B_+=N_0T_+, B_0=N_0T_0$.  

Let $K$ be a finite extension of $\mathbb{Q}_p$ with ring of integers $o$, uniformizer $\varpi$, and residue field $\kappa:=o/\varpi$. By a smooth $o$-torsion representation $\pi$ of $G$ (resp.\ of $B$) we mean a torsion $o$-module $\pi$ together with a smooth (ie.\ stabilizers are open) and linear action of the group $G$ (resp.\ of $B$). We will consider representations $\pi$ with $\varpi^h\pi=0$ for some $h\geq 1$ and put $A:=o/\varpi^h$.

The natural conjugation action of  $T_+$ on $N_0$ extends to an action on the Iwasawa $A$-algebra $A\bs N_0\js$. For $t\in T_+$ we denote this action of $t$ on $A\bs N_0\js$ by $\varphi_t$. The map $\varphi_t\colon A\bs N_0\js\to A\bs N_0\js$ is an injective ring homomorphism with a distinguished left inverse $\psi_t\colon A\bs N_0\js\to A\bs N_0\js$ satisfying $\psi_t\circ\varphi_t=\id_{A\bs N_0\js}$ and $\psi_t(u\varphi_t(\lambda))=\psi_t(\varphi_t(\lambda)u)=0$ for all $u\in N_0\setminus tN_0t^{-1}$ and $\lambda\in A\bs N_0\js$. Further, the normal subgroup $H_{\Delta,0}:=\prod_{\beta\in\Phi^+\setminus\Delta}N_{\beta,0}$ is invariant under the action of $T_+$ so the quotient group $N_{\Delta,0}:=N_0/H_{\Delta,0}\cong \prod_{\alpha\in\Delta}N_{\alpha,0}$ also inherits the action of $T_+$. The Iwasawa algebra $A\bs N_{\Delta,0}\js$ can be identified with the multivariable power series ring $A\bs X_\alpha\mid \alpha\in\Delta\js$ by the map $n_\alpha-1\mapsto X_\alpha$ ($\alpha\in\Delta$). We define $A\bg N_{\Delta,0}\jg$ as the localization $A\bs N_{\Delta,0}\js [X_\alpha^{-1},\alpha\in\Delta]$. We also denote by $\varphi_t\colon A\bs N_{\Delta,0}\js\to A\bs N_{\Delta,0}\js$ (resp.\ $\varphi_t\colon A\bg N_{\Delta,0}\jg \to  A\bg N_{\Delta,0}\jg$) the induced action of $t\in T_+$ on these rings. By an \'etale $T_+$-module over $A\bg N_{\Delta,0}\jg$ we mean a (unless otherwise mentioned) finitely generated module $M$ over $A\bg N_{\Delta,0}\jg$ together with a semilinear action of the monoid $T_+$ (also denoted by $\varphi_t$ for $t\in T_+$) such that the maps $$\id\otimes\varphi_t\colon \varphi_t^*M:=A\bg N_{\Delta,0}\jg\otimes_{A\bg N_{\Delta,0}\jg,\varphi_t}M\to M$$ are isomorphisms for all $t\in T_+$.

Since the centre of $G$ is assumed to be connected, there exists a cocharacter $\lambda_{\alpha^\vee}\colon \mathbb{Q}_p^\times\to T$ such that $\alpha\circ\lambda_{\alpha^\vee}$ is the identity on $\mathbb{Q}_p^\times$ for each $\alpha\in\Delta$ and $\beta\circ\lambda_{\alpha^\vee}=1$ for all $\beta\neq\alpha\in\Delta$. Note that $\lambda_{\alpha^\vee}$ is only unique up to a cocharacter of the centre $Z(G)$. We put $\xi:=\sum_{\alpha\in\Delta}\lambda_{\alpha^\vee}$, $\Gamma:=\xi(\mathbb{Z}_p^\times)\leq T$, and often denote the action of $s:=\xi(p)$ by $\varphi=\varphi_s$. Further, for each $\alpha\in\Delta$ we set $t_\alpha:=\lambda_{\alpha^\vee}(p)$.

For example,  $\mathbf{G}=\GL_n$,  $B$  is  the subgroup of upper triangular matrices,  $N$    consists of the strictly upper triangular matrices ($1$ on the diagonal),
$T$  is the diagonal subgroup,  $N_0=\mathbf{N}(\mathbb Z_p)$, the simple roots are $\alpha_1, \ldots, \alpha_{n-1}$ where  $\alpha_i(\diag(t_1,\ldots, t_n))= t_i t_{i+1}^{-1}$, $u_{ \alpha_i}(\cdot) $ is the strictly upper triangular matrix, with $(i,i+1)$-coefficient $\cdot$ and $0$ everywhere else.

For a finite index subgroup $\mathcal{G}_2$ in a group $\mathcal{G}_1$ we denote by $J(\mathcal{G}_1/\mathcal{G}_2)\subset \mathcal{G}_1$ a (fixed) set of representatives of the left cosets in $\mathcal{G}_1/\mathcal{G}_2$.

\subsection{Description of the results}

In section \ref{commutative} we describe the first properties of \'etale $T_+$-modules over $A\bg N_{\Delta,0}\jg$ (the ``multivariable $(\varphi,\Gamma)$-modules'' in the title). Even though the ring $A\bg N_{\Delta,0}\jg$ is not artinian, the existence of an action of $T_0$ improves its properties: by the nonexistence of $T_0$-invariant ideals in $\kappa\bg N_{\Delta,0}\jg$ it follows that any finitely generated module over $A\bg N_{\Delta,0}\jg$ admitting a semilinear action of $T_0$ has finite length (in the category of modules with semilinear $T_0$-action). This fact allows us to construct a functor $D^\vee_\Delta$ from the category of smooth $A$-representations of the Borel $B$ to projective limits of finitely generated \'etale $T_+$-modules over $A\bg N_{\Delta,0}\jg$ in an analogous way to Breuil's functor \cite{B}. More precisely, we consider the skew polynomial ring $A\bs N_{\Delta,0}\js [F_\alpha\mid \alpha\in\Delta]$ where the variables $F_\alpha$ commute with each other and we have $F_\alpha\lambda=(t_\alpha\lambda t_\alpha^{-1})F_\alpha$ for $\lambda\in A\bs N_{\Delta,0}\js$. For a smooth representation $\pi$ of $B$ over $A$ we denote by $\mathcal{M}_\Delta(\pi^{H_{\Delta,0}})$ the set of finitely generated $A\bs N_{\Delta,0}\js [F_\alpha\mid \alpha\in\Delta]$-submodules of $\pi^{H_\Delta,0}$ that are stable under the action of $T_0$ and admissible as a representation of $N_{\Delta,0}=N_0/H_{\Delta,0}$. Here $F_\alpha$ acts on $\pi^{H_\Delta,0}$ by the Hecke action of $t_\alpha\in T_+$, ie.\ $F_\alpha v:=\Tr_{H_{\Delta,0}/t_\alpha H_{\Delta,0}t_\alpha^{-1}}(t_\alpha v)$ for $v\in\pi^{H_\Delta,0}$. Then the functor $D^\vee_\Delta$ is defined by the projective limit $$D^\vee_\Delta(\pi):=\varprojlim_{M\in \mathcal{M}_\Delta(\pi^{H_{\Delta,0}})}M^\vee [1/X_{\Delta}]$$ where $X_\Delta=\prod_{\alpha\in\Delta}X_\alpha$ is the product of all the variables $X_\alpha=n_\alpha-1$ in the power series ring $A\bs N_{\Delta,0}\js$.

If we define $$\ell\colon N\to N/[N,N]=\prod_{\alpha\in \Delta}N_\alpha\overset{\sum_{\alpha\in\Delta}u_\alpha^{-1}}{\longrightarrow}\mathbb{Q}_p$$
in a generic way and extend this to the Iwasawa algebra $A\bs N_{\Delta,0}\js$ then we find that $\ell(X_{\alpha})=X$ for all $\alpha\in\Delta$ after the identification $A\bs \Zp\js \cong A\bs X\js$. Therefore we may extend $\ell$ to a map $\ell\colon A\bg N_{\Delta,0}\jg\to A\bg X\jg$ of Laurent series rings. Note that the kernel of $\ell$ is not stable under the action of $T_+$, but it is stable under the action of $\varphi$ and $\Gamma$. So we obtain a reduction map $A\bg X\jg \otimes_{A\bg N_{\Delta,0}\jg,\ell}\cdot$ from \'etale $T_+$-modules to usual \'etale $(\varphi,\Gamma)$-modules. We show that this reduction map is faithful and exact which implies 
\begin{thma}
The functor $D^\vee_\Delta$ is right exact. 
\end{thma}
In particular, one has a natural transformation from Breuil's functor $D^\vee_{\xi,\ell}$ to the composite $A\bg X\jg \otimes_{A\bg N_{\Delta,0}\jg,\ell}D^\vee_\Delta$. When restricted to the category $SP_A$, this is an isomorphism.  Moreover, this is also an isomorphism for objects obtained by parabolic induction from a subgroup with Levi component isomorphic to the product of copies of $\GL_2(\mathbb{Q}_p)$ and a split torus.

Section \ref{compatible} is devoted to various compatibility results. The first is the compatibility with products $G\times G'$ of groups with simple roots $\Delta$, resp.\ $\Delta'$. The value of $D^\vee_{\Delta\cup\Delta'}$ on a tensor product $\pi\otimes_\kappa\pi'$ of representations $\pi$ of $G$ (resp.\ $\pi'$ of $G'$) is the completed tensor product $D^\vee_\Delta(\pi)\hat{\otimes}_\kappa D^\vee_{\Delta'}(\pi')$. Note that this is a module over a multivariate Laurent series ring $A\bg N_{\Delta\cup\Delta',0}\jg$ in variables indexed by the union $\Delta\cup\Delta'$. Similarly, we have a compatibility result for parabolic induction: Let $P=L_PN_P$ be a parabolic subgroup containing $B$ and $\pi_P$ a smooth representation of $L_P$ over $A$ viewed as representation of the opposite parabolic $P^-$. Denote by $\Delta_P\subseteq \Delta$ the set of those simple roots whose root subgroups are contained in the Levi component $L_P$. We show
\begin{thma}
Let $\pi_P$ be a smooth locally admissible representation of $L_P$ over $A$ which we view by inflation as a representation of $P^-$. We have an isomorphism 
\begin{equation*}
D^\vee_{\Delta}\left(\Ind_{P^-}^G\pi_P\right)\cong A\bg N_{\Delta,0}\jg\hat{\otimes}_{A\bg N_{\Delta_P,0}\jg}D^\vee_{\Delta_P}(\pi_P)
\end{equation*}
in the category $\mathcal{D}^{pro-et}(T_+,A\bg N_{\Delta,0}\jg)$.
\end{thma}
On one hand, the above result shows that $D^\vee_{\Delta}$ is nonzero and finitely generated on parabolically induced representations from products of copies of $\GL_2(\mathbb{Q}_p)$ and a torus unless one of the representations of $\GL_2(\mathbb{Q}_p)$ is finite dimensional. Moreover, combined with the right exactness we also know this for extensions of representations of this type just like for Breuil's functor \cite{B}. On the other hand, this might lead to another characterization of supercuspidal representations: it would be natural to expect that if $\pi$ is an irreducible supercuspidal representation then $D^\vee_\Delta(\pi)$ cannot be induced from a $T_+$-module in less variables. However, showing this would require a better understanding of supercuspidals beyond $\GL_2$.

Let $SP_A$ be the category of smooth finite length representations of $G$ whose Jordan-H\"older factors are subquotients of principal series. We end section \ref{compatible} by showing
\begin{thma}
The restriction of $D^\vee_\Delta$ to $SP_A$ is exact and produces finitely generated objects.
\end{thma}
The proof of this builds on  showing that the finitely generated $A\bs N_{\Delta,0}\js [F_\alpha\mid \alpha\in\Delta]$-submodules of representations in $SP_A$ are in fact finitely presented. This substitutes the arguments using the coherence \cite{E1} of the one-variable analogue $A\bs X\js [F]$ in the classical $\GL_2$-situation as the ring $A\bs N_{\Delta,0}\js [F_\alpha\mid \alpha\in\Delta]$ is apparently not coherent.

In section \ref{noncommutative} we develop a noncommutative analogue of $D^\vee_\Delta$ as in \cite{EZ} for Breuil's functor. The first step is the construction of the ring $A\bg N_{\Delta,\infty}\jg$ as a projective limit $\varprojlim_k A\bg N_{\Delta,k}\jg$ where the finite layers $A\bg N_{\Delta,k}\jg:=A\bs N_{\Delta,k}\js[\varphi_s^{kn_0}(X_\alpha)^{-1}]$ (where $n_0=n_0(G)\in\mathbb{N}$ is the maximum of the degrees of the algebraic characters $\beta\circ\xi\colon \mathbf{G}_m\to\mathbf{G}_m$ for all positive roots $\beta\in \Phi^+$) are defined as localisations of the Iwasawa algebra $A\bs N_{\Delta,k}\js$. Here the group $N_{\Delta,k}:=N_0/H_{\Delta,k}$ is the extension of $N_{\Delta,0}$ by a finite $p$-group $H_{\Delta,0}/H_{\Delta,k}$ where $H_{\Delta,k}$ is the smallest normal subgroup in $N_0$ containing $s^kH_{\Delta,0}s^{-k}$. Note that unlike in the one variable localization $\Lambda_\ell(N_0)$ we do not have a section of the group homomorphism $N_{\Delta,k}\to N_{\Delta,0}$. However, restricting to the image of the conjugation by $s^{kn_0}$, we do: this allows us to build a functor $\mathbb{M}_{k,0}$ from the category $\mathcal{D}^{et}(T_+,A\bg N_{\Delta,0}\jg)$ of finitely generated \'etale $T_+$-modules over $A\bg N_{\Delta,0}\jg$ to the category $\mathcal{D}^{et}(T_+,A\bg N_{\Delta,k}\jg)$ of finitely generated \'etale $T_+$-modules over $A\bg N_{\Delta,k}\jg$. Putting $\mathbb{M}_{\infty,0}:=\varprojlim_k\mathbb{M}_{k,0}$ and $\mathbb{D}_{0,\infty}$ to be the functor from the category $\mathcal{D}^{et}(T_+,A \bg N_{\Delta,\infty}\jg)$ of finitely generated \'etale $T_+$-modules over $A\bg N_{\Delta,\infty}\jg$ to $\mathcal{D}^{et}(T_+,A \bg N_{\Delta,0}\jg)$ induced by the reduction map $A\bg N_{\Delta,\infty}\jg\to A\bg N_{\Delta,0}\jg$ we obtain
\begin{thma}\label{introequivcat}
The functors $\mathbb{M}_{\infty,0}$ and $\mathbb{D}_{0,\infty}$ are quasi-inverse equivalences of categories between $\mathcal{D}^{et}(T_+,A \bg N_{\Delta,0}\jg)$ and $\mathcal{D}^{et}(T_+,A \bg N_{\Delta,\infty}\jg)$.
\end{thma}

By considering finitely generated $A\bs N_{\Delta,k}\js[F_{\alpha,k}\mid \alpha\in\Delta]$-submodules of $\pi^{H_{\Delta,k}}$ that are stable under the action of $T_0$ and are admissible as representations of $N_{\Delta,k}$ we introduce the functors $D^\vee_{\Delta,k}$ analogous to $D^\vee_\Delta$ for all $k\geq 0$ and we put $D^\vee_{\Delta,\infty}(\pi):=\varprojlim_k D^\vee_{\Delta,k}(\pi)$ for a smooth representation $\pi$ of $B$ over $A$. This corresponds to $D^\vee_\Delta(\pi)$ via the extension of the equivalence of categories in Theorem \ref{introequivcat} to pro-objects on both sides. The universal property of $D^\vee_{\Delta,\infty}$ leads to its alternative description via the Schneider--Vigneras functor $D_{SV}(\pi)$ (and via its \'etale hull $\widetilde{D_{SV}}(\pi)$):

\begin{thma}
We have
\begin{equation*}
D^\vee_{\Delta,\infty}(\pi)\cong \varprojlim_{D}D
\end{equation*}
where $D$ runs through the finitely generated \'etale $T_+$-modules over $A\bg N_{\Delta,\infty}\jg$ arising as a quotient of $A\bg N_{\Delta,\infty}\jg\otimes_{A\bs N_0\js}\widetilde{D_{SV}}(\pi)$ such that the quotient map is continuous in the weak topology of $D$ and the final topology on $A\bg N_{\Delta,\infty}\jg\otimes_{A\bs N_0\js}\widetilde{D_{SV}}(\pi)$ of the map $1\otimes\iota\colon D_{SV}(\pi)\to A\bg N_{\Delta,\infty}\jg\otimes_{A\bs N_0\js}\widetilde{D_{SV}}(\pi)$.
\end{thma}

Finally, we turn to the question of reconstructing the smooth representation $\pi$ of $G$ from $D^\vee_\Delta(\pi)$. This is certainly not possible in general, as for instance finite dimensional representations are in the kernel of $D^\vee_\Delta$ (unless the set $\Delta$ of simple roots is empty). However, using the ideas of \cite{SVZ} we show the following positive results in this direction. For an object $M\in\mathcal{M}_\Delta(\pi^{H_{\Delta,0}})$ we denote by $\widetilde{M_\infty^\vee}$ the \'etale hull of the image $M_\infty^\vee$ of the natural map from $\pi^\vee$ to the \'etale $T_+$-module $\mathbb{M}_{\infty,0}(M^\vee[1/X_\Delta])$.

\begin{thma}
For any smooth $o$-torsion representation $\pi$ of $G$ and any $M\in\mathcal{M}_\Delta(\pi^{H_{\Delta,0}})$ there exists a $G$-equivariant sheaf $\mathfrak{Y}_{\pi,M}$ on $G/B$ with sections $\mathfrak{Y}_{\pi,M}(\mathcal{C}_0)$ on $\mathcal{C}_0$ isomorphic to $\widetilde{M_\infty^\vee}$ as an \'etale $T_+$-module over $A\bs N_0\js$. Moreover, we have a $G$-equivariant continuous map $\beta_{G/B,M}$ from the Pontryagin dual $\pi^\vee$ to the global sections $\mathfrak{Y}_{\pi,M}(G/B)$ that is natural in both $\pi$ and $M$, and is nonzero unless $M^\vee[1/X_{\Delta}]=0$. 
\end{thma}
Here we in fact use the $G$-action on $\pi$ in order to construct the sheaf $\mathfrak{Y}_{\pi,M}$ unlike in \cite{SVZ} where the operators $\mathcal{H}_g=\res(g\mathcal{C}_0\cap\mathcal{C}_0)\circ(g\cdot)$ for the open cell $\mathcal{C}_0:=N_0\overline{B}\subset G/\overline{B}\cong G/B$ are constructed as a limit. Apparently the formulas defining this limit do not converge in the weak topology of the finitely generated $A\bg N_{\Delta,\infty}\jg$-module $M_\infty^\vee[1/X_\Delta]$. Nevertheless, if $\pi$ is irreducible and $D^\vee_\Delta(\pi)\neq 0$ then we can realize $\pi^\vee$ as a subrepresentation of the global sections of a $G$-equivariant sheaf on $G/B$ whose space of sections on $\mathcal{C}_0$ is ``small'' in the sense that it is contained in a finitely generated $A\bg N_{\Delta,\infty}\jg$-module. Let us denote by $SP_A^0$ the full subcategory of $SP_A$ containing those representations whose Jordan-H\"older factors are irreducible principle series. As an application of the methods above we prove
\begin{thma}
The restriction of $D^\vee_\Delta$ to the category $SP_A^0$ is fully faithful.
\end{thma}
In particular, the forgetful functor restricting $\pi$ to $B$ is also fully faithful on $SP_A^0$ as $D^\vee_\Delta$ factors through this.

\subsection*{Acknowledgements}

My debt to the works of Christophe Breuil \cite{B}, Pierre Colmez \cite{Mira} \cite{C}, Peter Schneider, and Marie-France Vign\'eras \cite{SVig} \cite{SVZ} will be obvious to the reader. I would also like to thank M\'arton Erd\'elyi, Jan Kohlhaase, Vytautas Pa\v{s}k\={u}nas, Peter Schneider, and Tam\'as Szamuely for discussions on the topic. I am grateful to the referee for the careful reading of the manuscript and for their various comments.

\section{\'Etale $T_+$-modules over $A \bg N_{\Delta,0}\jg$}\label{commutative}

Since the centre of $G$ is assumed to be connected, there exists a system $(\lambda_{\alpha^\vee})_{\alpha\in\Delta}\in X^\vee(T)^\Delta$ of cocharacters with the property $\beta\circ\lambda_{\alpha^\vee}=1$ for all $\alpha\neq\beta\in\Delta$ and $\alpha\circ\lambda_{\alpha^\vee}=\id_{\mathbf{G}_m}$. As in \cite{B} and \cite{SVig} we put $\xi:=\sum_{\alpha\in\Delta}\lambda_{\alpha^\vee}$. Further, we put $t_\alpha:=\lambda_{\alpha^\vee}(p)\in T$ for each $\alpha\in \Delta$ and denote by $\varphi_\alpha$ the conjugation action of $t_\alpha$ on $A \bs N_{\Delta,0}\js$ and on $A \bg N_{\Delta,0}\jg$. By definition we have $s=\xi(p)=\prod_{\alpha\in\Delta}t_\alpha$. We form the skew-polynomial ring $A\bs N_{\Delta,0}\js[F_\Delta]:=A \bs N_{\Delta,0}\js[F_\alpha\mid \alpha\in\Delta]$ in the variables $F_\alpha$ ($\alpha\in\Delta$) that commute with each other and satisfy $F_\alpha\lambda=\varphi_\alpha(\lambda)F_\alpha$ for any $\lambda\in \kappa\bs N_{\Delta,0}\js$. Note that we may extend the conjugation action of the group $T_0$ on $A \bs N_{\Delta,0}\js$ to the ring $A \bs N_{\Delta,0}\js[F_\Delta]$ by acting trivially on the variables $F_\alpha$ ($\alpha\in\Delta$).  Note that $T_0$ and the elements $t_\alpha$ ($\alpha\in\Delta$) generate $T_+$ under our assumption that the centre of $G$ is connected.

\subsection{$T_0$-invariant ideals in $A\bg N_{\Delta,0}\jg$}

\begin{pro}\label{noideal}
The ring $\kappa\bg N_{\Delta,0}\jg$ does not have any nontrivial $T_0$-invariant ideals.
\end{pro}
\begin{proof}
By our assumption that $G$ has connected centre the group homomorphisms $\Zp^\times\overset{\lambda_{\alpha^\vee}}{\hookrightarrow} T_0$ give rise to a subgroup $T_{\Delta,0}:=\prod_{\alpha\in\Delta}\lambda_{\alpha^\vee}(\Zp^\times)\leq T_0$. Note that this product is direct since $T_{\alpha,0}:=\lambda_{\alpha^\vee}(\Zp^\times)$ is contained in the kernel of $\beta$ for each $\alpha\neq\beta\in\Delta$. So we have an action of $(\Zp^\times)^\Delta$ on $\kappa\bg N_{\Delta,0}\jg=\kappa\bg X_\alpha\mid \alpha\in\Delta\jg$ such that the copy of $\Zp^\times$ indexed by $\alpha$ acts naturally on the corresponding variable $X_\alpha$ by the formula $\gamma\in\Zp^\times\colon X_\alpha\mapsto (1+X_\alpha)^\gamma-1$ and trivially on all the other variables $X_\beta$ for all pairs $\beta\neq\alpha\in\Delta$. We prove the (formally stronger) statement that $\kappa\bg N_{\Delta,0}\jg$ does not have any $T_{\Delta,0}$-invariant ideals by induction on the cardinality of the set $\Delta$. For $|\Delta|=1$ the statement is trivial since $\kappa\bg X\jg$ is a field. Now assume the statement for $|\Delta|<r$ for some $1<r$ choose a $T_{\Delta,0}$-invariant ideal $0\neq I\lhd \kappa\bg N_{\Delta,0}\jg$ with $|\Delta|=r$. Put $J:=I\cap \kappa\bs N_{\Delta,0}\js\lhd \kappa\bs N_{\Delta,0}\js$ and choose any $\alpha\in\Delta$. Any element $\mu\in \kappa\bs N_{\Delta,0}\js$ can uniquely be written as an infinite sum $\mu=\sum_{j=0}^\infty\mu_jX_\alpha^j$ with $\mu_j\in \kappa\bs N_{\Delta\setminus\{\alpha\},0}\js$ ($j\geq 0$) where $N_{\Delta\setminus\{\alpha\},0}=\prod_{\beta\neq\alpha\in\Delta}N_{\beta,0}$. We define
\begin{equation*}
J_{\alpha,i}:=\{\lambda\in \kappa\bs N_{\Delta\setminus\{\alpha\},0}\js\mid \exists\mu=\sum_{j=0}^\infty\mu_jX_\alpha^j\in J,\text{ s.\ t.\ }\mu_i=\lambda\}
\end{equation*}
for each $i\geq 0$. These are ideals in $\kappa\bs N_{\Delta\setminus\{\alpha\},0}\js$.
\begin{lem}
We have $J_{\alpha,0}=J_{\alpha,1}=\dots=J_{\alpha,i}=\cdots$.
\end{lem}
\begin{proof}
The inclusions $J_{\alpha,0}\subseteq J_{\alpha,1}\subseteq\dots\subseteq J_{\alpha,i}\subseteq\cdots$ are clear (we can multiply an element in $J$ by $X_\alpha$). Conversely, assume that $J_{\alpha,0}\subsetneq J_{\alpha,i}$ for some integer $i>0$. Assume $i>0$ is minimal with this property and choose an element $\mu\in J$ with  $\mu_i\in J_{\alpha,i}\setminus J_{\alpha,0}$. Assume there is an index $j>0$ such that $\mu_j\neq 0$ and $\mu_j\in J_{\alpha,0}$, and choose a $\nu_j\in J$ with $\nu_{j,0}=\mu_j$. Then $\mu':=\mu-X_\alpha^j\nu_j$ also lies in $J$ and has the property that $\mu'_i\notin J_{\alpha,0}$. Indeed, if $i<j$ then this is clear. Otherwise by the minimality of $i$, the coefficient of $X_\alpha^{i-j}$ in $\nu_j$ lies in $J_{\alpha,0}$ (and is equal to $\mu_i-\mu_i'$ for $i\geq j$). Since any $\kappa\bs N_{\Delta,0}\js$ is noetherian, any ideal in it is closed. So all the coefficients of $\mu$ with positive exponent that are contained in $J_{\alpha,0}$ can be removed recursively this way: first the smallest $j$. Therefore we find an element $\mu''\in J$ such that $\mu''_i\notin J_{\alpha,0}$ and for all $j>0$ we either have $\mu''_j=0$ or $\mu''_j\notin J_{\alpha,0}$. Let $0<l=p^rl'$ ($p\nmid l'$) be the smallest integer with the property that $\mu''_l\neq 0$ and $l$ is divisible by the least possible power of $p$ among these indices. Since $I$ is $T_0$-invariant and $\lambda_{\alpha^\vee}(1+p^t)$ is in $T_0$ for $t>1$, we have 
\begin{align*}
I\ni\lambda_\alpha(1+p^{t})\mu''-\mu''=\sum_{j=0}^\infty\mu''_j\left(((X_\alpha+1)^{1+p^{t}}-1)^j-X_\alpha^j\right)=\\
=\sum_{j=1}^\infty\mu''_j\left((X_\alpha+X_\alpha^{p^{t}}+X_\alpha^{p^{t}+1})^j-X_\alpha^j\right)\ .
\end{align*}
For $j=p^kj'$ with $p\nmid j'$ the lowest degree term of $(X_\alpha+X_\alpha^{p^{t}}+X_\alpha^{p^{t}+1})^j-X_\alpha^j$ is $j'X_\alpha^{p^k(j'-1)}X_\alpha^{p^{k+t}}$. Now for an $l\neq j>0$ with $\mu''_j\neq 0$ we have either $k=r$ and $j'>l'$ or $k>r$ (by the choice of $l$). Any case we have $p^k(j'-1)+p^{k+t}>p^r(l'-1)+p^{r+t}$ for any such $j$ as soon as we choose $t$ so that $p^{t+1}-p^t>l'-1$ (since $p^{k+t}-p^{r+t}\geq p^r(p^{t+1}-p^t)$). With such a choice of $t$ we deduce that $\frac{\lambda_\alpha(1+p^{t})\mu''-\mu''}{X_\alpha^{p^r(l'-1)+p^{r+t}}}$ lies in $J=I\cap \kappa\bs N_{\Delta,0}\js$ and has constant term $l'\mu''_l$ that does not lie in $J_{\alpha,0}$. This is a contradiction.
\end{proof}
Now we claim that $J_{\alpha,0}\subseteq J\cap \kappa\bs N_{\Delta\setminus\{\alpha\},0}\js$ and hence $J\cap \kappa\bs N_{\Delta\setminus\{\alpha\},0}\js$ is nonzero. For an element $\lambda\in J_{\alpha,0}$ choose an element $\mu\in J$ with $\mu_0=\lambda$. If $\mu_j\neq 0$ for some $j>0$ then in view of the Lemma choose $\nu_j\in J$ with $\nu_{j,0}=\mu_j$ and let $\mu':=\mu-X_\alpha^j\nu_j$. By the same recursive argument as in the Lemma we find an element $\mu''\in J$ with $\mu''_0=\mu_0=\lambda$ and $\mu''_j=0$ for $j>0$ showing the claim. The statement of the proposition follows from the inductional hypothesis: $(J\cap \kappa\bs N_{\Delta\setminus\{\alpha\},0}\js)[X_\beta^{-1}\mid \beta\in\Delta\setminus\{\alpha\}]$ is a nonzero $T_0$-invariant ideal in $\kappa\bg N_{\Delta\setminus\{\alpha\},0}\jg$ therefore contains $1$.
\end{proof}

\subsection{A functor from smooth $B$-representations to \'etale $T_+$-modules over $A\bg N_{\Delta,0}\jg$}

We have the following generalization of Lemma 2.6 in \cite{B}.

\begin{pro}\label{fingendualetale}
Let $M$ be a finitely generated module over $A \bs N_{\Delta,0}\js[F_\Delta]$ with a semilinear action of $T_0$ such that $M$ is admissible and smooth as a module over $A \bs N_{\Delta,0}\js$ (ie.\ the Pontryagin dual $M^\vee$ is finitely generated over $A \bs N_{\Delta,0}\js$). Then the module $M^\vee[1/X_\alpha\mid \alpha\in\Delta]$ has naturally the structure of an \'etale $T_+$-module over $A \bg N_{\Delta,0}\jg$.
\end{pro}
In order to simplify notation we put $X_\Delta:=\prod_{\alpha\in\Delta}X_\alpha$ so that we have $(\cdot)[1/X_\Delta]=(\cdot)[1/X_\alpha\mid \alpha\in\Delta]$.
\begin{proof}
By passing to the modules $\varpi^rM/\varpi^{r+1}M$ ($0\leq r<h$) we may assume without loss of generality that $h=1$. Let $C_\alpha$ be the cokernel of the map $\kappa\bs N_{\Delta,0}\js\otimes_{\varphi_\alpha}M\overset{1\otimes F_\alpha}{\to}M$. Since $M$ is finitely generated over $\kappa\bs N_{\Delta,0}\js[F_\beta\mid \beta\in\Delta]$, $C_\alpha$ is finitely generated over the smaller ring $\kappa\bs N_{\Delta,0}\js[F_\beta\mid \beta\in\Delta\setminus\{\alpha\}]$. Let $m_1,\dots,m_r\in C_\alpha$ be the generators. Since $M$ is smooth as a representation of $N_{\alpha,0}\leq N_{\Delta,0}$, so is $C_\alpha$. Therefore there exists a power $X_\alpha^s$ ($s>0$) of $X_\alpha$ killing each $m_i$ ($1\leq i\leq r$). However, $X_\alpha$ is in the centre of $\kappa\bs N_{\Delta,0}\js[F_\beta\mid \beta\in\Delta\setminus\{\alpha\}]$ as each $F_\beta$ ($\beta\neq \alpha$) commutes with $X_\alpha$. Therefore we have $X_\alpha^sC_\alpha=0$. In particular, we deduce $C_\alpha^\vee[1/X_\alpha]=0$. This shows that the map
\begin{equation}\label{1otimesFalphavee}
 M^\vee[1/X_\Delta]\overset{(1\otimes F_\alpha)^\vee[1/X_\Delta]}{\longrightarrow} (\kappa\bs N_{\Delta,0}\js\otimes_{\varphi_\alpha,\kappa\bs N_{\Delta,0}\js}M)^\vee[1/X_\Delta]\cong \kappa\bs N_{\Delta,0}\js\otimes_{\varphi_\alpha,\kappa\bs N_{\Delta,0}\js}M^\vee[1/X_\Delta]
\end{equation}
is injective. Moreover, the generic rank over $\kappa\bg N_{\Delta,0}\jg$ of the two sides of \eqref{1otimesFalphavee} equals. Therefore the cokernel of \eqref{1otimesFalphavee} is a finitely generated torsion module over $\kappa\bg N_{\Delta,0}\jg$ since $M$ is admissible. Moreover, the global annihilator of this cokernel is $T_0$-invariant as the map \eqref{1otimesFalphavee} is $T_0$-equivariant. Proposition \ref{noideal} shows that in fact  \eqref{1otimesFalphavee} is an isomorphism for each $\alpha\in\Delta$.
\end{proof}

For a smooth representation $\pi$ of $B_+$ over $A$ we can make $\pi^{H_{\Delta,0}}$ a module over $A \bs N_{\Delta,0}\js[F_\Delta]$ by the Hecke action $F_\alpha(m):=\sum_{u\in J(H_{\Delta,0}/t_\alpha H_{\Delta,0}t_\alpha^{-1})}ut_\alpha m$. Let us denote by $\mathcal{M}_{\Delta}(\pi^{H_{\Delta,0}})$ the set of those finitely generated $A \bs N_{\Delta,0}\js[F_\Delta]$-submodules $M$ of $\pi^{H_{\Delta,0}}$ that are stable under the action of $T_0$ and are admissible as a module over $A \bs N_{\Delta,0}\js$. We define
\begin{equation*}
D_{\Delta}^\vee(\pi):=\varprojlim_{M\in\mathcal{M}_\Delta(\pi^{H_{\Delta,0}})}M^\vee[1/X_\Delta]\ .
\end{equation*}
Using Prop.\ \ref{fingendualetale} this is a projective limit of finitely generated \'etale $T_+$-module over $A \bg N_{\Delta,0}\jg$ attached functorially to $\pi$: If $f\colon \pi\to \pi'$ is a morphism of smooth $A$-representations of $B$ and $M$ lies in $\mathcal{M}_\Delta(\pi^{H_{\Delta,0}})$ then $f(M)$ lies in $\mathcal{M}_\Delta({\pi'}^{H_{\Delta,0}})$. Indeed, $f(M)$ is finitely generated over $A\bs N_{\Delta,0}\js [F_\Delta]$, stable under the action of $T_0$, and admissible as a representation of $N_\Delta$.

\subsection{The category $\mathcal{D}^{et}(T_+,A \bg N_{\Delta,0}\jg)$}

For a submonoid $T_\ast\leq T_+$ we denote by $\mathcal{D}^{et}(T_\ast,A \bg N_{\Delta,0}\jg)$ the category of finitely generated \'etale $T_\ast$-modules over $A \bg N_{\Delta,0}\jg$. We regard these objects as left modules over $A\bg N_{\Delta,0}\jg$. Further, we denote by $\mathcal{D}^{pro-et}(T_\ast,A \bg N_{\Delta,0}\jg)$ the category of projective limits of objects in $\mathcal{D}^{et}(T_\ast,A \bg N_{\Delta,0}\jg)$.

\begin{rem}
It is shown in Cor.\ 3.16 in \cite{Z16} that any object $D$ in $\mathcal{D}^{et}(T_+,\kappa\bg N_{\Delta,0}\jg)$ is free as a module over $\kappa\bg N_{\Delta,0}\jg$. However, we do not use this fact in the present paper.
\end{rem}

Since $\kappa\bg N_{\Delta,0}\jg$ is a localization of the local noetherian ring $\kappa\bs N_{\Delta,0}\js$ it is also noetherian and has finite global dimension ($\leq |\Delta|$). Moreover, any module over $\kappa\bg N_{\Delta,0}\jg$ admits a free resolution of finite length since any projective module over $\kappa\bs N_{\Delta,0}\js$ is free. In particular, we may define the (generic) rank of a module $M$ over $\kappa\bg N_{\Delta,0}\jg$ as the alternating sum $\rk(M):=\rk_{\kappa\bg N_{\Delta,0}\jg}(M):=\sum_{i=0}^{|\Delta|}(-1)^i r_i$ for a free resolution
\begin{equation*}
0\to F_{|\Delta|}\to\dots\to F_0\to M\to 0
\end{equation*}
with $F_i\cong \kappa\bg N_{\Delta,0}\jg^{r_i}$ ($i=0,\dots,|\Delta|$). This is equal to the dimension of $Q\bg N_{\Delta,0}\jg\otimes_{\kappa\bg N_{\Delta,0}\jg}M$ where $Q\bg N_{\Delta,0}\jg$ denotes the field of fractions of $\kappa\bg N_{\Delta,0}\jg$.

For a (left) module $D$ over $A\bg N_{\Delta,0}\jg$ we define the generic length of $D$ as $\genlength(D):=\mathrm{length}_{gen,A\bg N_{\Delta,0}\jg}(D):=\sum_{i=0}^{h-1}\rk_{\kappa\bg N_{\Delta,0}\jg}(\varpi^i D/\varpi^{i+1}D)$. Let $Q(A\bg N_{\Delta,0}\jg)$ be the localization of $A\bg N_{\Delta,0}\jg$ at the prime ideal generated by $\varpi$. This is an artinian local ring with maximal ideal generated by $\varpi$ and residue field isomorphic to $Q\bg N_{\Delta,0}\jg$. We have an isomorphism $$Q\bg N_{\Delta,0}\jg\otimes_{\kappa\bg N_{\Delta,0}\jg}(\varpi^i D/\varpi^{i+1}D)\cong \varpi^iQ(A\bg N_{\Delta,0}\jg)\otimes_{A\bg N_{\Delta,0}\jg}D/\varpi^{i+1}Q(A\bg N_{\Delta,0}\jg)\otimes_{A\bg N_{\Delta,0}\jg}D\ .$$ Therefore the generic length of an $A\bg N_{\Delta,0}\jg$-module $D$ equals the length of $Q(A\bg N_{\Delta,0}\jg)\otimes_{A\bg N_{\Delta,0}\jg}D$. In particular, the generic length is additive on short exact sequences.

\begin{lem}\label{equalgenlength}
For a finitely generated module $D$ in $A \bg N_{\Delta,0}\jg$ and $t\in T_+$ we have $\genlength(D)=\genlength(A\bg N_{\Delta,0}\jg\otimes_{\varphi_t}D)$.
\end{lem}
\begin{proof}
Note that we have $\varpi^iA\bg N_{\Delta,0}\jg\otimes_{\varphi_t}D/\varpi^{i+1}A\bg N_{\Delta,0}\jg\otimes_{\varphi_t}D\cong \kappa\bg N_{\Delta,0}\jg\otimes_{\varphi_t}(\varpi^iD/\varpi^{i+1})$. So we may assume $A=\kappa$ and $\genlength=\rk$. The statement is clear since $\varphi_t\colon \kappa\bg N_{\Delta,0}\jg\to \kappa\bg N_{\Delta,0}\jg$ is flat, so it takes free resolutions to free resolutions.
\end{proof}

\begin{pro}\label{abeliancat}
If $D_2$ is an object in $\mathcal{D}^{et}(T_+,A \bg N_{\Delta,0}\jg)$ and $D_1$ is a $T_+$-stable $A \bg N_{\Delta,0}\jg$-submodule then both $D_1$ and $D_2/D_1$ are \'etale for the inherited action of $T_+$, ie.\ objects in $\mathcal{D}^{et}(T_+,A \bg N_{\Delta,0}\jg)$. In particular, $\mathcal{D}^{et}(T_+,A \bg N_{\Delta,0}\jg)$ is an abelian category.
\end{pro}
\begin{proof}
It is clear that  $\mathcal{D}^{et}(T_+,A \bg N_{\Delta,0}\jg)$ is an additive category. So it suffices to show the first statement. Put $D_3:=D_2/D_1$ and for a fixed $t\in T_+$ consider the commutative diagram
\begin{equation*}
\xymatrix{
0\ar[r] & A \bg N_{\Delta,0}\jg\otimes_{\varphi_t}D_1\ar[r]\ar[d]^{f_1} & A \bg N_{\Delta,0}\jg\otimes_{\varphi_t}D_2 \ar[r]\ar[d]^{f_2} &  A \bg N_{\Delta,0}\jg\otimes_{\varphi_t}D_3\ar[r]\ar[d]^{f_3} &0\\
0\ar[r] & D_1\ar[r] & D_2 \ar[r] &  D_3\ar[r] &0
}
\end{equation*}
with exact rows. Since $f_2$ is an isomorphism, we deduce that $f_1$ is injective and $f_3$ is surjective. Therefore $\Coker(f_1)$ and $\Ker(f_3)$ have $0$ generic length by Lemma \ref{equalgenlength}. In particular, $\Coker(f_1)/\varpi \Coker(f_1)$ and $\Ker(f_3)/\varpi \Ker(f_3)$ are finitely generated torsion modules over $\kappa\bg N_{\Delta,0}\jg$ both admitting a semilinear action of $T_0$. Hence their global annihilator is a nonzero $T_0$-invariant ideal in $\kappa\bg N_{\Delta,0}\jg$ that contains $1$ by Prop.\ \ref{noideal}. We obtain that $\Coker(f_1)/\varpi \Coker(f_1)=\Ker(f_3)/\varpi \Ker(f_3)=0$ whence we also have $\Coker(f_1)=\Ker(f_3)=0$ showing that both $f_1$ and $f_3$ are isomorphisms.
\end{proof}

\begin{rem}
The above proof actually shows that $\mathcal{D}^{et}(T_\ast,A \bg N_{\Delta,0}\jg)$ is an abelian category for any submonoid $T_\ast\leq T_+$ containing an open subgroup of $T_0$.
\end{rem}

\subsection{A functor to usual $(\varphi,\Gamma)$-modules}\label{reducetousualsec}

Assume in this section that $\Delta\neq\emptyset$ or, equivalently, that $G\neq T$.

There exists a $\xi(\Zp\setminus\{0\})$-equivariant group homomorphism 
\begin{equation*}
\ell=\ell_{gen}:=\sum_{\alpha\in\Delta}u_\alpha^{-1}\colon N_{\Delta,0} \twoheadrightarrow \mathbb{Z}_p\ .
\end{equation*}
We denote by $H_{\ell,\Delta,0}\leq N_{\Delta,0}$ the kernel of the group homomorphism $\ell$ and by $H_{\ell,0}$ its preimage in $N_0$ under the quotient map $N_0\twoheadrightarrow N_{\Delta,0}$. The restriction of $\ell$ to $N_{\alpha,0}$ is an isomorphism for all $\alpha\in\Delta$, so $\ell$ is generic. Moreover, the induced ring homomorphism $\ell\colon A\bs N_{\Delta,0}\js\twoheadrightarrow A\bs \Zp\js\cong A\bs X\js$ extends to a surjective ring homomorphism $\ell\colon A\bg N_{\Delta,0}\jg\twoheadrightarrow A\bg X\jg$. Its kernel is generated by the elements $X_\alpha-X_\beta$ for $\alpha,\beta\in\Delta$. So we obtain a functor $A\bg X\jg\otimes_{\ell,A\bg N_{\Delta,0}\jg}\cdot$ from the category $\mathcal{D}^{et}(T_+,A\bg N_{\Delta,0}\jg)$ of \'etale $T_+$-modules over $A\bg N_{\Delta,0}\jg$ to the category $\mathcal{D}^{et}(\varphi,\Gamma,A\bg X\jg)$ of \'etale $(\varphi,\Gamma)$-modules over $A\bg X\jg$. 
\begin{pro}\label{reduceellexact}
The functor $A\bg X\jg\otimes_{\ell,A\bg N_{\Delta,0}\jg}\cdot$ from $\mathcal{D}^{et}(T_+,A\bg N_{\Delta,0}\jg)$ to $\mathcal{D}^{et}(\varphi,\Gamma,A\bg X\jg)$ is faithful and exact. In particular, if $D$ is a nonzero \'etale $T_+$-module over $A\bg N_{\Delta,0}\jg$ then $D_\ell:= A\bg X\jg\otimes_{\ell,A\bg N_{\Delta,0}\jg}D$ is nonzero either.
\end{pro}
\begin{proof}
At first we prove the exactness. The functor $A\bg X\jg\otimes_{\ell,A\bg N_{\Delta,0}\jg}\cdot$ is clearly right exact. We show by induction on $|\Delta|$ that it takes injective maps to injective maps. If $\Delta|=1$ then there is nothing to prove. So let $|\Delta|>1$ and choose $\alpha\neq\beta\in\Delta$. Denote by $T_{+,\alpha=\beta}$ the submonoid in $T_+$ on which the two characters $\alpha$ and $\beta$ agree and put $T_{0,\alpha=\beta}:=T_0\cap T_{+,\alpha=\beta}$. We have a functor $A\bg N_{\Delta,0}\jg/(X_\alpha-X_\beta)\otimes_{A\bg N_{\Delta,0}\jg}\cdot$ from $\mathcal{D}^{et}(T_+,A\bg N_{\Delta,0}\jg)$ to the category $\mathcal{D}^{et}(T_{+,\alpha=\beta},A\bg N_{\Delta,0}\jg/(X_\alpha-X_\beta))$ of \'etale $T_+$-modules over $A\bg N_{\Delta,0}\jg/(X_\alpha-X_\beta)$. Let $$0\to D_1\to D_2\to D_3\to 0$$ be an exact sequence in $\mathcal{D}^{et}(T_+,A\bg N_{\Delta,0}\jg)$. By induction it suffices to show that the sequence
$$0\to D_1/(X_\alpha-X_\beta)\to D_2/(X_\alpha-X_\beta)$$
is exact. Assume that $d_1+(X_\alpha-X_\beta)D_1$ lies in the kernel of the above map for some $d_1\in D_1$. Then we have $d_1\in D_1\cap (X_\alpha-X_\beta)D_2$ (viewing $D_1$ as a subobject of $D_2$). Therefore there exists a $d_2\in D_2$ such that $d_1=(X_\alpha-X_\beta)d_2$. Then the image $d_3$ of $d_2$ in $D_3$ is killed by $(X_\alpha-X_\beta)$. Assume that $d_3\neq 0$. Then there is an integer $0\leq r<h$ such that $d_3\in \varpi^rD_3\setminus \varpi^{r+1}D_3$. However, $\varpi^r D_3/\varpi^{r+1}D_3$ is torsion-free as a module over $\kappa\bg N_{\Delta,0}\jg$ since the global annihilator of its torsion part would be a $T_0$-invariant ideal that does not exist by Prop.\ \ref{noideal}. This is a contradiction as the class of $d_3$ in $\varpi^rD_3/\varpi^{r+1}D_3$ is killed by $X_\alpha-X_\beta$. So we conclude that $d_2$ lies in $D_1$ whence $d_1+(X_\alpha-X_\beta)D_1=0$.

For the faithfulness let $f\colon D_1\to D_2$ be a nonzero map in $\mathcal{D}^{et}(T_+,A\bg N_{\Delta,0}\jg)$. By passing to a suitable subquotient $\varpi^rD_2/\varpi^{r+1}D_2$ we may assume without loss of generality that $A=k$. Since $\mathcal{D}^{et}(T_+,\kappa\bg N_{\Delta,0}\jg)$ is an abelian category, $f(D_1)$ is a subobject in $D_2$. By the above left exactness it suffices to show that $\kappa\bg X\jg\otimes_{\ell,\kappa\bg N_{\Delta,0}\jg}f(D_1)$ is nonzero. However, this is clear since $f(D_1)$ is torsionfree as a module over $A\bg N_{\Delta,0}\jg$ (again by Prop.\ \ref{noideal}), therefore its localization at the maximal ideal $\Ker(\ell)$ of $\kappa\bg N_{\Delta,0}\jg$ is nonzero showing that $\kappa\bg X\jg\otimes_{\ell,\kappa\bg N_{\Delta,0}\jg}f(D_1)\neq 0$.
\end{proof}

\begin{rems}
\begin{enumerate}
\item It is shown in \cite{Z16} that the functor $A\bg X\jg\otimes_{\ell,A\bg N_{\Delta,0}\jg}\cdot$ from $\mathcal{D}^{et}(T_+,A\bg N_{\Delta,0}\jg)$ to $\mathcal{D}^{et}(\varphi,\Gamma,A\bg X\jg)$ corresponds to restriction to the diagonal embedding of $\Gal(\overline{\mathbb{Q}_p}/\mathbb{Q}_p)$ into its $|\Delta|$th direct power on the Galois side. In particular, it is not full if $|\Delta|>1$.
\item We can extend the functor $A\bg X\jg\otimes_{\ell,A\bg N_{\Delta,0}\jg}\cdot$ to $\mathcal{D}^{pro-et}(T_+,A\bg N_{\Delta,0}\jg)$ the following way. For $\hat{D}=\varprojlim_{i\in I}D_i$ with $D_i$ being an object in $\mathcal{D}^{et}(T_+,A\bg N_{\Delta,0}\jg)$, we define
\begin{equation*}
A\bg X\jg\hat{\otimes}_{\ell,A\bg N_{\Delta,0}\jg}\hat{D}:=\varprojlim_{i\in I}A\bg X\jg\otimes_{\ell,A\bg N_{\Delta,0}\jg}D_i\ .
\end{equation*}
Then $A\bg X\jg\hat{\otimes}_{\ell,A\bg N_{\Delta,0}\jg}\cdot$ is an exact and faithful functor from $\mathcal{D}^{pro-et}(T_+,A\bg N_{\Delta,0}\jg)$ to $\mathcal{D}^{pro-et}(\varphi,\Gamma,A\bg X\jg)$ since $A\bg X\jg$ is an artinian ring therefore the category of pseudocompact modules over $A\bg X\jg$ have exact projective limits. For any smooth representation $\pi$ of $B$ over $A$ we have a surjective map $D^\vee_\xi(\pi)\to A\bg X\jg\hat{\otimes}_{\ell,A\bg N_{\Delta,0}\jg}D^\vee_\Delta(\pi)$. Whether or not this is always an isomorphism (or even in the case $\pi$ satisfies certain admissibility conditions) is an open question. We shall see later on that for $\pi$ in the category $SP_A$ this is true.
\end{enumerate}
\end{rems}

Let $D$ be an object in $\mathcal{D}^{et}(T_\ast,A \bg N_{\Delta,0}\jg)$ for some submonoid $T_\ast\leq T_+$. For any set of representatives $J(N_{\Delta,0}/tN_{\Delta,0}t^{-1})\subset N_{\Delta,0}$ and elements $t\in T_+$, $d\in D$ we can write $d$ uniquely as $d=\sum_{u\in J(N_{\Delta,0}/tN_{\Delta,0}t^{-1})}u\varphi_t(d_{t,u})$. As usual \cite{SVZ} \cite{EZ} we put $\psi_t(d):=d_{t,1}$ for a subset $J(N_{\Delta,0}/tN_{\Delta,0})$ containing $1$. We have $\psi_t(u^{-1}d)=d_{t,u}$ for any $u\in J(N_{\Delta,0}/tN_{\Delta,0}t^{-1})$ and $t\in T_\ast$. Moreover, we have $\psi_t(\lambda\varphi_t(d))=\psi_t(\lambda)d$ and $\psi_t(\varphi_t(\lambda)d)=\lambda\psi_t(d)$ for all $d\in D$ and $\lambda\in A\bg N_{\Delta,0}\jg$ (here we define $\psi_t(\lambda)$ similarly as above---note that $A\bg N_{\Delta,0}\jg$ with the conjugation action of $T_\ast$ is also an object in $\mathcal{D}^{et}(T_\ast,A \bg N_{\Delta,0}\jg)$). Further, for $t_1,t_2\in T_\ast$ we have $\psi_{t_1t_2}=\psi_{t_2}\circ\psi_{t_1}$. We call this the $\psi$-action of $T_+$ on $D$. 

\begin{lem}\label{reducepsi}
Let $D$ be an object in $\mathcal{D}^{et}(T_+,A \bg N_{\Delta,0}\jg)$ and $d$ be in $D$. Choose the set $1\in J(N_{\Delta,0}/sN_{\Delta,0}s^{-1})$ of representatives so that whenever $\ell(u)$ and $\ell(v)$ are in the same coset in $\Zp/p\Zp$ for some $u,v\in J(N_{\Delta,0}/sN_{\Delta,0}s^{-1})$ then actually we have $\ell(u)=\ell(v)$. (This can be done as we have a splitting so that $N_{\Delta,0}$ decomposes as $N_{\Delta,0}\cong \Ker(\ell)\times \Zp$.) Then for the element $1\otimes d\in A\bg X\jg\otimes_{\ell,A\bg N_{\Delta,0}\jg}D=D_\ell$ we have $\psi_s(1\otimes d)=1\otimes (\sum_{u\in J(\Ker(\ell)/s\Ker(\ell)s^{-1})}\psi_s(u^{-1}d))$.
\end{lem}
\begin{proof}
We may decompose $d$ as $d=\sum_{u\in J(N_{\Delta,0}/sN_{\Delta,0}s^{-1})}u\varphi_s\circ\psi_s(u^{-1}d)$. The image of this equation in $D_\ell$ reads as
\begin{equation*}
1\otimes d=1\otimes\sum_{u\in J(N_{\Delta,0}/sN_{\Delta,0}s^{-1})}u\varphi_s\circ\psi_s(u^{-1}d)=\sum_{u\in J(N_{\Delta,0}/sN_{\Delta,0}s^{-1})}\ell(u)\varphi_s(1\otimes \psi_s(u^{-1}d))\ .
\end{equation*}
Now since we have $\Ker(\ell)\cap sN_{\Delta,0}s^{-1}=s\Ker(\ell)s^{-1}$ (as conjugation by $s$ on $\mathrm{Im}(\ell)=\mathbb{Z}_p$ is injective), we obtain that $\ell(u)$ runs through a set of representatives of $\Zp/p\Zp$ and each value is taken $|\Ker(\ell)/s\Ker(\ell)s^{-1}|$-times by our assumption on $J(N_{\Delta,0}/sN_{\Delta,0}s^{-1})$. So we compute
 \begin{equation*}
1\otimes d=\sum_{v\in J(\Zp/p\Zp)}v\varphi_s(\sum_{u\in J(N_{\Delta,0}/sN_{\Delta,0}s^{-1}), \ell(u)=v}1\otimes \psi_s(u^{-1}d))\ .
\end{equation*}
Since $1$ lies in $J(N_{\Delta,0}/sN_{\Delta,0}s^{-1})$, $1=\ell(1)$ lies in $J(\Zp/p\Zp)$. So we deduce the statement by the uniqueness of decomposition of $1\otimes d$ as a sum above.
\end{proof}

\subsection{Right exactness of $D^\vee_\Delta$}

Assume $T_\ast=T_+$ and let $D_0$ be a finitely generated $A\bs N_{\Delta,0}\js$-submodule in an object $D$ in $\mathcal{D}^{et}(T_+,A \bg N_{\Delta,0}\jg)$ that is stable under the $\psi$-action of $T_+$. Since $D_0$ is compact in its canonical topology, its Pontryagin dual $D_0^\vee=\Hom^{cont}_A(D_0,A)$ is discrete. Moreover, the group $N_{\Delta,0}$ acts on $D_0^\vee$ by the formula $um(d):=m(u^{-1}d)$ ($u\in N_{\Delta,0}$, $m\in D_0^\vee$, $d\in D_0$). This extends to an action of the Iwasawa algebra $A\bs N_{\Delta,0}\js$ by continuity making $D_0^\vee$ a discrete left $A\bs N_{\Delta,0}\js$-module. For an element $m\in D_0^\vee$ and a simple root $\alpha\in\Delta$ we define $F_\alpha(m)\in D_0^\vee$ by the formula $F_\alpha(m)(d):=m(\psi_{t_\alpha}(d))$. Denote by $\#$ the continuous (anti-)involution on the (commutative!) ring $A\bs N_{\Delta,0}\js$ that sends group elements $u\in N_{\Delta,0}\subset A\bs N_{\Delta,0}\js$ to $u^\#:=u^{-1}$. We have $\varphi_t(\lambda^\#)=\varphi_t(\lambda)^\#$ for any $\lambda\in A\bs N_{\Delta,0}\js$ and $t\in T_+$ since this is true if $\lambda=u$ in $N_{\Delta,0}$. For $\lambda\in A\bs N_{\Delta,0}\js$ and $m\in D_0^\vee$ we have $F_\alpha(\lambda m)=\varphi_\alpha(\lambda)F_\alpha(m)$ in $D_0^\vee$ since
\begin{align*}
F_\alpha(\lambda m)(d)=(\lambda m)(\psi_{t_\alpha}(d))=m(\lambda^\#\psi_{t_\alpha}(d))=\\
=m(\psi_{t_\alpha}(\varphi_{t_\alpha}(\lambda^\#)d))=F_\alpha(m)(\varphi_{t_\alpha}(\lambda)^\#d)=(\varphi_{t_\alpha}(\lambda)F_\alpha(m))(d)\ .
\end{align*}
Moreover, we have $F_\alpha\circ F_\beta=F_\beta\circ F_\alpha$ on $D_0^\vee$ as $\psi_{t_\alpha}$ and $\psi_{t_\beta}$ commute on $D_0$ ($T_+$ is commutative). So by this action of the operators $F_\alpha$ ($\alpha\in\Delta$) we equipped $D_0^\vee$ with the structure of a module over the ring $A\bs N_{\Delta,0}\js[F_\Delta]$.

On the other hand, for each $t\in T_+$ the $\psi$-action induces a $A\bs N_{\Delta,0}\js$-module homomorphism
\begin{eqnarray}
D_0&\to & A\bs N_{\Delta,0}\js\otimes_{\varphi_t, A\bs N_{\Delta,0}\js}D_0\notag\\
d&\mapsto&\sum_{u\in J(N_{\Delta,0}/tN_{\Delta,0}t^{-1})}u\otimes\psi_t(u^{-1}d)\ .\label{psitoncompact}
\end{eqnarray}
The above map \eqref{psitoncompact} is injective for all $t\in T_+$ (ie.\ the $\psi$-action is nondegenerate in the sense of section 4 in \cite{EZ}) since $D_0$ is a $\psi$-invariant submodule of an \'etale $T_+$-module. Therefore in this case $D_0[1/X_\Delta]$ is an object in $\mathcal{D}^{et}(T_+,A \bg N_{\Delta,0}\jg)$: the map
\begin{equation*}
D_0[1/X_\Delta]\to  A\bg N_{\Delta,0}\jg\otimes_{\varphi_t. A\bg N_{\Delta,0}\jg}D_0[1/X_\Delta]
\end{equation*}
induced by \eqref{psitoncompact} is a $T_0$-equivariant injective map between modules of equal generic length, so it is an isomorphism by Prop.\ \ref{noideal}.

\begin{pro}\label{dualcompactfingen}
Let $D$ be an object in $\mathcal{D}^{et}(T_+,A \bg N_{\Delta,0}\jg)$ and $D_0\subseteq D$ be a finitely generated $A\bs N_{\Delta,0}\js$-submodule that is stable under the $\psi$-action of $T_+$. Then $D_0^\vee$ is a finitely generated module over $A \bs N_{\Delta,0}\js[F_\Delta]$ that is admissible as a (smooth) representation of $N_{\Delta,0}$ and admits a semilinear action of $T_0$.
\end{pro}
\begin{proof}
The case $G=T$ is trivial, so we assume $\Delta\neq \emptyset$.

The admissibility of $D_0^\vee$ follows from the assumption that $(D_0^\vee)^\vee\cong D_0$ is a finitely generated module over $A\bs N_{\Delta,0}\js$. Similarly, $T_0$ acts on $D_0$ therefore also on $D_0^\vee$. It remains to show that $D_0^\vee$ is finitely generated as a module over $A\bs N_{\Delta,0}\js[F_\Delta]$. By the above discussion we see that $D_0[1/X_\Delta]$ is a subobject in $D$ (in the category $\mathcal{D}^{et}(T_+,A \bg N_{\Delta,0}\jg)$) containing $D_0$  therefore we may assume without loss of generality that $D=D_0[1/X_\Delta]$ (Prop.\ \ref{abeliancat}). Denote by $D_{\ell,0}$ the image of $D_0$ under the surjective map $D\to D_\ell$ sending $d\in D$ to $1\otimes d$. By Lemma \ref{reducepsi} $D_{\ell,0}\subset D_\ell$ is a $\psi$- and $\Gamma$-invariant treillis in the \'etale $(\varphi,\Gamma)$-module $D_\ell$. By Lemma 3.5 in \cite{EZ} $D_{\ell,0}^\vee$ is a finitely generated module over the ring $A\bs X\js [F]$. Let $m_1,\dots,m_r\in D_{\ell,0}^\vee$ be a finite set of generators. Note that $D_{\ell,0}^\vee$ is contained in the $H_{\ell,\Delta,0}$-invariant part $(D_0^\vee)^{H_{\ell,\Delta,0}}$ since $H_{\ell,\Delta,0}$ acts trivially on it. (Recall from section \ref{reducetousualsec} that $H_{\ell,\Delta,0}\leq N_{\Delta,0}$ is the kernel of the group homomorphism $\ell\colon N_{\Delta,0}\to \Zp$.) Moreover, the action of the ring $A\bs X\js[F]$ comes from the ring homomorphism 
\begin{eqnarray*}
A\bs N_{\Delta,0}\js[\Tr_{H_{\ell,\Delta,0}/sH_{\ell,\Delta,0}s^{-1}}\circ\prod_{\alpha\in\Delta}F_\alpha]&\twoheadrightarrow & A\bs X\js[F]\\
A\bs N_{\Delta,0}\js\ni \lambda&\mapsto &\ell(\lambda)\\
\Tr_{H_{\ell,\Delta,0}/sH_{\ell,\Delta,0}s^{-1}}\circ\prod_{\alpha\in\Delta}F_\alpha&\mapsto& F
\end{eqnarray*}
where we consider $$\Tr_{H_{\ell,\Delta,0}/sH_{\ell,\Delta,0}s^{-1}}\circ\prod_{\alpha\in\Delta}F_\alpha=\sum_{u\in J(H_{\ell,\Delta,0}/sH_{\ell,\Delta,0}s^{-1})}u\prod_{\alpha\in\Delta}F_{\alpha}$$ as an element in $A\bs N_{\Delta,0}\js [F_\Delta]$. Denote by $M$ the $A\bs N_{\Delta,0}\js [F_\Delta]$-submodule of $D_0^\vee$ generated by the elements $t_0m_i$ with $t_0\in T_0$ and $1\leq i\leq r$. By the discussion above we deduce that $D_{\ell,0}^\vee$ is contained in $M^{H_{\ell,\Delta,0}}$. On the other hand, the orbit of each $m_i$ ($1\leq i\leq r$) under $T_0$ is finite. Indeed, if $m_i$ lies in $(D_0^\vee)^{N_{\Delta,0}^{p^n}}$ for some integer $n$ then we also have $t_0m_i\in (D_0^\vee)^{N_{\Delta,0}^{p^n}}$ since the subgroup $N_{\Delta,0}^{p^n}$ is normalized by $T_0$. However, $(D_0^\vee)^{N_{\Delta,0}^{p^n}}$ is finite by the admissibility. In particular, $M$ is finitely generated over $A\bs N_{\Delta,0}\js [F_\Delta]$. Therefore by Proposition \ref{fingendualetale} $M^\vee[1/X_\Delta]$ has the structure of an \'etale $T_+$-module over $A\bg N_{\Delta,0}\jg$. The surjective map $D_0\twoheadrightarrow M^\vee$ induces a surjective morphism $f\colon D\to M^\vee[1/X_\Delta]$. The map $D\to D_\ell$ factors through $f$ since $D_{\ell,0}^\vee$ is contained in $M$. From Prop.\ \ref{reduceellexact} we deduce that $f$ is an isomorphism as $D_\ell\overset{\sim}{\to}(M^\vee[1/X_\Delta])_\ell$ by construction. Since $D_0$ is a subspace in $D$ we obtain that $D_0\cong M^\vee$ whence $D_0^\vee=M$ is finitely generated over $A \bs N_{\Delta,0}\js[F_\Delta]$.
\end{proof}

\begin{pro}\label{intersectionfingen}
Let $\pi_1$ be a smooth subrepresentation of the smooth $A$-representation $\pi_2$ of $B$ and $M\in \mathcal{M}_\Delta(\pi_2^{H_{\Delta,0}})$. Then there exists an $M_1\subseteq M\cap \pi_1$ such that $M_1$ lies in $\mathcal{M}_\Delta(\pi_1^{H_{\Delta,0}})$ and $M_1^\vee[1/X]=(M\cap \pi_1)^\vee[1/X]$.
\end{pro}
\begin{proof}
The case $G=T$ is trivial, so we assume $\Delta\neq \emptyset$. By the exactness of Pontryagin duality and localization we have an exact sequence
\begin{equation*}
0\to (M/(M\cap \pi_1))^\vee[1/X_\Delta]\to M^\vee[1/X_\Delta]\to (M\cap \pi_1)^\vee[1/X_\Delta]\to 0\ .
\end{equation*}
Note that $M/(M\cap \pi_1)$ is finitely generated over $A \bs N_{\Delta,0}\js[F_\Delta]$ and admissible as a representation of $N_{\Delta,0}$ (being a quotient of the finitely generated and admissible $M$),  and it admits a semilinear action of $T_0$. Therefore by Prop.\ \ref{fingendualetale} and \ref{abeliancat} all three modules in the above short exact sequence are \'etale $T_+$-modules. So if $M_1\leq M\cap \pi_1$ is the submodule dual to the image of $(M\cap \pi_1)^\vee$ in $(M\cap \pi_1)^\vee[1/X_\Delta]$ then $M_1$ is an object in $\mathcal{M}_\Delta(\pi_1^{H_{\Delta,0}})$ by Prop.\ \ref{dualcompactfingen}.
\end{proof}
\begin{rem}
If we assume further that the map $M^\vee\to M^\vee[1/X]$ is injective (ie.\ $M^\vee$ is torsion-free) then we may choose $M_1=M\cap \pi_1$.
\end{rem}

Our main result in this section is the following

\begin{thm}\label{rightexact}
The functor $D^\vee_\Delta$ is right exact.
\end{thm}
\begin{proof}
The case $G=T$ is trivial, so we assume $\Delta\neq \emptyset$. The proof is now similar to that of Prop.\ 2.7(ii) in \cite{B}, however, we needed Prop.\ \ref{reduceellexact} in the preparation. Let $0\to\pi_1\to\pi_2\to\pi_3$ be an exact sequence of smooth $A$-representations of $B$. Then we have an exact sequence $0\to \pi_1^{H_{\Delta,0}}\overset{f_1}{\to}\pi_2^{H_{\Delta,0}}\overset{f_2}{\to} \pi_3^{H_{\Delta,0}}$. Now if $M_i$ is in $\mathcal{M}_\Delta(\pi_i^{H_{\Delta,0}})$ then $f_i(M_i)$ is in $\mathcal{M}_\Delta(\pi_{i+1}^{H_{\Delta,0}})$ ($i=1,2$) since the image of a finitely generated module over a ring is again finitely generated and the admissible representations of $N_{\Delta,0}$ form an abelian category. Now let $M_2$ be in $\mathcal{M}_\Delta(\pi_2^{H_{\Delta,0}})$ and put $M_3:=f_2(M_2)$. Then we have $M_3\in \mathcal{M}_\Delta(\pi_3^{H_{\Delta,0}})$ and by Prop.\ \ref{intersectionfingen} there exists an $M_1$ in $\mathcal{M}_\Delta(\pi_1^{H_{\Delta,0}})$ such that we have an exact sequence
\begin{equation*}
0\to M_3^\vee[1/X]\to M_2^\vee[1/X]\to M_1^\vee[1/X]\to 0
\end{equation*}
of objects in $\mathcal{D}^{et}(T_+,A \bg N_{\Delta,0}\jg)$. Note that for $M_2\subseteq M_2'$ we also have $M_3\subseteq M_3'=f_2(M_2')$. Therefore the projective system $(f(M_2)^\vee[1/X])_{M_2\in\mathcal{M}_\Delta(\pi_2^{H_{\Delta,0}})}$ satisfies the Mittag-Leffler property since for $M_2,M_2'\in \mathcal{M}_\Delta(\pi_2^{H_{\Delta,0}})$ we also have $M_2+M_2'\in \mathcal{M}_\Delta(\pi_2^{H_{\Delta,0}})$. So by taking the projective limit we obtain an exact sequence
\begin{equation*}
0\to \varprojlim_{M_2\in \mathcal{M}_\Delta(\pi_2^{H_{\Delta,0}})} f_2(M_2)^\vee[1/X]\to \varprojlim_{M_2\in \mathcal{M}_\Delta(\pi_2^{H_{\Delta,0}})} M_2^\vee[1/X]\to \varprojlim_{M_2\in \mathcal{M}_\Delta(\pi_2^{H_{\Delta,0}})} (M_2\cap \pi_1)^\vee[1/X]\to 0
\end{equation*}
In the above short exact sequence the left term is the image of $D^\vee_\Delta(\pi_3)$ in $D^\vee_\Delta(\pi_2)$, the middle term is $D^\vee_\Delta(\pi_2)$, and the term on the right hand side equals $D^\vee_\Delta(\pi_1)$ since all $M_1\in \mathcal{M}_\Delta(\pi_1^{H_{\Delta,0}})$ can also be viewed as a subspace in $\pi_2^{H_{\Delta,0}}$.
\end{proof}

\section{Compatibility with tensor products and parabolic induction}\label{compatible}

Our goal in this section is to generalize the results in sections 5--7 in \cite{B} to the functor $D^\vee_\Delta$.

\subsection{Tensor products}

As in section 5 in \cite{B} let $G'$ be another $\mathbb{Q}_p$-split reductive group over $\mathbb{Q}_p$ with connected centre and Borel subgroup $B'$ with $\mathbb{Q}_p$-split torus $T'$ and unipotent radical $N'$ with compact open subgroup $N_0'$ that we also assume to be totally decomposed. We denote by $\Phi^{+'}$ (resp.\ by $\Delta'$) the set of positive (resp.\ simple) roots corresponding to $B'$. For each $\alpha'\in \Phi^{+'}$ we fix isomorphisms $u_{\alpha'} :\mathbb G_a \to  N_{\alpha'}$ for the root subgroups $N_{\alpha'}\leq N'$ such that $u_{\alpha'}(\Zp)=N_{\alpha'}\cap N_0'=:N'_{\alpha',0}$. Since the centre of $G'$ is connected, there exists a system $(\lambda_{\alpha'^\vee})_{\alpha'\in\Delta'}\in X^\vee(T')^{\Delta'}$ of cocharacters with the property $\beta'\circ\lambda_{\alpha'^\vee}=1$ for all $\alpha'\neq\beta'\in\Delta'$ and $\alpha'\circ\lambda_{\alpha'^\vee}=\id_{\mathbf{G}_m}$. As in \cite{B} and \cite{SVig} we put $\xi:=\sum_{\alpha'\in\Delta'}\lambda_{\alpha'^\vee}$. We define $N'_{\Delta',0}:=\prod_{\alpha'\in\Delta'}N'_{\alpha',0}$ arising naturally as a quotient of the group $N'_0$. We put $H'_{\Delta',0}:=\Ker(N_0'\twoheadrightarrow N'_{\Delta',0})$. Similarly, we consider the rings $A\bg N'_{\Delta',0}\jg$ and $A\bs N'_{\Delta',0}\js[F_{\Delta'}]$ and consider \'etale $T'_+$-modules over the former forming the categories $\mathcal{D}^{et}(T'_+,A\bg N'_{\Delta',0}\jg)$ and $\mathcal{D}^{pro-et}(T'_+,A\bg N'_{\Delta',0}\jg)$ in the usual sense.

We consider the group $G\times G'$ with Borel $B\times B'$ and torus $T\times T'$. We also form the rings $A\bg N_{\Delta,0}\times N'_{\Delta',0}\jg$ and $A\bs N_{\Delta,0}\times N'_{\Delta',0}\js[F_{\Delta'}]$ and consider \'etale $T_+\times T'_+$-modules over the former ring forming the categories $\mathcal{D}^{et}(T_+\times T'_+,A\bg N_{\Delta,0}\times N'_{\Delta',0}\jg)$ and $\mathcal{D}^{pro-et}(T_+\times T'_+,A\bg N_{\Delta,0}\times N'_{\Delta',0}\jg)$ in the usual sense.

For finitely generated modules $M_0$ (resp.\ $M'_0$) over the Iwasawa algebra $A\bs N_{\Delta,0}\js$ (resp.\ over $A\bs N'_{\Delta',0}\js$) the completed tensor product $M_0\hat{\otimes}_A M'_0:=(M_0^\vee\otimes_A M_0'^\vee)^\vee$ is a finitely generated module over the Iwasawa algebra $A\bs N_{\Delta,0}\times N'_{\Delta',0}\js$. It is isomorphic to the projective limit of $M_\ast\otimes_A M'_\ast$ where $M_\ast$ (resp.\ $M'_\ast$) runs through the finite quotients of $M_0$ (resp.\ of $M'_0$).

Now let $M$ (resp.\ $M'$) be a finitely generated module over $A\bg N_{\Delta,0}\jg$ (resp.\ over $A\bg N'_{\Delta',0}\jg$) and choose a finitely generated $A\bs N_{\Delta,0}\js$-submodule $M_0\subset M$ (resp.\ $A\bs N'_{\Delta',0}\js$-submodule $M_0'\subset M'$) such that we have $M=M_0[1/X_\Delta]$ (resp.\ $M'=M_0'[1/X_{\Delta'}]$) (we call these submodules lattices in $M$ (resp.\ in $M'$)). We define $M\hat{\otimes}_A M':=(M_0\hat{\otimes}_A M_0')[1/X_\Delta X_{\Delta'}]$. This is a finitely generated module over $A\bg N_{\Delta,0}\times N'_{\Delta',0}\jg$.

\begin{lem}
$M\hat{\otimes}_A M'$ does not depend on the choice of $M_0$ and $M'_0$ up to a natural isomorphism.
\end{lem}
\begin{proof}
Let $M_1\subset M$ (resp.\ $M'_1\subset M'$) be another lattice. It suffices to treat the case when $M_1'=M_0'$. Assume first that $M_1=X_\Delta^nM_0$ for some (positive) integer $n$. Since the multiplication by $X_\Delta^n$ on $M_0$ induces the multiplication by $(X_\Delta^\#)^n$ on the Pontryagin dual, it induces again the multiplication by $X_\Delta^n$ on the completed tensor product $M_0\hat{\otimes}_A M'_0$ which becomes an isomorphism after inverting $X_\Delta X_{\Delta'}$. So the natural inclusion of $M_1$ in $M$ induces an isomorphism  $(M_0\hat{\otimes}_A M_0')[1/X_\Delta X_{\Delta'}]\cong (M_1\hat{\otimes}_A M_0')[1/X_\Delta X_{\Delta'}]$.

Now let $M_1$ be arbitrary. Since we have $M=M_1[1/X_\Delta]=M_0[1/X_\Delta]$, there exist positive integers $r_1,r_2$ such that $M_1\subseteq X_\Delta^{r_1}M_0\subseteq X_\Delta^{r_2}M_1(\subseteq X_\Delta^{r_1+r_2}M_0)$. By the above case these inclusions induce a sequence of map
\begin{align*}
(M_1\hat{\otimes}_A M_0')[1/X_\Delta X_{\Delta'}]\overset{f_1}{\to} ((X_\Delta^{r_1}M_0)\hat{\otimes}_A M_0')[1/X_\Delta X_{\Delta'}]\overset{f_2}{\to}\\ \overset{f_2}{\to}((X_\Delta^{r_2}M_1)\hat{\otimes}_A M_0')[1/X_\Delta X_{\Delta'}]\overset{f_3}{\to} ((X_\Delta^{r_1+r_2}M_0)\hat{\otimes}_A M_0')[1/X_\Delta X_{\Delta'}]
\end{align*}
such that $f_2\circ f_1$ and $f_3\circ f_2$ are isomorphisms. We obtain that $f_2$ is both injective and surjective therefore an isomorphism. This combined with the above case we obtain statement. The isomorphism is induced by the inclusions $M_1\cap M_0\subseteq M_0$ and $M_1\cap M_0\subseteq M_1$.
\end{proof}

Whenever $D$ (resp.\ $D'$) is an object in $\mathcal{D}^{et}(T_+,A\bg N_{\Delta,0}\jg)$ (resp.\ in $\mathcal{D}^{et}(T'_+,A\bg N'_{\Delta',0}\jg)$) then the completed tensor product $D\hat{\otimes} D'$ can be equipped with the structure of an \'etale $T_+\times T_+'$-module over $A\bg N_{\Delta,0}\times N'_{\Delta',0}\jg$ the following way. Choose an $A\bs N_{\Delta,0}\js$-lattice $D_0$ (resp.\ an $A\bs N'_{\Delta',0}\js$-lattice $D'_0$) in $D$ (resp.\ in $D'$) and an element $t\in T_+$. The action of $\varphi_t$ on $D$ provides us with an isomorphism 
\begin{equation}
1\otimes\varphi_t\colon \varphi_t^* D_0\to A\bs N_{\Delta,0}\js \varphi_t(D_0)\ . \label{compactphitpullback}
\end{equation}
(Here we put $\varphi_t^* D_0:=A\bs N_{\Delta,0}\js \otimes_{\varphi_t}D_0$ as in \cite{EZ}.) Further, we compute
\begin{align*}
(\varphi_t^* D_0)\hat{\otimes}_A D_0'=((\varphi_t^*D_0)^\vee\otimes_A D_0'^\vee)^\vee\cong (A\bs N_{\Delta,0}\js \otimes_{\varphi_t,A\bs N_{\Delta,0}\js}D_0^\vee\otimes_A D_0'^\vee)^\vee\cong\\
\cong A\bs N_{\Delta,0}\js \otimes_{\varphi_t,A\bs N_{\Delta,0}\js}(D_0^\vee\otimes_A D_0'^\vee)^\vee\cong A\bs N_{\Delta,0}\js \otimes_{\varphi_t,A\bs N_{\Delta,0}\js}(D_0\hat{\otimes}_A D_0')\cong\\
\cong A\bs N_{\Delta,0}\times N'_{\Delta',0}\js \otimes_{\varphi_t,A\bs N_{\Delta,0}\times N'_{\Delta',0}\js}(D_0\hat{\otimes}_A D_0')
\end{align*}
as $t\in T_+$ acts trivially on $N'_{\Delta',0}$. By the above identification the isomorphism \eqref{compactphitpullback} induces an isomorphism
\begin{equation*}
1\otimes(\varphi_t\otimes 1)\colon A\bs N_{\Delta,0}\times N'_{\Delta',0}\js \otimes_{\varphi_t,A\bs N_{\Delta,0}\times N'_{\Delta',0}\js}(D_0\hat{\otimes}_A D_0')\to (A\bs N_{\Delta,0}\js \varphi_t(D_0))\hat{\otimes}D_0'\ .
\end{equation*}
Inverting $X_\Delta X_{\Delta'}$ and noting that $A\bs N_{\Delta,0}\js \varphi_t(D_0)$ is a lattice in $D$ (as $D$ is an \'etale $T_+$-module) we obtain an isomorphism
\begin{equation*}
1\otimes(\varphi_t\otimes 1)\colon A\bg N_{\Delta,0}\times N'_{\Delta',0}\jg \otimes_{\varphi_t,A\bg N_{\Delta,0}\times N'_{\Delta',0}\jg}(D\hat{\otimes}_A D')\to D\hat{\otimes}_AD'\ .
\end{equation*}
By symmetry, we obtain a similar isomorphism for any $t'\in T'_+$. This equips $D\hat{\otimes}_A D'$ with the structure of an \'etale $T_+\times T'_+$-module over $A\bg N_{\Delta,0}\times N'_{\Delta',0}\jg$. Now if $D=\varprojlim_{i\in I}D_i$ (resp.\ $D'=\varprojlim_{i'\in I'}D'_{i'}$) is an object in $\mathcal{D}^{pro-et}(T_+,A\bg N_{\Delta,0}\jg)$ (resp.\ in $\mathcal{D}^{pro-et}(T'_+,A\bg N'_{\Delta',0}\jg)$) then we may form the completed tensor product $D\hat{\otimes}_A D':=\varprojlim_{i\in I, i'\in I'}D_i\hat{\otimes}_A D'_{i'}$. This is an object in $\mathcal{D}^{pro-et}(T_+\times T'_+,A\bg N_{\Delta,0}\times N'_{\Delta',0}\jg)$.

\begin{pro}
Let $\pi$ (resp.\ $\pi'$) be a smooth representation of $B$ (resp.\ of $B'$) over $\kappa$. Then we have an isomorphism 
\begin{equation*}
D_{\Delta\cup\Delta'}^\vee(\pi\otimes_\kappa\pi')\cong D^\vee_\Delta(\pi)\hat{\otimes}_{\kappa}D^\vee_{\Delta'}(\pi')
\end{equation*}
in $\mathcal{D}^{pro-et}(T_+\times T'_+,\kappa\bg N_{\Delta,0}\times N'_{\Delta',0}\jg)$.
\end{pro}
\begin{proof}
We proceed in $3$ steps.
\emph{Step 1.} We show that we have 
\begin{equation}
(\pi\otimes_\kappa \pi')^{H_{\Delta,0}\times H'_{\Delta',0}}=\pi^{H_{\Delta,0}}\otimes_\kappa\pi^{H'_{\Delta',0}}\ . \label{tensorinv}
\end{equation}
Since $\kappa$ is a field, $\otimes_\kappa$ is exact, therefore the right hand side of \eqref{tensorinv} is can be viewed as a subspace in $\pi\otimes_\kappa \pi'$. The inclusion $\supseteq$ in \eqref{tensorinv} is therefore clear. Let now $\sum_{i=1}^rm_i\otimes m_i'$ be an arbitrary element in the left hand side of \eqref{tensorinv}. Since $\kappa$ is a field, we may assume that the elements $m_1',\dots,m_r'\in\pi'$ are linearly independent over $\kappa$. Now for any $h\in H_{\Delta,0}$ we have 
\begin{align*}
\sum_{i=1}^rm_i\otimes m_i'=h(\sum_{i=1}^rm_i\otimes m_i')=\sum_{i=1}^r(hm_i)\otimes m_i'\ .
\end{align*}
Since the elements $m_i'$ ($1\leq i\leq r$) are linearly independent, we deduce that $m_i=hm_i$ for all $1\leq i\leq r$ showing that $\sum_{i=1}^rm_i\otimes m_i'$ lies in $\pi^{H_{\Delta,0}}\otimes_\kappa\pi'$. Now we may rewrite $\sum_{i=1}^rm_i\otimes m_i'$ so that the $m_i\in\pi^{H_{\Delta,0}}$ ($1\leq i\leq r$) are linearly independent (possibly losing the assumption that $m_i'$ ($1\leq i\leq r$) are linearly independent). By the same process with $H'_{\Delta',0}$ we deduce \eqref{tensorinv}.

\emph{Step 2.} Now let $M$ (resp.\ $M'$) be an object in $\mathcal{M}_{\Delta}(\pi^{H_{\Delta,0}})$ (resp.\ in $\mathcal{M}_{\Delta'}({\pi'}^{H'_{\Delta',0}})$). Then $M\otimes_\kappa M'$ is $T_0$-invariant and admissible as a representation of $N_{\Delta,0}\times N'_{\Delta',0}$ since for an open subgroup $N_{\Delta,\ast}\leq N_{\Delta,0}$ (resp.\ $N'_{\Delta',\ast}\leq N'_{\Delta',0}$) the product $N_{\Delta,\ast}\times N'_{\Delta',\ast}$ is open in $N_{\Delta,0}\times N'_{\Delta',0}$ and the $\kappa$-vectorspace $(M\otimes_\kappa M')^{N_{\Delta,\ast}\times N'_{\Delta',\ast}}=M^{N_{\Delta,\ast}}\otimes_\kappa {M'}^{N'_{\Delta',\ast}}$ is finite dimensional (using again the argument above). Moreover, choose generators $m_1,\dots,m_r$ (resp.\ $m'_1,\dots,m'_{r'}$) of $M$ (resp.\ of $M'$) as a module over $\kappa\bs N_{\Delta,0}\js[F_\Delta]$ (resp.\ over $\kappa\bs N'_{\Delta',0}\js[F_{\Delta'}]$). Then the elements $m_i\otimes m'_j$ ($1\leq i\leq r$, $1\leq j\leq r'$) generate $M\otimes_\kappa M'$ as a module over $\kappa\bs N_{\Delta,0}\times N'_{\Delta',0}\js[F_{\Delta\cup\Delta'}]$ since $\kappa\bs N_{\Delta,0}\js[F_\Delta]$ and $\kappa\bs N'_{\Delta',0}\js[F_{\Delta'}]$ are subrings of this ring and any element in $M\otimes_\kappa M'$ can be written as a sum of elements of the form $m\otimes m'$ where $m\in M$ and $m'\in M'$. So $M\otimes_\kappa M'$ is an object in $\mathcal{M}_{\Delta\cup\Delta'}((\pi\otimes_\kappa\pi')^{H_{\Delta,0}\times H'_{\Delta',0}})$.

\emph{Step 3: We show that the elements of the form $M\otimes_\kappa M'$ for $M$ (resp.\ $M'$) in $\mathcal{M}_{\Delta}(\pi^{H_{\Delta,0}})$ (resp.\ in $\mathcal{M}_{\Delta'}({\pi'}^{H'_{\Delta',0}})$) are cofinal in $\mathcal{M}_{\Delta\cup\Delta'}((\pi\otimes_\kappa\pi')^{H_{\Delta,0}\times H'_{\Delta',0}})$.} Let $M''$ be in $\mathcal{M}_{\Delta\cup\Delta'}((\pi\otimes_\kappa\pi')^{H_{\Delta,0}\times H'_{\Delta',0}})$ with a finite set $m_1'',\dots,m_{r''}''$ of generators as a module over $\kappa\bs N_{\Delta,0}\times N'_{\Delta',0}\js[F_{\Delta\cup\Delta'}]$ and by \eqref{tensorinv} write each of the $m_i''$ ($1\leq i\leq r''$) as a sum of elementary tensors of the form $m\otimes m'$ with $m\in \pi^{H_{\Delta,0}}$ and $m'\in {\pi'}^{H'_{\Delta',0}}$. Let $U$ (resp.\ $U'$) be the set of $m\in  \pi^{H_{\Delta,0}}$ (resp.\ of $m'\in {\pi'}^{H'_{\Delta',0}}$) appearing in these elementary tensors. These are finite sets. We may assume without loss of generality that the elements of the set $U'$ are linearly independent and each $m'\in U'$ appears only once in each $m_i''$ ($1\leq i\leq r''$). Moreover, for each $m'\in U'$ let $U(m')\subseteq U$ be the subset of those $m\in U$ such that $m\otimes m'$ appears as an elementary tensor in one of the $m_i''$ ($1\leq i\leq r''$). We then have $U=\bigcup_{m'\in U'}U(m')$.
\begin{lem}
The $\kappa\bs N_{\Delta,0}\js[F_\Delta]$-submodule of $\pi^{H_{\Delta,0}}$ generated by the elements of $U$ is admissible as a representation of $N_{\Delta,0}$.
\end{lem}
\begin{proof}
Since the subring $\kappa\bs N_{\Delta,0}\js[F_\Delta]\leq \kappa\bs N_{\Delta,0}\times N'_{\Delta',0}\js[F_{\Delta\cup\Delta'}]$ acts via the first component on the tensor product $M\otimes_\kappa M'$, any element $m_\ast$ in $\sum_{i=1}^{r''}\kappa\bs N_{\Delta,0}\js[F_\Delta]m_i''$ can be uniquely written as a sum $m_\ast=\sum_{m'\in U'}m_\ast(m')\otimes m'$ for some elements $m_\ast(m')\in \sum_{m\in U(m')}\kappa\bs N_{\Delta,0}\js[F_\Delta] m$. Since the representation $\pi'$ of $B'$ is smooth and $U'$ is a finite set, there exists an open subgroup $N'_{\Delta',\ast}\leq N'_{\Delta',0}$ stabilizing all the elements $m'$ in $U'$. Now for an open subgroup $N_{\Delta,\ast}\leq N_{\Delta,0}$, an element $m_\ast=\sum_{m'\in U'}m_\ast(m')\otimes m'$ in  $\sum_{i=1}^{r''}\kappa\bs N_{\Delta,0}\js[F_\Delta]m_i''$ is stabilized by $N_{\Delta,\ast}\times N'_{\Delta',\ast}$ if and only if each $m_\ast(m')$ ($m'\in U'$) is stabilized by $N_{\Delta,\ast}$. Since $M\otimes_\kappa M'$ is admissible as a representation of $N_{\Delta,0}\times N'_{\Delta',0}$, we deduce that $$(\sum_{i=1}^{r''}\kappa\bs N_{\Delta,0}\js[F_\Delta]m_i'')^{N_{\Delta,\ast}}$$
is finite dimensional for any open subgroup $N_{\Delta,\ast}\leq N_{\Delta,0}$. In particular, $\sum_{i=1}^{r''}\kappa\bs N_{\Delta,0}\js[F_\Delta]m_i''$ is admissible as a representation of $N_{\Delta,0}$. Now for each $m'_0\in U'$ consider the projection 
\begin{eqnarray*}
\Pi_{m'_0}\colon \sum_{m'\in U'}\pi^{H_{\Delta,0}}\otimes m'&\to & \pi^{H_{\Delta,0}}\\
\sum_{m'_\ast\in U'}m(m')\otimes m' &\mapsto & m(m'_0)\ .
\end{eqnarray*}
$\Pi_{m'_0}$ is $\kappa\bs N_{\Delta,0}\js[F_\Delta]$-linear. Moreover, we have $$\Pi_{m'_0}\left(\sum_{i=1}^{r''}\kappa\bs N_{\Delta,0}\js[F_\Delta]m_i''\right)=\sum_{m\in U(m'_0)}\kappa\bs N_{\Delta,0}\js[F_\Delta]m\ .$$ In particular, $\sum_{m\in U(m'_0)}\kappa\bs N_{\Delta,0}\js[F_\Delta]m$ is admissible as a representation of $N_{\Delta,0}$ being the image of an admissible representation. We deduce the statement from the equality $U=\bigcup_{m'\in U'}U(m')$ noting that the finite sum of admissible representations is also admissible. 
\end{proof}
We choose a basis $U_\ast\subset U$ of the $\kappa$-vectorspace generated by $m\in U$. Now by possibly grouping together some elementary tensors in $m_i''$ ($1\leq i\leq r''$), we may write each element $m_i''$ as a finite sum of elementary tensors $m_i''=\sum_{m\in U_\ast}m\otimes m'_i(m)$. By using again the above Lemma in the symmetric situation we deduce that the $\kappa\bs N'_{\Delta',0}\js[F_{\Delta'}]$-submodule of ${\pi'}^{H'_{\Delta',0}}$ generated by $U'_\ast:=\{m'_i(m)\mid 1\leq i\leq r'', m\in U_\ast\}$ is admissible as a representation of $N'_{\Delta',0}$. Moreover, $M''$ is clearly contained in $M\otimes_\kappa M'$ where $M$ (resp.\ $M'$) is the $\kappa\bs N_{\Delta,0}\js[F_{\Delta}]$-submodule (resp.\ $\kappa\bs N'_{\Delta',0}\js[F_{\Delta'}]$-submodule) of $\pi^{H_{\Delta,0}}$ (resp.\ of  ${\pi'}^{H'_{\Delta',0}}$) generated by $U_\ast$ (resp.\ by $U'_\ast$). Now if we replace $U_\ast$ and $U'_\ast$ by their finite $T_0$-orbit (resp.\ $T'_0$-orbit) we may assume that $M$ and $M'$ are stable under the action of $T_0$ (resp.\ of $T_0'$). This finishes the proof of Step $3$.

The statement follows from the cofinality of the $\kappa\bs N_{\Delta,0}\times N'_{\Delta',0}\js [F_{\Delta\times \Delta'}]$-submodules of the form $M\otimes_\kappa M'$ for $M\in \mathcal{M}_{\Delta}(\pi^{H_{\Delta,0}})$ and $M'\in \mathcal{M}_{\Delta'}({\pi'}^{H'_{\Delta',0}})$ in $\mathcal{M}_{\Delta\cup\Delta'}((\pi\otimes_\kappa\pi')^{H_{\Delta,0}\times H'_{\Delta',0}})$.
\end{proof}

\begin{rem}
The above proof seems to only work in the case of $\kappa$-representations, but not for $A$-representations. However, it might be possible to deduce the general case (or at least the case when the representations are on free $A$-modules) from the above Proposition. Note that in case $A=\kappa$ the above statement is slightly stronger than Prop.\ 5.5 in \cite{B} since we do not assume the property Ad (or its analogue in this situation) in Def.\ 2.2 in \cite{B} that all the finitely generated $\kappa\bs N_{\Delta,0}\js [F_{\Delta}]$-submodules (resp.\ $\kappa\bs N'_{\Delta',0}\js [F_{\Delta'}]$-submodules) of $\pi^{H_{\Delta,0}}$ (resp.\ of ${\pi'}^{H'_{\Delta',0}}$) are admissible as representations of $N_{\Delta,0}$ (resp.\ of $N'_{\Delta',0}$) (see also Remark 5.6(ii) of \cite{B}).
\end{rem}

\subsection{Parabolic induction}

As in \cite{B} let $P\leq G$ be a parabolic subgroup containing $B$ with Levi component $L_P$ and unipotent radical $N_P\leq N$. Denote by $P^-$ the opposite parabolic and by $N_{P^-}$ its unipotent radical. The group $L_P$ is also a $\mathbb{Q}_p$-split reductive group and we fix its Borel subgroup $B_{L_P}:=B\cap L_P$ with unipotent radical $N_{L_P}:=N\cap L_P$. We denote by $\Phi_P$ the roots of the pair $(L_P,T)$ and by $\Phi_P^+\subseteq \Phi_P$ the subset of positive roots with respect to $B_{L_P}$ and by $\Delta_P\subseteq \Phi^+_P$ the set of simple roots. Since the centre $Z(G)$ is connected, the $\mathbb{Z}$-module $X(T)/(\bigoplus_{\alpha\in\Delta}\mathbb{Z}\alpha)$ is torsion-free. Therefore $X(T)/(\bigoplus_{\alpha\in\Delta_P}\mathbb{Z}\alpha)$ is torsion-free, too, because $\bigoplus_{\alpha\in\Delta_P}\mathbb{Z}\alpha$ is a direct summand in $\bigoplus_{\alpha\in\Delta}\mathbb{Z}\alpha$. 

We adapt the constructions of the previous sections to the group $L_P$: We denote by $N_{L_P,0}:=N_0\cap N_{L_P}$ and $N_{\Delta_P,0}:=\prod_{\alpha\in\Delta_P}N_{\alpha,0}$. We regard $N_{\Delta_P,0}$ as a direct summand in $N_{\Delta,0}$ and also as a quotient of the group $N_{L_P,0}$. We denote by $H_{\Delta_P,0}$ the kernel of the quotient map $N_{L_P,0}\twoheadrightarrow N_{\Delta_P,0}$.

As in \cite{B} denote by $W:=N_G(T)/T$ (resp.\ by $W_P:=N_{L_P}(T)/T$) the Weyl group of $G$ (resp.\ of $L_P$) and by $w_0\in W$ the element of maximal length. We have a canonical system $$K_P:=\{w\in W\mid w^{-1}(\Phi_P^+)\subseteq\Phi^+\}$$
of representatives (the Kostant representatives) of the right cosets $W_P\backslash W$. We have a generalized Bruhat decomposition
$$G=\coprod_{w\in K_P}P^-wB=\coprod_{w\in K_P}P^-wN\ .$$

Now let $\pi_P$ be a smooth representation of $L_P$ over $A$. We regard $\pi_P$ as a representation of $P^-$ via the quotient map $P^-\twoheadrightarrow L_P$. Then the parabolically induced representation $\Ind_{P^-}^G\pi_P$ admits \cite{V} (see also \cite{E} \S 4.3) a filtration by $B$-subrepresentations whose graded pieces are contained in
$$\mathcal{C}_w(\pi_P):=c-\Ind_{P^-}^{P^-wN}\pi_P$$
for $w\in K_P$ where $c-\Ind$ stands for the space of locally constant functions with compact support modulo $P^-$. $B$ acts on $\mathcal{C}_w(\pi_P)$ by right translations. Moreover, the first graded piece equals $\mathcal{C}_1(\pi_P)$.

\begin{lem}\label{DDeltaCw}
Let $\pi'\leq \mathcal{C}_w(\pi_P)$ be any $B$-subrepresentation for some $w\in K_P\setminus\{1\}$. Then we have $D^\vee_\Delta(\pi')=0$.
\end{lem}
\begin{proof}
By the right exactness (and contravariance) of $D^\vee_\Delta$ (Thm.\ \ref{rightexact}) it suffices to treat the case $\pi'=\mathcal{C}_w(\pi_P)$. This follows from Prop.\ 6.2 in \cite{B} using that $A\bg X\jg\otimes_{\ell,A\bg N_{\Delta,0}\jg}\cdot$ is faithful (Prop.\ \ref{reduceellexact}) and that we have a surjective map from $D^\vee_{\xi}(\pi')$ to $A\bg X\jg\otimes_{\ell,A\bg N_{\Delta,0}\jg}D^\vee_\Delta(\pi')$ by the Remark after Prop.\ \ref{reduceellexact}.
\end{proof}

\begin{thm}\label{parindDvee}
Let $\pi_P$ be a smooth locally admissible representation of $L_P$ over $A$ which we view by inflation as a representation of $P^-$. We have an isomorphism 
\begin{equation*}
D^\vee_{\Delta}\left(\Ind_{P^-}^G\pi_P\right)\cong A\bg N_{\Delta,0}\jg\hat{\otimes}_{A\bg N_{\Delta_P,0}\jg}D^\vee_{\Delta_P}(\pi_P)
\end{equation*}
in the category $\mathcal{D}^{pro-et}(T_+,A\bg N_{\Delta,0}\jg)$.
\end{thm}
\begin{rem}
Here $\hat{\otimes}$ stands for taking tensor products on finitely generated quotients and then inverse limits.
\end{rem}
\begin{proof}
By Lemma \ref{DDeltaCw} and Thm.\ \ref{rightexact} it suffices to show the isomorphism $D^\vee_\Delta(\mathcal{C}_1(\pi_P))\cong A\bg N_{\Delta,0}\jg\hat{\otimes}_{A\bg N_{\Delta_P,0}\jg}D^\vee_{\Delta_P}(\pi_P)$.

The map $\mathcal{C}_1(\pi_P)\to c-\Ind_{N_{L_P}}^N\pi_P$ sending a function $f\colon P^-N\to \pi_P$ to its restriction to $N$ is a $B$-equivariant isomorphism. So we make this identification. Denote by $\mathcal{C}_{1,0}(\pi_P)$ the $B_+$-subrepresentation of $\mathcal{C}_1(\pi_P)$ of functions supported on $N_{L_P}N_0$ and let $M$ be arbitrary in $\mathcal{M}_\Delta(\mathcal{C}_1(\pi_P)^{H_{\Delta,0}})$. So we have a short exact sequence
\begin{equation*}
0\to \left(M/M\cap \mathcal{C}_{1,0}(\pi_P)\right)^\vee[1/X_\Delta]  \to M^\vee[1/X_\Delta]\to (M\cap \mathcal{C}_{1,0}(\pi_P))^\vee[1/X_\Delta]\to 0
\end{equation*}
in $\mathcal{D}^{et}(T_+,A\bg N_{\Delta,0}\jg)$. By Lemma 6.6 in \cite{B} and Prop.\ \ref{reduceellexact} we deduce that the left hand side $\left(M/M\cap \mathcal{C}_{1,0}(\pi_P)\right)^\vee[1/X_\Delta]$ vanishes. On the other hand, by Prop.\ \ref{dualcompactfingen} we find as in the proof of Prop.\ \ref{intersectionfingen} a finitely generated $A\bs N_{\Delta,0}\js[F_\Delta]$-submodule $M_1\leq M\cap \mathcal{C}_{1,0}(\pi_P)$ that is $T_0$-stable and admissible as a representation of $N_{\Delta,0}$ such that we have $M_1^\vee[1/X_\Delta]=(M\cap \mathcal{C}_{1,0}(\pi_P))^\vee[1/X_\Delta]$. Therefore we obtain an isomorphism $D^\vee_\Delta(\mathcal{C}_1(\pi_P))\cong D^\vee_\Delta(\mathcal{C}_{1,0}(\pi_P))$.

Now we identify $\mathcal{C}_{1,0}(\pi_P)$ with $\Ind_{N_{L_P,0}}^{N_0}\pi_P$. On the latter $t\in T_+$ acts via the formula 
\begin{align}
(tf)(v)=\begin{cases}0&\text{if }v\in N_0\setminus N_{L_P,0}tN_0t^{-1};\\ u_1t\cdot f(v_1)&\text{if }v=u_1tv_1t^{-1}\in N_{L_P,0}tN_0t^{-1}\end{cases}=\notag\\
=\sum_{u\in J(N_{L_P,0}/tN_{L_P,0}t^{-1})}ut\cdot f(t^{-1}u^{-1}vt)\label{tfv}
\end{align}
by extending in \eqref{tfv} the functions $f\in \Ind_{N_{L_P,0}}^{N_0}\pi_P$ to $N$ with $0$ outside $N_0$.

\begin{lem}\label{HDeltainvariants}
We have a $B_0$-equivariant identification $\mathcal{C}_{1,0}(\pi_P)^{H_{\Delta,0}}\cong\Ind_{N_{\Delta_P,0}}^{N_{\Delta,0}}\pi_P^{H_{\Delta_P,0}}$.
\end{lem}
\begin{proof}
Let $f\colon N_0\to \pi_P$ be an $H_{\Delta,0}$-invariant function in $\mathcal{C}_{1,0}(\pi_P)$. Now for any $v\in N_0$ and $h\in H_{\Delta,0}$ we have $f(v)=(hf)(v)=f(vh)$. Hence, since $H_{\Delta,0}$ is normal in $N_0$, we deduce $f(h^{-1}vh)=f(v\cdot v^{-1}h^{-1}vh)=f(v)$. Now if we have $h\in H_{\Delta_P,0}\leq H_{\Delta,0}$ and $v\in N_0$ then we compute $f(v)=f(h\cdot h^{-1}vh)=h\cdot f(h^{-1}vh)=h\cdot f(v)$. So we obtain that the image of $f$ lies in $\pi_P^{H_{\Delta_P,0}}$. Using again that $f$ is $H_{\Delta,0}$-invariant the function $\tilde{f}\colon N_{\Delta,0}\to \pi_P^{H_{\Delta_P,0}}$, $\tilde{f}(vH_{\Delta,0}):=f(v)$ is well-defined. Vica versa, if $\widetilde{f}\colon N_{\Delta,0}\to \pi_P^{H_{\Delta_P,0}}$ is a $N_{\Delta_P,0}$-invariant map then we may lift it to a map $f\colon N_{L_P,0}\backslash N_0/H_{\Delta,0}$. Therefore the map $f\mapsto\tilde{f}$ is a $B_0$-equivariant bijection. 
\end{proof}

\begin{lem}
Let $M$ be an object in $\mathcal{M}_\Delta(\mathcal{C}_{1,0}(\pi_P)^{H_{\Delta,0}})$. Then the set $M_\ast:=\{f(1)\in \pi_P^{H_{\Delta_P,0}}\mid f\in M\}$ is an object in $\mathcal{M}_{\Delta_P}(\pi_P^{H_{\Delta_P,0}})$.
\end{lem}
\begin{proof}
We proceed in $3$ steps.

\emph{Step 1: We show that $M_\ast$ is an $A\bs N_{\Delta_P,0}\js[F_{\Delta_P}]$-submodule.} At first note that there is an alternative description of $M_\ast$ as $M_\ast=\{m\in \pi_P^{H_{\Delta_P,0}}\mid \exists f\in M,v\in N_0\colon f(v)=m\}$ since for any $v\in N_0$ and $f\in M$ we have $(vf)(1)=f(v)$. Let $t$ be in $T_+$ and $f$ be in $\mathcal{C}_{1,0}(\pi_P)^{H_{\Delta,0}}$ arbitrary and compute
\begin{align}
F_t(f)(v)=\sum_{u\in J(H_{\Delta,0}/tH_{\Delta,0}t^{-1})}(utf)(v)=\sum_{u\in J(H_{\Delta,0}/tH_{\Delta,0}t^{-1})}(tf)(vu)=\notag\\
=\sum_{u\in J(H_{\Delta,0}/tH_{\Delta,0}t^{-1})}\sum_{u'\in J(N_{L_P,0}/tN_{L_P,0}t^{-1})}u't\cdot f(t^{-1}u'^{-1}vut)\label{FtactingC10}
\end{align}
using \eqref{tfv}. Set $J:=J(H_{\Delta,0}/tH_{\Delta,0}t^{-1})$ and $J':=J(N_{L_P,0}/tN_{L_P,0}t^{-1})$. Now consider the bipartite graph on the set $J \coprod J'$ in which there is an edge between $u\in J$ and $u'\in J'$ if and only if $t^{-1}u'^{-1}vut$ lies in $N_0$. If a pair $(u,u')$ is not an edge then we have $f(t^{-1}u'^{-1}vut)=0$ as $f$ is supported on $N_0$. Note that the degree of each vertex in this graph is at most $1$. Whenever $v$ lies in $N_{L_P}tN_0t^{-1}H_{\Delta,0}$ then we do have a (non-unique) $u_0'\in J'$ and a unique $u(u_0')\in J$ such that $u_0'^{-1}vu(u_0')$ lies in $tN_0t^{-1}$. Moreover, the class of $u_0'$ in $N_{\Delta_P,0}/tN_{\Delta_P,0}t^{-1}$ is unique since $H_{\Delta,0}$ lies in the kernel of the quotient map $N_0\twoheadrightarrow N_{\Delta,0}$ and the value $f(t^{-1}u'^{-1}vu(u')t)$ only depends on the class of $u'$ in $N_{\Delta_P,0}$ for each $u'\in J'$. Since \eqref{FtactingC10} does not depend on the choice of $J'$, we may choose $J':=J(N_{L_P,0}/tN_{L_P,0}t^{-1}H_{\Delta_P,0})J(H_{\Delta_P,0}/tH_{\Delta_P,0}t^{-1})$ and deduce
\begin{align}
F_t(f)(v)=\sum_{\omega\in J(H_{\Delta_P,0}/tH_{\Delta_P,0}t^{-1})}\sum_{u\in J}\sum_{u'\in J(N_{L_P,0}/tN_{L_P,0}t^{-1}H_{\Delta_P,0})}u'\omega t\cdot f(t^{-1}\omega^{-1}u'^{-1}vut)=\notag\\ 
=\sum_{u\in J}\sum_{u'\in J(N_{L_P,0}/tN_{\Delta_P,0}t^{-1}H_{\Delta_P,0})}u'F_t(f(t^{-1}u'^{-1}vut))\label{Ftfv}
\end{align}
so that for each $v\in N_0$ there is at most one pair $(u,u')\in J\times J(N_{L_P,0}/tN_{L_P,0}t^{-1}H_{\Delta_P,0})$ for which $f(t^{-1}u'^{-1}vut)\neq 0$. In particular, we have $F_t(f)(1)=F_t(f(1))$ whence $M_\ast$ is an $A\bs N_{\Delta_P,0}\js[F_{\Delta_P}]$-submodule. 

\emph{Step 2. We show that $M_\ast$ is finitely generated over $A\bs N_{\Delta_P,0}\js[F_{\Delta_P}]$.} Now let $f_1,\dots,f_r$ be a finite set of generators of $M$ as a module over $A\bs N_{\Delta,0}\js[F_\Delta]$. Since each $f_i\colon N_0\to \pi_P^{H_{\Delta_P,0}}$ is continuous, $N_0$ is compact, $\pi_P^{H_{\Delta_P,0}}$ is discrete, the set $U_0:=\{f_i(v)\mid v\in N_0,1\leq i\leq r\}$ is finite. Moreover, since $\pi_P$ is locally admissible and $t_\alpha$ lies in the centre of $L_P$ for all $\alpha\in \Delta\setminus\Delta_P$, the orbit of the elements in $U_0$ under the action of $\prod_{\alpha\in\Delta\setminus\Delta_P} t_{\alpha}^{\mathbb{Z}}$ is also finite. Therefore the union $U:=\prod_{\alpha\in\Delta\setminus\Delta_P} t_{\alpha}^{\mathbb{Z}}U_0$ is also finite. Since $f_1,\dots,f_r$ generates $M$, we may write any function $f\in M$ as a finite sum of functions of the form $vF_t(f_i)$ for $v\in N_0$,  $t\in \prod_{\alpha\in\Delta} t_{\alpha}^{\mathbb{N}}$, and $1\leq i\leq r$. We decompose $t=t_1t_2$ for $t_1\in \prod_{\alpha\in\Delta_P} t_{\alpha}^{\mathbb{N}}$ and $t_2\in\prod_{\alpha\in\Delta\setminus\Delta_P} t_{\alpha}^{\mathbb{N}}$. By \eqref{Ftfv} we obtain that 
\begin{align*}
(vF_t(f_i))(1)=F_t(f_i)(v)=\sum_{u\in J}\sum_{u'\in J(N_{L_P,0}/tN_{L_P,0}t^{-1}H_{\Delta_P,0})}u'F_{t_1}(F_{t_2}(f_i(t^{-1}u'^{-1}vut))) 
\end{align*}
showing that $U$ generates $M_\ast$ as a module over $A\bs N_{\Delta_P,0}\js [F_{\Delta_P}]$ since $F_{t_2}(f_i(t^{-1}u'^{-1}vut))$ lies in $U$.

\emph{Step 3: We show that $M_\ast$ is admissible as a representation of $N_{\Delta_P}$.} For each generator $m\in U$ of $M_\ast$ choose a function $f_m\in M$ with $f_m(1)=m$. We put $V:=\{f_m\in M\mid m\in U\}$. Let $N_\ast$ be an open subgroup of $N_{\Delta,0}$ stabilizing all the functions in $V$ and put $N_{\Delta\setminus\Delta_P,\ast}:=N_\ast\cap N_{\Delta\setminus\Delta_P,0}$ where $N_{\Delta\setminus\Delta_P,0}:=\prod_{\alpha\in \Delta\setminus\Delta_P}N_{\alpha,0}$. Let $t$ be in $\prod_{\alpha\in\Delta_P} t_{\alpha}^{\mathbb{N}}$ and $v$ be in $N_{\Delta_P,0}$ be arbitrary. Using $F_t(f)(1)=F_t(f(1))$ which follows from \eqref{Ftfv} and that every function $f\in \Ind_{N_{\Delta_P,0}}^{N_0}\pi_P$ satisfies $f(v)=v\cdot f(1)$ by definition, we obtain 
\begin{equation}
(vF_t(f_m))(1)=F_t(f_m)(v)=v\cdot F_t(f_m)(1)=v\cdot F_t(f_m(1))=v\cdot F_t(m)\ .\label{ANDeltaPFDeltaPf1}
\end{equation}
This shows that the map
\begin{equation*}
M_1:=\sum_{m\in U}A\bs N_{\Delta_P,0}\js [F_{\Delta_P}]f_m\ni f\mapsto f(1)\in M_\ast
\end{equation*}
is a surjective $A\bs N_{\Delta_P,0}\js [F_{\Delta_P}]$-module homomorphism. Moreover, both $N_{\Delta_P,0}$ and $\prod_{\alpha\in\Delta_P} t_{\alpha}^{\mathbb{Z}}$ commute with $N_{\Delta\setminus\Delta_P,0}$, therefore any element in $M_1$ is $N_{\Delta\setminus\Delta_P,\ast}$-invariant. Therefore $M_1$ is admissible as a representation of $N_{\Delta_P,0}$ since for any open subgroup $N_{\Delta_P,\ast}\leq N_{\Delta_P,0}$ we have $$M_1^{N_{\Delta_P,\ast}}=M_1^{N_{\Delta_P,\ast}\times N_{\Delta\setminus\Delta_P,\ast}}\subset M^{N_{\Delta_P,\ast}\times N_{\Delta\setminus\Delta_P,\ast}}$$
wich is finite. In particular, $M_\ast$ is also admissible as it arises as a quotient of $M_1$.

Finally, it is clear that $M_\ast$ is invariant under the action of $T_0$ hence belongs to $\mathcal{M}_{\Delta_P}(\pi_P^{H_{\Delta_P,0}})$. 
\end{proof}
Now let $M'$ be arbitrary in $\mathcal{M}_{\Delta_P}(\pi_P)$ and define $\widetilde{M'}:=\{f\in \mathcal{C}_{1,0}(\pi_P)^{H_{\Delta,0}}\mid f(v)\in M'\text{ for all }v\in N_0\}\cong \Ind_{N_{\Delta_P,0}}^{N_{\Delta,0}}M'$.
\begin{lem}
$\widetilde{M'}$ lies in $\mathcal{M}_\Delta(\mathcal{C}_{1,0}(\pi_P)^{H_{\Delta,0}})$. For any $M$ in $\mathcal{M}_\Delta(\mathcal{C}_{1,0}(\pi_P)^{H_{\Delta,0}})$ we have $M\subseteq \widetilde{M_\ast}$.
\end{lem}
\begin{proof}
By the definition of $M_\ast$ and $\widetilde{M_\ast}$, $M$ is contained in $\widetilde{M_\ast}$.

Let $M'$ be arbitrary in $\mathcal{M}_{\Delta_P}(\pi_P^{H_{\Delta_P,0}})$. By \eqref{Ftfv}, $\widetilde{M'}$ is an $A\bs N_{\Delta,0}\js[F_\Delta]$-submodule of $\mathcal{C}_{1,0}(\pi_P)^{H_{\Delta,0}}$. Choose a finite set $m_1,\dots,m_r\in M'$ of generators as a module over $A\bs N_{\Delta_P,0}\js[F_{\Delta_P}]$. For each $1\leq i\leq r$ let $f_i\in \mathcal{C}_{1,0}(\pi_P)$ be the constant function $f_i(v):=m_i$ for all $v\in N_0$. Since $f_i$ is constant, it is $H_{\Delta,0}$-invariant therefore lies in $\widetilde{M'}$. We claim that the functions $f_1,\dots,f_r$ generate $\widetilde{M'}$ as a module over $A\bs N_{\Delta,0}\js [F_\Delta]$. Let $f$ be any function in $\widetilde{M'}$. By Lemma \ref{HDeltainvariants}, we may view $f$ as a function on $N_{\Delta,0}$ that is determined by its values on the subgroup $N_{\Delta\setminus\Delta_P,0}$. Moreover, since the restriction of $f$ to $N_{\Delta\setminus\Delta_P,0}$ is continuous, it is constant on the cosets of subgroup $tN_{\Delta\setminus\Delta_P,0}t^{-1}$ for some $t\in \prod_{\alpha\in \Delta\setminus\Delta_P}t_\alpha^{\mathbb{N}}$ as these subgroups form a system of neighbourhoods of $1$ in $N_{\Delta\setminus\Delta_P,0}$. So it suffices to show that for each coset $v_0tN_{\Delta\setminus\Delta_P,0}t^{-1}$ ($v_0\in N_{\Delta\setminus\Delta_P,0}$) and $m$ in $M'$ the function
\begin{eqnarray*}
f_{v_0,t,m}\colon N_0&\to& \pi_P^{H_{\Delta_P,0}}\\
v&\mapsto&\begin{cases}u\cdot m&\text{if }v\in uv_0tN_{P,0}t^{-1}H_{\Delta,0}\text{ for some }u\in N_{L_P,0}\ ;\\0 &\text{otherwise },\end{cases}
\end{eqnarray*}
lies in $\sum_{i=1}^rA\bs N_{\Delta,0}\js [F_\Delta]f_i$. Since $m_1,\dots,m_r$ generate $M'$, we may write $m$ as a finite sum $\sum_{i=1}^r\lambda_i m_i$ for some $\lambda_i\in A\bs N_{\Delta_P,0}\js [F_{\Delta_P}]$ ($1\leq i\leq r$). We have $\sum_{i=1}^r\lambda_if_i(1)=m$ by \eqref{ANDeltaPFDeltaPf1}. On the other hand, each $\lambda_i f_i$ ($1\leq i\leq r$) is constant on $N_{\Delta\setminus\Delta_P,0}$ since $N_{\Delta\setminus\Delta_P,0}$ commutes with all the elements in $A\bs N_{\Delta_P,0}\js [F_{\Delta_P}]$ and $f_i$ is constant on $N_{\Delta\setminus\Delta_P,0}$. Using \eqref{Ftfv} we deduce $f_{v_0,t,m}=\sum_{i=1}^rv_0F_t\lambda_if_i$ since for $t\in \prod_{\alpha\in \Delta\setminus\Delta_P}t_\alpha^{\mathbb{N}}$ we have $N_{L_P,0}=tN_{L_P,0}t^{-1}H_{\Delta_P,0}$.

It is clear that $\widetilde{M'}$ is $T_0$-stable. Moreover, we have $\widetilde{M'}^\vee=A\bs N_{\Delta,0}\js\otimes_{A\bs N_{\Delta_P,0}\js}M'^\vee$. In particular, $\widetilde{M'}^\vee$ is finitely generated over $A\bs N_{\Delta,0}\js$ since $M'$ is admissible as a representation of $N_{\Delta_P,0}$.
\end{proof}
So the elements of the form $\widetilde{M'}$ ($M'\in \mathcal{M}_{\Delta_P}(\pi_P)$) are cofinal in $\mathcal{M}_\Delta(\mathcal{C}_{1,0}(\pi_P)^{H_{\Delta,0}})$. For these we have $\widetilde{M'}[1/X_\Delta]\cong A\bg N_{\Delta,0}\jg\otimes_{A\bg N_{\Delta_P,0}\jg}M'[1/X_{\Delta_P}]$. We finish the proof of Thm.\ \ref{parindDvee} by taking the projective limit.
\end{proof}

\begin{cor}\label{princserDDelta}
For a character $\chi\colon T\to A^\times$, the \'etale $T_+$-module $D^\vee_{\Delta}(\Ind_B^G\chi)\cong A\bg N_{\Delta,0}\jg\otimes_A\chi$ is free of rank $1$ on which $T_+$ acts via the twist by $\chi$.
\end{cor}

\begin{que}
Let $\pi$ be a supercuspidal representation of $G$. Is the converse of Theorem \ref{parindDvee} true in the sense that---assuming $D^\vee_\Delta(\pi)\neq 0$---there does \emph{not} exist a proper subset $I\subsetneq\Delta$ and an \'etale $T_+$-module $D$ over $A\bg N_{I,0}\jg$ such that we have 
\begin{equation*}
D^\vee_{\Delta}(\pi)\cong A\bg N_{\Delta,0}\jg\hat{\otimes}_{A\bg N_{I,0}\jg}D\ ?
\end{equation*}
\end{que}

\subsection{Exactness of $D^\vee_\Delta$ on the category $SP_A$}

We start this section with the following abstract Lemmata that will be needed several times in the sequel. 

\begin{lem}\label{finpresgenR}
Let $R$ be an arbitrary ring and $0\to M_1\to M_2\to M_3\to 0$ be a short exact sequence of $R$-modules.
\begin{enumerate}[$(i)$]
\item If $M_1$ and $M_3$ are both finitely presented then so is $M_2$.
\item If $M_3$ is finitely presented and $M_2$ is finitely generated then $M_1$ is finitely generated.
\end{enumerate}
\end{lem}
\begin{proof}
This should be classical, but here is a proof: $(i)$ We have a presentation $\bigoplus_{j=1}^{r_l} Rd_j^{(l)}\overset{f_l}{\to} \bigoplus_{i=1}^{k_l} Re_i^{(l)}\to M_l\to 0$ sending the free generator $e_i^{(l)}$ to some elements $m_i^{(l)}\in M_l$ ($1\leq i\leq k_l$, $l=1,3$) by assumption. We regard $M_1$ as a submodule of $M_2$ and lift the elements $m_i^{(3)}$ to $\widetilde{m_i^{(3)}}\in M_2$ ($1\leq i\leq k_3$). So we have a surjective map
\begin{eqnarray*}
\Psi\colon \bigoplus_{i=1}^{k_1} Re_i^{(1)}\oplus \bigoplus_{i'=1}^{k_3} Re_{i'}^{(3)}&\twoheadrightarrow& M_2\\
e_i^{(1)}&\mapsto& m_i^{(1)}\\
e_{i'}^{(3)}&\mapsto& \widetilde{m_{i'}^{(3)}}\ .\\
\end{eqnarray*}
Moreover, $\Ker(\Psi)\cap \bigoplus_{i=1}^{k_1} Re_i^{(1)}=\mathrm{Im}(f_1)$ and $\Ker(\Psi)/(\Ker(\Psi)\cap \bigoplus_{i=1}^{k_1} Re_i^{(1)})=\mathrm{Im}(f_3)$ are finitely generated whence so is $\Ker(\Psi)$.

$(ii)$ Let $U\subseteq M_2$ be a finite set of generators containing a set $V\subseteq U$ of lifts of the generators of $M_3$ arising in the finite presentation of $M_3$. Since the image $\overline{V}$ of $V$ in $M_3$ generates $M_3$, we may assume without loss of generality that $U\setminus V$ is contained in $M_1$. The natural surjection $\bigoplus_{\overline{v}\in\overline{V}}Re_{\overline{v}}\to M_3$ sending the formal generator $e_{\overline{v}}$ to $\overline{v}$ has finitely generated kernel $M_0$. Moreover, $\sum_{v\in V}\lambda v$ lies in $M_1$ if and only if $\sum_{v\in V}\lambda_v e_{\overline{v}}$ lies in $M_0$. This shows that $\sum_{v\in V}Rv\cap M_1$ is a quotient of $M_0$---in particular, finitely generated. Finally, $M_1$ is generated by $U\setminus V$ and $\sum_{v\in V}Rv\cap M_1$ therefore it is finitely generated.
\end{proof}

\begin{lem}\label{polfinpres}
Let $R$ be an arbitrary ring and $R\leq S=R[s_1,\dots,s_k]$ be a ring extension that is finitely generated as an $R$-algebra by elements $s_1,\dots,s_k\in S$ commuting with each other and satisfying $s_iR\subseteq Rs_i$. If $M$ is an $S$-module that is finitely presented as an $R$-module then $M$ is also finitely presented as an $S$-module.
\end{lem}
\begin{proof}
Choose a finite presention $f\colon F:=\bigoplus_{j=1}^n Re_j\twoheadrightarrow M$ sending $e_j$ to some element $m_j\in M$ ($1\leq j\leq n$). Consider the map $\id\otimes f\colon  S\otimes_R F=\bigoplus_{j=1}^n Se_j\twoheadrightarrow M$ sending $e_j$ to $m_j$. Since $f$ is surjective, for each pair $(i,j)$ ($1\leq i\leq k$, $1\leq j\leq n$) there exists an element $r_{i,j}\in F$ such that $f(r_{i,j}) =(\id\otimes f)(s_ie_j)$.  Let $I$ be the $S$-submodule of $S\otimes_R F$ generated by $S\otimes_R \Ker(f)$ and $\{s_ie_j-1\otimes r_{i,j}\mid 1\leq i\leq k,1\leq j\leq n\}$. It is clear that $I\subseteq \Ker(\id\otimes f)$. On the other hand, let $\sum_{j=1}^nt_je_j$ be in $\Ker(\id\otimes f)$ ($t_j\in S$, $1\leq j\leq n$). We may write each $t_j$ as a polynomial in the $s_i$'s with coefficients in $R$ ($1\leq i\leq k$, $1\leq j\leq n$) such that the $s_i$'s are on the right. Using the relations $s_ie_j-1\otimes r_{i,j}$ we find that $\sum_{j=1}^nt_je_j$ is congruent to an element in $1\otimes F$ modulo $I$. So we deduce that $\sum_{j=1}^nt_je_j$ lies in $I$ so that $I=\Ker(\id\otimes f)$ is finitely generated over $S$.
\end{proof}

For a subset $\Delta'\subset \Delta$ we introduce the following notations $N_{\Delta',0}:=\prod_{\alpha\in\Delta'}N_{\alpha,0}$ (as a direct summand of the group $N_{\Delta,0}$) and $[F_{\Delta'}]:=[F_{\alpha}\mid \alpha\in\Delta']$.

\begin{lem}\label{abstractmod}
Let $\Delta_0$ and $\Delta'$ be subsets in $\Delta$ (with $\Delta_1:=\Delta\setminus\Delta_0$ and $\Delta'':=\Delta\setminus\Delta'$) and let $N_{\Delta_1,\ast}$ be an open subgroup of $N_{\Delta_1,0}$ such that we have $tN_{\Delta_1,\ast}t^{-1}\subset N_{\Delta_1,\ast}$ for all $t\in\prod_{\alpha\in\Delta'}t_\alpha^{\mathbb{N}}$. Further, let $\chi\colon T\to \kappa^\times$ be a continuous character and consider the space $\mathcal{C}(N_{\Delta,0},\kappa)^{N_{\Delta_1,\ast}}$ of continuous functions as an $N_{\Delta,0}$-representation with the following action of the operators $F_\alpha$ ($\alpha\in\Delta$): For $\alpha\in\Delta'$ we let $F_\alpha$ act by zero. For $\alpha\in \Delta''$, $v\in N_{\Delta,0}$ and $f\in \mathcal{C}(N_{\Delta,0},\kappa)^{N_{\Delta_1,\ast}}$ we put 
\begin{equation}\label{abstractFalpha}
F_{\alpha}(f)(v):=\sum_{u\in J((N_{\Delta,0}\cap t_\alpha^{-1}N_{\Delta_1,\ast}t_\alpha)/ N_{\Delta_1,\ast})}\chi(t_\alpha)f(t_\alpha^{-1}vt_\alpha u)\ . 
\end{equation}
Then any finitely generated $\kappa\bs N_{\Delta,0}\js[F_\Delta]$-submodule $M\leq \mathcal{C}(N_{\Delta,0},\kappa)^{N_{\Delta_1,\ast}}$ is finitely presented and has finite length. Moreover, if we further assume $\Delta'\subseteq\Delta_1$ then the $\kappa\bs N_{\Delta,0}\js[F_\Delta]$-module $\mathcal{C}(N_{\Delta,0},\kappa)^{N_{\Delta_1,\ast}}$ is cyclic. If $\Delta'\subseteq \Delta_1$ and $N_{\Delta_1,\ast}=N_{\Delta_1,0}$ then $\mathcal{C}(N_{\Delta,0},\kappa)^{N_{\Delta_1,\ast}}$ is simple (ie.\ has no nontrivial submodules).
\end{lem}
\begin{proof}
\emph{Step 1: We assume $\Delta'\subseteq \Delta_1$ and show in this case that $\mathcal{C}(N_{\Delta,0},\kappa)^{N_{\Delta_1,\ast}}$ is generated by the characteristic function $\mathbbm{1}:=\mathbbm{1}_{N_{\Delta_1,\ast}\times N_{\Delta_0,0}}$.}  Note that any function $f$ in $\mathcal{C}(N_{\Delta,0},\kappa)^{N_{\Delta_1,\ast}}$ is constant on the right cosets of $N_{\Delta_1,\ast}\times s_0^rN_{\Delta_0,0}s_0^{-r}$ for $r$ large enough depending on $f$ (here we put $s_0:=\prod_{\alpha\in\Delta_0}t_\alpha$). Moreover, the function $\chi(s_0^{-r})uF_{s_0^r}(\mathbbm{1})$ is the characteristic function of the right coset $uN_{\Delta_1,\ast}\times s_0^rN_{\Delta_0,0}s_0^{-r}$ therefore $f$ can be written as a finite $\kappa$-linear combination of elements of the form $uF_{s_0^r}(\mathbbm{1})$. 

\emph{Step 2: We assume $\Delta'\subseteq\Delta_1$ and $N_{\Delta_1,\ast}=N_{\Delta_1,0}$ and show that $\mathcal{C}(N_{\Delta,0},\kappa)^{N_{\Delta_1,\ast}}$ is simple.} Let $0\neq M\leq \mathcal{C}(N_{\Delta,0},\kappa)^{N_{\Delta_1,\ast}}$ be an arbitrary submodule.  Then $M$ contains a nonzero fixed point under the action of the pro-$p$ group $N_{\Delta,0}$ therefore also the function $\mathbbm{1}_{N_{\Delta,0}}$. By our assumption $\mathbbm{1}_{N_{\Delta,0}}=\mathbbm{1}_{N_{\Delta_1,\ast}\times N_{\Delta_0,0}}$ generates $\mathcal{C}(N_{\Delta,0},\kappa)^{N_{\Delta_1,\ast}}$ so we have $M= \mathcal{C}(N_{\Delta,0},\kappa)^{N_{\Delta_1,\ast}}$.

\emph{Step 3: We still asume $\Delta'\subseteq\Delta_1$ and $N_{\Delta_1,\ast}=N_{\Delta_1,0}$. We claim that the annihilator left-ideal $I:=\mathrm{Ann}(\mathbbm{1}_{N_{\Delta,0}})\lhd \kappa\bs  N_{\Delta,0}\js[F_\Delta]$ of $\mathbbm{1}_{N_{\Delta,0}}$ is generated by the finite set $$U:=\{X_\alpha, X_{\alpha'}, \chi(t_\alpha)-\sum_{u\in J(N_{\Delta_0,0}/t_\alpha N_{\Delta_0,0}t_\alpha^{-1})}uF_\alpha,F_{\alpha'}\mid \alpha'\in\Delta',\alpha\in\Delta''\}\ .$$}
The containment $U\subset I$ is clear. Let $\sum_{i=1}^r\lambda_i F_{t_i}\in I$ be arbitrary ($\lambda_i\in \kappa\bs  N_{\Delta,0}\js$, $t_i\in \prod_{\alpha\in \Delta}t_\alpha^{\mathbb{N}}$, $i=1,\dots,r$). We may omit the terms $\lambda_iF_{t_i}$ with $t_i$ divisible by $t_{\alpha'}$ for some $\alpha'\in\Delta'$ as those are contained in the left ideal generated by $U$. Note that for a common multiple $t\in \prod_{\alpha\in \Delta''}t_\alpha^{\mathbb{N}}$ there exists a $\lambda\in \kappa\bs  N_{\Delta,0}\js$ such that we have $$\sum_{i=1}^r\lambda_i F_{t_i}\equiv \lambda F_t\pmod{\sum_{m\in U}\kappa\bs  N_{\Delta,0}\js[F_\Delta]m}\ .$$
Indeed, if $t=t_it_i'$ then we have $$F_{t_i}\equiv \chi^w(t_i'^{-1})\sum_{u\in J(N_{\Delta_0,0}/t_i' N_{\Delta_0,0}t_i'^{-1})}t_iut_i^{-1}F_t\pmod{\sum_{m\in U}\kappa\bs  N_{\Delta,0}\js[F_\Delta]m}\ .$$ 
Further, for $\alpha\in \Delta_1\cap\Delta''$ we have $t_\alpha N_{\Delta_0,0}t_\alpha^{-1}=N_{\Delta_0,0}$ whence $\chi(t_\alpha)-F_\alpha$ lies in $U$. Therefore we may assume without loss of generality that $t$ lies in $\prod_{\alpha\in\Delta_0}t_\alpha^{\mathbb{N}}$ whence $J(N_{\Delta_0,0}/t N_{\Delta_0,0}t^{-1})$ is a set of representatives for the cosets $N_{\Delta,0}/tN_{\Delta,0}t^{-1}$, too. Therefore we may write $\lambda=\sum_{u\in J(N_{\Delta_0,0}/t N_{\Delta_0,0}t^{-1})}u\varphi_t(\lambda_{u,t})$ with $\lambda_{u,t}\in \kappa\bs  N_{\Delta,0}\js$ whence we deduce
\begin{align*}
\lambda F_t=\sum_{u\in J(N_{\Delta_0,0}/t N_{\Delta_0,0}t^{-1})}u\varphi_t(\lambda_{u,t})F_t=\sum_{u\in J(N_{\Delta_0,0}/t N_{\Delta_0,0}t^{-1})}uF_t \lambda_{u,t}\equiv\\
\equiv \sum_{u\in  J(N_{\Delta_0,0}/t N_{\Delta_0,0}t^{-1})}c_{u,t}uF_t \pmod{\sum_{m\in U}\kappa\bs  N_{\Delta,0}\js[F_\Delta]m}
\end{align*}
where $c_{u,t}\in\kappa$ is the constant term of $\lambda_{u,t}$ ($u\in J(N_{\Delta_0,0}/t N_{\Delta_0,0}t^{-1})$). Now the function $0=\sum_{u\in  J(N_{\Delta_0,0}/t N_{\Delta_0,0}t^{-1})}c_{u,t}uF_t(\mathbbm{1}_{N_{\Delta_0,0}})$ is constant $c_{u,t}$ on the coset $utN_{\Delta_0,0}t^{-1}$ implying $c_{u,t}=0$ for each $u\in  J(N_{\Delta_0,0}/t N_{\Delta_0,0}t^{-1})$. We obtain $I=\sum_{m\in U}\kappa\bs  N_{\Delta,0}\js[F_\Delta]m$ as claimed.

\emph{Step 4: We assume $N_{\Delta_1,\ast}=N_{\Delta_1,0}$, but drop the assumption that $\Delta'\subseteq \Delta_1$.} Choose a finite set $U$ of generators of $M$ as a module over $\kappa\bs N_{\Delta,0}\js [F_\Delta]$. By assumption $F_{\alpha'}$ ($\alpha'\in\Delta'$) acts trivially on $M$. On the other hand, $F_\alpha$ acts by $\chi(t_\alpha)$ for each $\alpha\in \Delta_1\cap \Delta''$. Therefore $U$ generates $M$ as a module over $\kappa\bs N_{\Delta,0}\js [F_{\Delta_0\cap \Delta''}]$, too, and by Lemma \ref{polfinpres} it suffices to show that it is finitely presented and has finite length as such. Moreover, $N_{\Delta_1,0}$ also acts trivially on $M$ and lies in the centre of the ring $\kappa\bs N_{\Delta,0}\js [F_{\Delta_0\cap \Delta''}]$. So $U$ generates $M$ as a module over $\kappa\bs N_{\Delta_0,0}\js [F_{\Delta_0\cap \Delta''}]$ and by Lemma \ref{polfinpres} it suffices to show that it is finitely presented and has finite length as such. There exists a subgroup $N_{\Delta_0\cap\Delta',\ast}\leq N_{\Delta_0\cap \Delta',0}$ stabilizing all the elements in $U$. On the other hand, the subalgebra $\kappa\bs  N_{\Delta_0\cap \Delta',0}\js$ lies in the centre of $\kappa\bs  N_{\Delta_0,0}\js [F_{\Delta''}]$ therefore $N_{\Delta_0\cap \Delta',\ast}$ acts trivially on the whole $M$. Hence the finite orbit $U_1=N_{\Delta_0\cap \Delta',0}U$ of $U$ generates $M$ as a module over $\kappa\bs  N_{\Delta_0\cap\Delta''}\js [F_{\Delta_0\cap\Delta''}]$ and by Lemma \ref{polfinpres} we are reduced to showing that $M$ is finitely presented and has finite length as such. The elements in $M$ may be regarded as functions on $N_{\Delta_0,0}/N_{\Delta_0\cap \Delta',\ast}=N_{\Delta_0\cap\Delta'',0}\times(N_{\Delta_0\cap \Delta',0}/N_{\Delta_0\cap \Delta',\ast})$. Therefore $M$ is a submodule of $$\bigoplus_{v\in N_{\Delta_0\cap \Delta',0}/N_{\Delta_0\cap \Delta',\ast}}\mathcal{C}(N_{\Delta_0\cap\Delta'',0},\kappa)\ .$$
Each direct summand above is simple and finitely presented as a module over $\kappa\bs  N_{\Delta_0\cap\Delta''}\js [F_{\Delta_0\cap\Delta''}]$ by Steps 2 and 3 therefore $M$ is also finitely presented by Lemma \ref{finpresgenR} and has finite length.

\emph{Step 5: no assumptions.} We write $N_{\Delta,0}/N_{\Delta_1,\ast}$ as a direct product of $N_{\Delta_0,0}$ and the finite $p$-group $N_{\Delta_1,0}/N_{\Delta_1,\ast}$ and let $J$ be the Jacobson radical of the group ring $\kappa[N_{\Delta_1,0}/N_{\Delta_1,\ast}]$. Note that $M^i:=J^i\mathcal{C}(N_{\Delta,0},\kappa)^{N_{\Delta_1,\ast}}\cap M$ is a $\kappa\bs N_{\Delta,0}\js [F_\Delta]$-submodule of $M$ for all $i\geq 0$. By Lemma \ref{finpresgenR} it suffices to show that each graded piece $M^i/M^{i+1}$ is finitely presented and has finite length ($i\geq 0$). $M/M^1$ is a finitely generated submodule of 
\begin{equation*}
\mathcal{C}/\mathcal{C}^1:=\mathcal{C}(N_{\Delta,0},\kappa)^{N_{\Delta_1,\ast}}/J\mathcal{C}(N_{\Delta,0},\kappa)^{N_{\Delta_1,\ast}}\cong \mathcal{C}(N_{\Delta_0,0},\kappa)
\end{equation*}
on which $\kappa\bs N_{\Delta,0}\js [F_\Delta]$ acts via its quotient $\kappa\bs N_{\Delta_0,0}\js [F_\Delta]$. Moreover, all $\alpha'\in\Delta'$ act by zero on $\mathcal{C}/\mathcal{C}^1$. On the other hand, the classes of functions $f\in\mathcal{C}$ supported on $N_{\Delta_1,\ast}\times N_{\Delta_0,0}$ generate $\mathcal{C}/\mathcal{C}^1$. On such an $f$ each $F_\alpha$ ($\alpha\in \Delta''$) acts by the formula $F_\alpha(f)(v)=\chi(t_\alpha)f(t_\alpha^{-1}vt_\alpha)$ since in the sum \eqref{abstractFalpha} all the other terms vanish. Therefore $M/M^1$ is finitely presented and has finite length by Step 4. In particular, $M^1$ is finitely generated by Lemma \ref{finpresgenR}.

Now for $i\geq 1$ we have an identification
\begin{equation*}
\mathcal{C}^i/\mathcal{C}^{i+1}:=J^i\mathcal{C}(N_{\Delta,0},\kappa)^{N_{\Delta_1,\ast}}/J^{i+1}\mathcal{C}(N_{\Delta,0},\kappa)^{N_{\Delta_1,\ast}}\cong J^i/J^{i+1}\otimes_\kappa\mathcal{C}/\mathcal{C}^1\ .
\end{equation*}
Now $J^i/J^{i+1}$ is generated over $\kappa$ by the elements $\prod_{\alpha\in\Delta_1}(n_\alpha-1)^{k_\alpha}$ with $\sum_{\alpha\in\Delta_1}k_\alpha=i$ and $n_\alpha$ topological generator in $N_{\alpha,0}$. On a generator $\prod_{\alpha\in\Delta_1}(n_\alpha-1)^{k_\alpha}$ the operator $F_\alpha$ ($\alpha\in\Delta_1$) acts by $0$ if $k_\alpha\neq 0$ otherwise by $1$. Using Step 4 we deduce by induction on $i$ that $M^i/M^{i+1}$ is finitely presented and of finite length (whence $M^{i+1}$ is finitely generated). The statement follows using Lemma \ref{finpresgenR}.
\end{proof}

Following \cite{B} we denote by $SP_A$ the category of smooth $G$-representations over $A$ consisting of finite length representations whose Jordan-H\"older constituents are subquotients of principal series. Let $\chi\colon T\to A^\times$ be a continuous character. The principal series representation $\Ind_B^G\chi$ admits a filtration by $B$-subrepresentations whose graded pieces are $\mathcal{C}_w(\chi)=c-\Ind_{B^-}^{B^-wN}\chi\cong c-\Ind_{w^{-1}B^-w\cap N}^{N}\chi^w$ for $w\in W=N_G(T)/T$ where $\chi^w$ is the character given by the formula $\chi^w(t)=\chi(wtw^{-1})$. We set $\mathcal{C}_{w,0}(\chi):=\Ind_{B^-}^{B^-wN_0}\chi$ which is a generating $B_+$-subrepresentation in $\mathcal{C}_w(\chi)$.

Assume now that $A=\kappa$. We introduce the filtration (indexed by $t\in\prod_{\alpha\in\Delta}t_\alpha^{\mathbb{N}}$) 
\begin{equation*}
\Fil^t:=\Fil^t\mathcal{C}_w(\chi):=\Ind_{w^{-1}B^-w\cap N}^{(w^{-1}B^-w\cap N)t^{-1}N_0t}\chi^w
\end{equation*}
by $B_+$-subrepresentations. We have $\Fil^1=\mathcal{C}_{w,0}(\chi)$. As a representation of $N_0$, $\Fil^t$ can be written as a direct sum 
\begin{align*}
\Fil^t=\bigoplus_{N_{t,v}\in D_{t,w}}\mathcal{C}_{t,v}\ .
\end{align*}
where $N_{t,v}:=(w^{-1}B^-w\cap t^{-1}N_0t)vN_{0}$ is the double coset of $v\in t^{-1}N_0t$, $D_{t,w}:=(w^{-1}B^-w\cap t^{-1}N_0t) \backslash (t^{-1}N_{0}t)/N_{0}$ is the set of double cosets, and $\mathcal{C}_{t,v}:=\Ind_{w^{-1}B^-w\cap N}^{N_{t,v}}\chi^w$. Now each $t_\ast\in \prod_{\alpha\in\Delta}t_\alpha^{\mathbb{N}}$ acts on the set $D_{t,w}$ via conjugation: $t_\ast\cdot N_{t,v}:=N_{t,t_\ast vt_\ast^{-1}}$. Indeed, this definition does not depend on the choice $v$ of the representative in $N_{t,v}$ since we have $t_\ast N_0t\ast^{-1}\subseteq N_0$ and $t_\ast (w^{-1}B^-w)t_\ast^{-1}=w^{-1}B^-w$.

\begin{lem}\label{reprcoset}
Assume that we have $t_\ast\cdot N_{t,v}=N_{t,v}$ for some $t,t_\ast\in\prod_{\alpha\in\Delta}t_\alpha^{\mathbb{N}}$ and $v\in t^{-1}N_0t$. Then there exists a representative $v_\ast$ in the double coset $N_{t,v}$ such that $t_\ast v_\ast t_\ast^{-1}=v_\ast$. Moreover, we have $t_\alpha\cdot N_{t,v}=N_{t,v}$ for each $\alpha\in\Delta$ with $\alpha(t_\ast)\neq 1$.
\end{lem}
\begin{proof}
We have $N_{t,v}=N_{t,t_\ast vt_\ast^{-1}}$. By induction on $i$ we find that $N_{t,v}=N_{t,t_\ast^i vt_\ast^{-i}}$ for all $i\geq 1$. Since $\beta(t_\ast)$ is a nonnegative power of $p$ for each $\beta\in \Phi^+$, the limit $v_\ast:=\lim_i t_\ast^ivt_\ast^{-i}$ exists in $N$ and satisfies $t_\ast v_\ast t_\ast^{-1}=v_\ast$ and $N_{t,v}=N_{t,v_\ast}$. Now we can write $v_\ast$ uniquely in the form $v_\ast=\prod_{\beta\in\Phi^+}n_\beta$ with $n_\beta\in N_\beta$. So we have $v_\ast=t_\ast v_\ast t_\ast^{-1}=\prod_{\beta\in\Phi^+}n_\beta^{\beta(t_\ast)}$. So for each $\beta\in\Phi^+$ we have $n_\beta=1$ or $\beta(t_\ast)=1$. Now if $\alpha\Delta$ is a simple root with $\alpha(t_\ast)\neq 1$ then we have $n_\beta=1$ for all $\beta\in\Phi^+$ with $\beta\circ\lambda_{\alpha^\vee}\neq 1$, or equivalently, for all $\beta\in\Phi^+$ with $\beta(t_\alpha)\neq 1$. In particular, we deduce $t_\alpha v_\ast t_\alpha^{-1}=v_\ast$ whence $t_\alpha\cdot N_{t,v}=N_{t,v}$.
\end{proof}

We order the set $D_{t,w}$ partially by putting $N_{t,v_1}\leq N_{t,v_2}$ if there exists an element $t'\in \prod_{\alpha\in\Delta}t_\alpha^{\mathbb{N}}$ such that $t'\cdot N_{t,v_2}=N_{t,v_1}$. 
\begin{lem}\label{orderDtw}
The ordering $\leq$ on $D_{t,w}$ is transitive, reflexive, and antisymmetric.
\end{lem}
\begin{proof}
The transitivity and reflexivity are clear ($\prod_{\alpha\in\Delta}t_\alpha^{\mathbb{N}}$ is a monoid with $1\in \prod_{\alpha\in\Delta}t_\alpha^{\mathbb{N}}$). Assume now that $t'\cdot N_{t,v_1}=N_{t,v_2}$ and $t''\cdot N_{t,v_2}=N_{t,v_1}$. Put $t_\ast:=t't''$. By Lemma \ref{reprcoset} we deduce $N_{t,v_1}=t''\cdot N_{t,v_2}=N_{t,v_2}$ as $t''$ can be written as a product of $t_\alpha$'s with $\alpha(t_\ast)\neq 1$. 
\end{proof}

So we can refine the ordering $\leq$ on $D_{t,w}$ to a total ordering of $D_{t,w}$ giving a filtration $$\Fil^{t,v}:=\bigoplus_{N_{t,v}\geq N_{t,v'}\in (w^{-1}B^-w\cap t^{-1}N_0t) \backslash (t^{-1}N_{0}t)/N_{0}}\mathcal{C}_{t,v'}$$ on $\Fil^t$ by finitely many $N_0\prod_{\alpha\in\Delta}t_\alpha^{\mathbb{N}}$-subrepresentations. The graded piece $$\mathrm{gr}^{t,v}(\mathcal{C}_w):=\Fil^{t,v}/\bigcup_{N_{t,v'}<N_{t,v}}\Fil^{t,v'}$$ is isomorphic to $\mathcal{C}_{t,v}$ as a representation of $N_0$.

\begin{lem}\label{finpresgr}
Any finitely generated $\kappa\bs N_{\Delta,0}\js [F_\Delta]$-submodule of $\mathrm{gr}^{t,v}(\mathcal{C}_{w})^{H_{\Delta,0}}$ is finitely presented.
\end{lem}
\begin{proof}
Put $\Delta_1:=\Delta\cap w^{-1}(\Phi^-)$ and $\Delta_0:=\Delta\setminus w^{-1}(\Phi^-)$. $H_{\Delta,0}$ is normal in $N_0$ with commutative quotient $N_{\Delta,0}$ such that the image of $$N_0(t,v):=v^{-1}(w^{-1}B^-w\cap t^{-1}N_0t)v\cap N_0$$ in $N_{\Delta,0}$ is an open subgroup $N_{\Delta_1,\ast}$ in $N_{\Delta_1,0}$. Hence we have an identification 
\begin{align}
\mathrm{gr}^{t,v}(\mathcal{C}_{w})^{H_{\Delta,0}}\cong \Ind_{w^{-1}B^-w\cap t^{-1}N_0t}^{(w^{-1}B^-w\cap t^{-1}N_0t)vN_{0}/H_{\Delta,0}}\chi^w
\cong\mathcal{C}(N_{\Delta,0},\kappa)^{N_{\Delta_1,\ast}}\label{grtvCw}
\end{align}
as a representation of $N_{\Delta,0}$.

Now we describe the action of each $F_\alpha$ ($\alpha\in\Delta$) on $\mathrm{gr}^{t,v}(\mathcal{C}_{w})^{H_{\Delta,0}}$. Let $\alpha\in\Delta$ be arbitrary. If we have $N_{t,v}> N_{t,t_\alpha vt_\alpha^{-1}}$ then $F_\alpha$ acts by $0$ on $\mathrm{gr}^{t,v}(\mathcal{C}_{w})^{H_{\Delta,0}}$ since even $t_\alpha$ acts by $0$ on $\mathrm{gr}^{t,v}(\mathcal{C}_{w})$. On the other hand, if $N_{t,v}=N_{t,t_\alpha vt_\alpha^{-1}}$ then by Lemma \ref{reprcoset} we may assume without loss of generality that the representative $v$ of the double coset $(w^{-1}B^-w\cap t^{-1}N_0t)vN_{0}$ is chosen so that we have $t_\alpha vt_\alpha^{-1}=v$. Therefore we compute
\begin{align}
F_\alpha(f)(vv_1)=\sum_{u\in J(H_{\Delta,0}/t_\alpha H_{\Delta,0}t_\alpha^{-1})}(t_\alpha f)(vv_1u)=\notag\\
=\sum_{u\in J(H_{\Delta,0}/t_\alpha H_{\Delta,0}t_\alpha^{-1})}\chi^w(t_\alpha)f(t_\alpha^{-1}vv_1ut_\alpha)=\sum_{u\in J(H_{\Delta,0}/t_\alpha H_{\Delta,0}t_\alpha^{-1})}\chi^w(t_\alpha)f(vt_\alpha^{-1}v_1ut_\alpha)\label{Falphafv}
\end{align}
for any $v_1\in N_0$. Now if $v_1$ does not lie in $N_0(t,v)t_\alpha N_0t_\alpha^{-1}H_{\Delta,0}$ then all the terms in \eqref{Falphafv} are zero whence $F_\alpha(f)(v_1)=0$. However, if $v_1=n_0t_\alpha v_2t_\alpha^{-1}u_0$ for some $n_0\in N_0(t,v)$, $v_2\in N_0$, $u_0\in H_{\Delta,0}$ then we may replace $J(H_{\Delta,0}/t_\alpha H_{\Delta,0}t_\alpha^{-1})$ with $J':=u_0 J(H_{\Delta,0}/t_\alpha H_{\Delta,0}t_\alpha^{-1})$ and assume without loss of generality that $u_0=1$. Moreover, since $t_\alpha v_2t_\alpha^{-1}$ normalizes both $H_{\Delta,0}$ and $t_\alpha H_{\Delta,0}t_\alpha^{-1}$, $J'':=t_\alpha v_2t_\alpha^{-1} J't_\alpha v_2^{-1}t_\alpha^{-1}$ is also a set of representatives of $H_{\Delta,0}/t_\alpha H_{\Delta,0}t_\alpha^{-1}$. Therefore we compute 
\begin{align*}
F_\alpha(f)(vv_1)=\sum_{u\in J(H_{\Delta,0}/t_\alpha H_{\Delta,0}t_\alpha^{-1})}\chi^w(t_\alpha)f(vt_\alpha^{-1}v_1ut_\alpha)=\\
=\sum_{u\in J(H_{\Delta,0}/t_\alpha H_{\Delta,0}t_\alpha^{-1})}\chi^w(t_\alpha)f(vt_\alpha^{-1}n_0t_\alpha v_2t_\alpha^{-1}u_0ut_\alpha)=\\
=\sum_{u''\in J''}\chi^w(t_\alpha)f(vt_\alpha^{-1}n_0u''t_\alpha v_2)=\sum_{\substack{v_2\in N_0(t,v)\backslash N_0/H_{\Delta,0}\\ v_1\in N_0(t,v)t_\alpha v_2t_\alpha^{-1}H_{\Delta,0}}}\sum_{u''\in N_0(t,v)t_\alpha H_{\Delta,0}t_\alpha^{-1}\cap J''}\chi^w(t_\alpha)f(vv_2)\\
=\sum_{\substack{v_2\in N_0(t,v)\backslash N_0/H_{\Delta,0}\\ v_1\in N_0(t,v)t_\alpha v_2t_\alpha^{-1}H_{\Delta,0}}}|N_0(t,v)\cap H_{\Delta,0}:N_0(t,v)\cap t_\alpha H_{\Delta,0}t_\alpha^{-1}|\chi^w(t_\alpha)f(vv_2)
\end{align*}
since the double coset $(N_0(t,v)\cap H_{\Delta,0})1t_\alpha H_{\Delta,0}t_\alpha^{-1}$ contains $|N_0(t,v)\cap H_{\Delta,0}:N_0(t,v)\cap t_\alpha H_{\Delta,0}t_\alpha^{-1}|$ left cosets in $H_{\Delta,0}/t_\alpha H_{\Delta,0}t_\alpha^{-1}$.  Assume first $N_0(t,v)\cap H_{\Delta,0}\not\subseteq t_\alpha H_{\Delta,0}t_\alpha^{-1}$ whence this number is divisible by $p$. In this case $F_\alpha(f)=0$ for all $f\in\mathcal{C}_{w,0}(\chi)^{H_{\Delta,0}}$ since $p=0$ in $\kappa$. So we deduce that $F_\alpha$ acts by $0$. On the other hand, if $N_0(t,v)\cap H_{\Delta,0}\subseteq t_\alpha H_{\Delta,0}t_\alpha^{-1}$ and $\widetilde{f}\colon N_{\Delta,0}\to\kappa$ denotes the function corresponding to $f$ under the identification \eqref{grtvCw} then we compute 
\begin{align*}
\widetilde{F_\alpha(f)}(\overline{v_1})=\sum_{\substack{v_2\in N_0(t,v)\backslash N_0/H_{\Delta,0}\\ v_1\in N_0(t,v)t_\alpha v_2t_\alpha^{-1}H_{\Delta,0}}}\chi^w(t_\alpha)f(vv_2)=\sum_{u\in J((N_{\Delta,0}\cap t_\alpha^{-1}N_{\Delta_1,\ast}t_\alpha)/ N_{\Delta_1,\ast})}\chi(t_\alpha)\widetilde{f}(t_\alpha^{-1}\overline{v_1}t_\alpha u)
\end{align*}
for all $\overline{v_1}:=v_1H_{\Delta,0}\in N_{\Delta,0}$.

We denote by $\Delta'\subseteq \Delta$ the subset of those $\alpha$ simple roots satisfying $N_{t,v}> N_{t,t_\alpha vt_\alpha^{-1}}$ or $N_0(t,v)\cap H_{\Delta,0}\not\subseteq t_\alpha H_{\Delta,0}t_\alpha^{-1}$ so that $F_\alpha$ acts by $0$ on $\mathrm{gr}^{t,v}(\mathcal{C}_{w})^{H_{\Delta,0}}$. Then the assumptions of Lemma \ref{abstractmod} are satisfied and the statement follows.
\end{proof}

\begin{rem}
Let $\chi\colon T\to \kappa^\times$ be a character and $w\in W$ such that for all roots $\beta\in (\Phi^+\setminus\Delta)\cap w^{-1}(\Phi^-)$ and $\alpha\in \Delta\setminus w^{-1}(\Phi^-)$ we have $\beta\circ\alpha^\vee=1$. Then the $\kappa\bs N_{\Delta,0}\js[F_\Delta]$-module $\mathcal{C}_{w,0}(\chi)^{H_{\Delta,0}}$ is finitely presented and simple. Otherwise the module $\mathcal{C}_{w,0}(\chi)^{H_{\Delta,0}}$ is not even finitely generated over $\kappa\bs N_{\Delta,0}\js[F_\Delta]$.
\end{rem}

\begin{pro}\label{finpresCw}
Let $\chi\colon T\to \kappa^\times$ be a character. Then any finitely generated $\kappa\bs N_{\Delta,0}\js[F_\Delta]$-submodule $M$ of $\mathcal{C}_{w}(\chi)^{H_{\Delta,0}}$ is finitely presented, has finite length, and is admissible as a representation of $N_{\Delta,0}$ for all $w\in W$.
\end{pro}
\begin{proof}
Any finitely generated $\kappa\bs N_{\Delta,0}\js[F_\Delta]$-submodule $M$ of $\mathcal{C}_{w}(\chi)^{H_{\Delta,0}}$ is contained in $\Fil^{t,v}$ for some $t\in \prod_{\alpha\in\Delta}t_\alpha^{\mathbb{N}}$ and $v\in N_0$. We prove the statement by induction on the number $d$ of elements in $D_{t,w}$ that are smaller than $N_{t,v}$. If $d=0$ then we have $v=1$ and $\Fil^{t,1}=\mathrm{gr}^{t,1}(\mathcal{C}_w)=\mathcal{C}_{w,0}(\chi)$ whence the statement follows from Lemma \ref{finpresgr}. So assume $d>0$. The image $\overline{M}$ of $M$ in $\mathrm{gr}^{t,v}(\mathcal{C}_w)=\Fil^{t,v}/\Fil^{t,v'}$ is finitely presented and has finite length by Lemma \ref{finpresgr} (here $N_{t,v'}$ is the biggest element in $D_{t,w}$ that is strictly smaller than $N_{t,v}$). Hence by Lemma \ref{finpresgenR}(ii) $M\cap \Fil^{t,v'}$ is finitely generated and the statement follows by induction on $d$ using Lemma \ref{finpresgenR}(i).

The admissibility of $M$ follows from the fact that $\Fil^t$ is admissible as a representation of $N_0$ for each $t\in \prod_{\alpha\in\Delta}t_\alpha^{\mathbb{N}}$. Indeed, $(\Fil^t)^\vee$ is generated by at most $|t^{-1}N_0t:N_0|$ elements as a module over $\kappa\bs N_0\js$.
\end{proof}

\begin{pro}\label{finpresSPA}
Let $\pi$ be an object in $SP_A$. Then any finitely generated $A\bs N_{\Delta,0}\js[F_\Delta]$-submodule $M$ of $\pi^{H_{\Delta,0}}$ is finitely presented, has finite length, and is admissible as a representation of $N_{\Delta,0}$.
\end{pro}
\begin{proof}
This follows from Prop.\ \ref{finpresCw} and Lemma \ref{finpresgenR} noting that each irreducible subquotient of $\pi_{\mid B}$ is isomorphic to $\mathcal{C}_w(\chi)$ for some character $\chi\colon T\to \kappa^\times$ and $w\in W$ by \cite{V}.
\end{proof}

\begin{pro}\label{genlengthDveeSPA}
Let $\pi$ be an object in $SP_A$. The generic length of $D^\vee_\Delta(\pi)$ as a module over $A\bg N_{\Delta,0}\jg$ equals the number of Jordan-H\"older factors of $\pi_{\mid B}$ isomorphic to $\mathcal{C}_1(\chi)$ for some character $\chi\colon T\to\kappa^\times$. In particular, $D^\vee_\Delta(\pi)$ is finitely generated over $A\bg N_{\Delta,0}\jg$.
\end{pro}
\begin{proof}
We prove by induction on the length of $\pi_{\mid B}$ (that is finite by \cite{V}). If $\pi_{\mid B}$ is irreducible then the statement follows combining Thm.\ \ref{rightexact} and Lemma \ref{DDeltaCw} with Cor.\ \ref{princserDDelta}. By \cite{V} we have a short exact sequence $0\to \pi_1\to \pi\to \mathcal{C}_w(\chi)\to 0$ of $B$-representations for some character $\chi\colon T\to\kappa^\times$ and $w\in W$. By Thm.\ \ref{rightexact} this induces an exact sequence $D^\vee_\Delta( \mathcal{C}_w(\chi))\to D^\vee_\Delta(\pi) \to D^\vee_\Delta(\pi_1)\to 0$. If $w\neq 1$ then we have $D^\vee_\Delta( \mathcal{C}_w(\chi))=0$ by Lemma \ref{DDeltaCw} therefore there is nothing to prove. So assume $w=1$. Then $D^\vee_\Delta( \mathcal{C}_1(\chi))$ has generic length $1$ by Cor.\ \ref{princserDDelta}. Moreover,  $M:= \mathcal{C}_{1,0}(\chi)^{H_{\Delta,0}}$ is finitely presented and simple as a module over $A\bs N_{\Delta,0}\js [F_\Delta]$ by Lemma \ref{finpresgr} and the Remark thereafter (as it is generated by the single $\mathbbm{1}_{N_0}$) such that we have $M^\vee[1/X_\Delta]=D^\vee_\Delta( \mathcal{C}_1(\chi))$. 

The $A\bs X\js[F]$-submodule $M_\ell$ of $\mathcal{C}_1(\chi)^{H_{\ell,0}}$ (see section \ref{reducetousualsec} for the definition of $H_{\ell,0}$) generated by $\mathbbm{1}_{N_0}$ is the unique minimal element in $\mathcal{M}(\mathcal{C}_1(\chi)^{H_{\ell,0}})$ with $M_\ell^\vee[1/X]\neq 0$. (To see this one could either use the same argument as in the proof of Lemma \ref{abstractmod} to show that $M_\ell$ is simple or note that $D_\xi^\vee(\mathcal{C}_1(\chi)^{H_{\ell,0}})M_\ell^\vee[1/X]$ has dimension $1$ over $\kappa\bg X\jg$ and $M_\ell^\vee$ is free of rank $1$ over $\kappa\bs X\js$ so we have $M_\ell'^\vee[1/X]=0$ for any proper submodule $M_\ell'\lneq M_\ell$. On the other hand, $D_\xi^\vee$ is exact on the category $SP_A$ (Cor.\ 9.2 in \cite{B}). Therefore the natural map $M_\ell^\vee[1/X]\to D_\xi^\vee(\pi)$ is injective. So there exists an element $M_{1,\ell}\in \mathcal{M}(\pi^{H_{\ell,0}})$ whose image in $\mathcal{C}_1(\chi)^{H_{\ell,0}}$ contains $M_\ell$. In particular, $\mathbbm{1}_{N_0}$ is the image of an element $m\in \pi^{H_{\ell,0}}\subseteq \pi^{H_{\Delta,0}}$. Let us denote by $M_1$ the $A\bs N_{\Delta,0}\js[F_\Delta]$-submodule of $\pi^{H_{\Delta,0}}$ generated by the finite $T_0$-orbit of $m$. By Prop.\ \ref{finpresSPA} $M_1$ is an object in $\mathcal{M}_\Delta(\pi^{H_{\Delta,0}})$ admitting a surjective $T_0$- and $A\bs N_{\Delta,0}\js[F_\Delta]$-equivariant map $M_1\to M$. Therefore the composite map $M^\vee[1/X_\Delta]\cong D^\vee_\Delta(\mathcal{C}_1(\chi))\to D^\vee_\Delta(\pi)\twoheadrightarrow M_1^\vee[1/X_\Delta]$ is injective so we have a short exact sequence $0\to D^\vee_\Delta( \mathcal{C}_1(\chi))\to D^\vee_\Delta(\pi) \to D^\vee_\Delta(\pi_1)\to 0$.
\end{proof}

\begin{thm}\label{princserexact}
Let $0\to\pi_1\to\pi_2\to\pi_3\to0$ be an exact sequence in $SP_A$. Then the sequence $0\to D^\vee_\Delta(\pi_3)\to D_\Delta^\vee(\pi_2)\to D^\vee_\Delta(\pi_1)\to0$ is also exact.
\end{thm}
\begin{proof}
By Thm.\ \ref{rightexact} we have an exact sequence $D^\vee_\Delta(\pi_3)\overset{f_3}{\to} D_\Delta^\vee(\pi_2)\to D^\vee_\Delta(\pi_1)\to0$. By Prop.\ \ref{genlengthDveeSPA} we have $\genlength(D_\Delta^\vee(\pi_2))=\genlength(D_\Delta^\vee(\pi_1))+\genlength(D_\Delta^\vee(\pi_3))$. Hence we deduce $\genlength(\mathrm{Im}(f_3))=\genlength(D_\Delta^\vee(\pi_3))$ and $\genlength(\Ker(f_3))=0$. By Prop.\ \ref{abeliancat} $\Ker(f_3)\cap D_\Delta^\vee(\pi_3)[\varpi]$ is a finitely generated and torsion \'etale $T_+$-module over $\kappa\bg N_{\Delta,0}\jg$. Therefore its global annihilator is a $T_0$-invariant ideal that equals $\kappa\bg N_{\Delta,0}\jg$ by Prop.\ \ref{noideal}. We deduce that $f_3$ is injective as desired.
\end{proof}

\begin{cor}\label{reduceellSP}
Let $\pi$ be an object in $SP_A$. Then we have $D_\xi^\vee(\pi)\cong A\bg X\jg\otimes_{\ell, A\bg N_{\Delta,0}\jg}D_\Delta^\vee(\pi)$.
\end{cor}
\begin{proof}
We have a surjective map $D_\xi^\vee(\pi)\to A\bg X\jg\otimes_{\ell, A\bg N_{\Delta,0}\jg}D_\Delta^\vee(\pi)$ by the universal property of $D_\xi^\vee$ (see Remark 2 after Prop.\ \ref{reduceellexact}) and the two sides have the same length as a module over the artinian ring $A\bg X\jg$ by Prop. \ref{reduceellexact} and Thm.\ \ref{princserexact}.
\end{proof}

\section{Noncommutative theory}\label{noncommutative}

\subsection{The ring $A\bg N_{\Delta,\infty}\jg$ and its first ring-theoretic properties}

Let $H_{\Delta,k}$ be the normal subgroup of $N_0$ generated by $s^kH_{\Delta,0}s^{-k}$, ie.\ we put $$H_{\Delta,k}=\langle n_0s^kH_{\Delta,0}s^{-k}n_0^{-1}\mid n_0\in N_0\rangle\ .$$  $H_{\Delta,k}$ is an open subgroup of $H_{\Delta,0}$ normal in $N_0$ and we have $\bigcap_{k\geq 0} H_{\Delta,k}=\{1\}$. Put $N_{\Delta,k}:=N_0/H_{\Delta,k}$ and consider the Iwasawa algebra $A\bs N_{\Delta,k}\js$. There is a natural surjective homomorphism $N_{\Delta,k}\twoheadrightarrow N_{\Delta,0}$ with finite kernel $H_{\Delta,0}/H_{\Delta,k}$. This group homomorphism does not admit a splitting in general. However, there is a canonical splitting from the subgroup $s^kN_{\Delta,0}s^{-k}\leq N_{\Delta,0}$. Indeed, the image of the map $s^k(\cdot)s^{-k}\colon N_{\Delta,k}\to N_{\Delta,k}$ maps isomorphically onto $s^kN_{\Delta,0}s^{-k}$ since its intersection with the finite subgroup $H_{\Delta,0}/H_{\Delta,k}$ is trivial ($H_{\Delta,0}/H_{\Delta,k}$ is also the kernel of the homomorphism $s^k(\cdot)s^{-k}\colon N_{\Delta,k}\to N_{\Delta,k}$, so its image is torsion-free). Let $n_0=n_0(G)\in\mathbb{N}$ be the maximum of the degrees of the algebraic characters $\beta\circ\xi\colon \mathbf{G}_m\to\mathbf{G}_m$ for all $\beta$ in $\Phi^+$. Note that for any positive root $\beta\in\Phi^+$ we have $sN_{\beta,0}s^{-1}=N_{\beta,0}^{p^{\deg(\beta\circ\xi)}}\geq N_{\beta,0}^{p^{n_0}}$. Then $s^{kn_0}N_{\Delta,k}s^{-kn_0}$ even lies in the centre of the group $N_{\Delta,k}$: Since $N_0$ is totally decomposed, it suffices to verify that the commutator $[s^{kn_0}n_\alpha s^{-kn_0},n_\beta]$ lies in $H_{\Delta,k}$ for the generators $n_\alpha:=u_\alpha(1)\in N_{\alpha,0}$ and $n_\beta:=u_\beta(1)\in N_{\beta,0}$ with $\alpha,\beta\in\Phi^+$. Now $s^{kn_0}n_\alpha s^{-kn_0}$ lies in $N_{\alpha,0}^{p^{kn_0\alpha(s)}}\leq N_{\alpha,0}^{p^{kn_0}}$ therefore the commutator $[s^{kn_0}n_\alpha s^{-kn_0},n_\beta]$ is in $$[N_0,N_0]^{p^{kn_0}}\leq H_{\Delta,0}^{p^{kn_0}}=\langle N_{\beta,0}^{p^{kn_0}}\mid \beta\in\Phi^+\setminus\Delta\rangle\leq \prod_{\beta\in\Phi^+\setminus\Delta}s^kN_{\beta,0}s^{-k}=s^kH_{\Delta,0}s^{-k}\leq H_{\Delta,k}\ .$$ By an abuse of notation we also denote the class of $n_\alpha$ in $N_{\Delta,k}$ by $n_\alpha$. Since $n_\alpha^{p^{kn_0}}$ lies in the centre of $N_{\Delta,k}$ we may form the ring $$A \bg N_{\Delta,k}\jg :=A \bs N_{\Delta,k}\js[(n_\alpha^{p^{kn_0}}-1)^{-1}\mid \alpha\in\Delta]\ .$$

Note that $n_\alpha^{p^{kn_0}}-1$ is divisible by $n_\alpha-1$ therefore $n_\alpha-1$ is also invertible in the ring $A \bg N_{\Delta,k}\jg$. Vice versa, if $n_\alpha-1$ is invertible in a ring in which $p^h=0$ then so is $n_\alpha^{p^{kn_0}}-1$ since the latter differs from the invertible element $(n_\alpha-1)^{p^{kn_0}}$ by something divisible by $p$ which is nilpotent. So the ring $A \bg N_{\Delta,k}\jg$ does not depend on the choice of our power $n_\alpha^{p^{kn_0}}$, nor on the choice of the topological generator $n_\alpha$. Moreover, for integers $0\leq k_1\leq k_2$ we have a natural surjective homomorphism
\begin{equation*}
A \bg N_{\Delta,k_2}\jg \twoheadrightarrow A \bg N_{\Delta,k_1}\jg
\end{equation*}
induced by the group homomorphism $N_{\Delta,k_2}\twoheadrightarrow N_{\Delta,k_1}$. 
\begin{lem}\label{HhomologAbgjg}
For integers $0\leq k_1\leq k_2$ we have $A \bg N_{\Delta,k_1}\jg\cong H_0(H_{\Delta,k_1}/H_{\Delta,k_2},A \bg N_{\Delta,k_2}\jg)$ and $H_i(H_{\Delta,k_1}/H_{\Delta,k_2},A \bg N_{\Delta,k_2}\jg)=0$ for all $i>0$. On finitely generated $A \bg N_{\Delta,k_2}\jg$-modules the functors $H_0(H_{\Delta,k_1}/H_{\Delta,k_2},\cdot)$ and $A \bg N_{\Delta,k_1}\jg\otimes_{A \bg N_{\Delta,k_2}\jg}\cdot$ are naturally isomorphic.
\end{lem}
\begin{proof}
For the completed group algebras we certainly have the analogous statements. More explicitly, since $H_{\Delta,k_1}/H_{\Delta,k_2}$ is a finite group, the kernel of the quotient map $A \bs N_{\Delta,k_2}\js \twoheadrightarrow A \bs N_{\Delta,k_1}\js$ is generated as a right (or left) ideal in $A \bs N_{\Delta,k_2}\js$ by the elements $h-1$ ($h\in S$) for any fixed set $S$ of generators of the group $H_{\Delta,k_1}/H_{\Delta,k_2}=\Ker(N_{\Delta,k_2}\twoheadrightarrow N_{\Delta,k_1})$. Since localization is exact, the same is true for the kernel $I$ of the natural map $A \bg N_{\Delta,k_2}\jg \twoheadrightarrow A \bg N_{\Delta,k_1}\jg$ showing that $A \bg N_{\Delta,k_1}\jg\cong H_0(H_{\Delta,k_1}/H_{\Delta,k_2},A \bg N_{\Delta,k_2}\jg)$. Now if $M$ is a finitely generated module over $A \bg N_{\Delta,k_2}\jg$ then $H_0(H_{\Delta,k_1}/H_{\Delta,k_2},M)$ is the quotient of $M$ by the submodule generated by $(h-1)M$ ($h\in S$). We deduce $$H_0(H_{\Delta,k_1}/H_{\Delta,k_2},M)\cong M/IM\cong A \bg N_{\Delta,k_1}\jg\otimes_{A \bg N_{\Delta,k_2}\jg}M\ .$$

Finally, note that the completed group ring $A\bs N_{\Delta,k_2}\js$ is a finite free module over its subring $A\bs s^{k_2n_0}N_{\Delta,k_2}s^{-k_2n_0}\js[H_{\Delta,0}/H_{\Delta,k_2}]$. Inverting central elements $(n_\alpha^{p^{k_2n_0}}-1)$ for $\alpha\in\Delta$ we deduce that $A \bg N_{\Delta,k_2}\jg $ is finite free over $A\bg s^{k_2n_0}N_{\Delta,k_2}s^{-k_2n_0}\jg[H_{\Delta,0}/H_{\Delta,k_2}]$. In particular, it is induced as a representation of the finite group $H_{\Delta,0}/H_{\Delta,k_2}$ whence also as a representation of the subgroup $H_{\Delta,k_1}/H_{\Delta,k_2}$. Hence its higher homology groups vanish.
\end{proof}

So we may form the projective limit
\begin{equation*}
A \bg N_{\Delta,\infty}\jg:=\varprojlim_k A \bg N_{\Delta,k}\jg\ .
\end{equation*}

\begin{lem}\label{augker}
The kernel of the natural map $A \bg N_{\Delta,\infty}\jg\to A \bg N_{\Delta,0}\jg$ is generated as a left (or as a right) ideal by the elements $n_\beta-1$ for topological generators $n_\beta\in N_{\beta,0}$ for $\beta\in\Phi^+\setminus\Delta$.
\end{lem}
\begin{proof}
We prove by induction on $\dim N_0$ as a $p$-adic Lie group. Note that $N_0$ is nilpotent so unless $\Delta=\Phi^+$ (whence there is nothing to prove), we have a $\beta$ in $\Phi^+\setminus\Delta$ such that $N_{\beta,0}$ lies in the centre of $N_0$. Then we may form the ring $A \bg N_{\Delta,\infty,\beta}\jg$ with the group $N_0$ replaced by $N_0/N_{\beta,0}$. For the $k$th layer it is clear that the kernel of the natural map $A \bg N_{\Delta,k}\jg\to A \bg N_{\Delta,k,\beta}\jg$ is generated by $n_\beta-1$ . Moreover, the maps 
$A \bg N_{\Delta,k}\jg(n_\beta-1)\to A \bg N_{\Delta,k-1}\jg(n_\beta-1)$ are clearly surjective. Therefore the projective limit of the diagrams
\begin{equation*}
0\to A \bg N_{\Delta,k}\jg(n_\beta-1)\to A \bg N_{\Delta,k}\jg\to A \bg N_{\Delta,k,\beta}\jg\to 0
\end{equation*}
is exact whence the kernel of the map $A \bg N_{\Delta,\infty}\jg\to A \bg N_{\Delta,\infty,\beta}\jg$ is also generated by $n_\beta-1$. Now we have a composite map
\begin{equation*}
A \bg N_{\Delta,\infty}\jg\overset{\Pi_\beta}{\to} A \bg N_{\Delta,\infty,\beta}\jg\overset{\Pi_{\Phi^+\setminus(\Delta\cup\{\beta\})}}{\to} A \bg N_{\Delta,0}\jg\ .
\end{equation*}
We have just shown that $\Ker(\Pi_\beta)$ is generated by $n_\beta-1$ and by the inductional hypothesis we see that $\Ker(\Pi_{\Phi^+\setminus(\Delta\cup\{\beta\})})$ is generated by $\{n_\gamma-1\mid \gamma\in \Phi^+\setminus(\Delta\cup\{\beta\})\}$. The statement follows by combining these two.
\end{proof}

We denote by $\varphi_t$ the ring endomorphism of $A \bg N_{\Delta,k}\jg$ induced by the conjugation of $t\in T_+$ on $N_{\Delta,k}$ for all $k\geq 0$. This induces a ring endomorphism (still denoted by $\varphi_t$) on $A \bg N_{\Delta,\infty}\jg$ by taking the projective limit.

\begin{lem}\label{decAbgNDeltainftyjg}
For each $t\in T_+$ the ring homomorphism $\varphi_t\colon A \bg N_{\Delta,\infty}\jg \to A \bg N_{\Delta,\infty}\jg $ is injective and we have
\begin{equation*}
A \bg N_{\Delta,\infty}\jg =\bigoplus_{u\in J(N_0/tN_0t^{-1})}u\varphi_t(A \bg N_{\Delta,\infty}\jg)=\bigoplus_{u\in J(N_0/tN_0t^{-1})}\varphi_t(A \bg N_{\Delta,\infty}\jg)u
\end{equation*}
as left (resp.\ right) modules over the the subring $\varphi_t(A \bg N_{\Delta,\infty}\jg)$.
\end{lem}
\begin{proof}
Let $t\in T_+$ be arbitrary and choose an integer $k(t)>0$ so that we have $H_{\Delta,k(t)}\leq tN_0t^{-1}$ and $s^{k(t)n_0}t^{-1}\in T_+$, and let $k\geq k(t)$. Then we have $|N_0:\varphi_t(N_0)|=|N_{\Delta,k}:\varphi_t(N_{\Delta,k})|$. In particular, for the Iwasawa algebra $A \bs N_{\Delta,k}\js$ we have 
\begin{equation}\label{iwdirsum}
A \bs N_{\Delta,k}\js=\bigoplus_{u\in J(N_0/tN_0t^{-1})}u\varphi_t(A \bs N_{\Delta,k}\js)\ . 
\end{equation}
By the choice of $k$ the element $s^{kn_0}n_\alpha s^{-kn_0}-1$ lies in the subring $\varphi_t(A \bs N_{\Delta,k}\js)$ for each $\alpha\in\Delta$. So inverting $s^{kn_0}n_\alpha s^{-kn_0}-1$ for each $\alpha\in\Delta$ on both sides of \eqref{iwdirsum} we obtain
\begin{align}
A \bg N_{\Delta,k}\jg=\bigoplus_{u\in J(N_0/tN_0t^{-1})}u\varphi_t(A \bs N_{\Delta,k}\js[\varphi_{s^{kn_0}t^{-1}}(n_\alpha-1)^{-1}\mid\alpha\in\Delta])=\notag\\
=\bigoplus_{u\in J(N_0/tN_0t^{-1})}u\varphi_t(A \bg N_{\Delta,k}\jg) \label{kleveldirsum}
\end{align}
since by inverting $\varphi_{s^{kn_0}t^{-1}}(n_\alpha-1)$ the element $\varphi_{s^kn_0}(n_\alpha-1)$ becomes invertible, too, and vice versa. We deduce the statement by taking projective limits. The left module version follows similarly.

For the injectivity of $\varphi_t$ on $A \bg N_{\Delta,\infty}\jg $ note that we have a short exact sequence
\begin{equation*}
0\to (t^{-1}H_{\Delta,k}t-1)A \bs N_{\Delta,k}\js\to A \bs N_{\Delta,k}\js\overset{\varphi_t}{\to}\varphi_t(A \bs N_{\Delta,k}\js)\to 0\ .
\end{equation*}
Inverting again the central elements $\varphi_{s^kn_0}(n_\alpha-1)$ for each $\alpha\in\Delta$ we obtain a projective system of short exact sequences
\begin{equation}\label{kerphitlevelk}
0\to (t^{-1}H_{\Delta,k}t-1)A \bg N_{\Delta,k}\jg\to A \bg N_{\Delta,k}\jg\overset{\varphi_t}{\to}\varphi_t(A \bg N_{\Delta,k}\jg)\to 0\ .
\end{equation}
Now there exists an integer $k_2>k$ such that $H_{\Delta,k_2}\leq tH_{\Delta,k}t^{-1}$ whence $(t^{-1}H_{\Delta,k_2}t-1)\subseteq H_{\Delta,k}-1$. So the connecting map for the left hand side of \eqref{kerphitlevelk} is the zero map from the $k_2$th level to the $k$th level. By taking projective limits we obtain the required injectivity.
\end{proof}

\begin{pro}\label{noeth}
The ring $A \bg N_{\Delta,\infty}\jg$ is (left and right) noetherian.
\end{pro}
\begin{proof}
We may assume without loss of generality that $N_0$ is uniform as a pro-$p$ group since any ring that is finitely generated as a module over a noetherian subring is itself noetherian. Moreover, it suffices to treat the case $h=1$. We consider the filtration on $\kappa\bg N_{\Delta,\infty}\jg$ by the powers of the ideal $I:=\Ker(\kappa\bg N_{\Delta,\infty}\jg\to \kappa\bg N_{\Delta,0}\jg)$. This filtration is cofinal with the filtration induced by the kernels of the maps $\kappa\bg N_{\Delta,\infty}\jg\to \kappa\bg N_{\Delta,n}\jg$ ($n\geq 0$) therefore $\kappa\bg N_{\Delta,\infty}\jg$ is complete with respect to the $I$-adic topology. By Lemma \ref{augker} the powers $I^n$ of the ideal $I$ are generated by $n$-term products of the elements $n_\beta-1$ for $\beta\in\Phi^+\setminus\Delta$. Since $N_0$ is uniform, its commutator subgroup is contained in $N_0^p$. Moreover, the commutator $[N_0,N_0]$ of $N_0$ is also contained in $H_{\Delta,0}$ since $N_0/H_{\Delta,0}\cong N_{\Delta,0}$ is commutative. Since $N_{\Delta,0}$ does not have elements of order $p$, $[N_0,N_0]$ is contained in $H_{\Delta,0}^p$. This shows that the graded ring $\gr \kappa\bg N_{\Delta,\infty}\jg$ of $\kappa\bg N_{\Delta,\infty}\jg$ is a commutative polynomial ring over $\kappa\bg N_{\Delta,0}\jg$ in the variables $X_\beta=\gr (n_\beta-1)$ for $\beta\in\Phi^+\setminus\Delta$. This is a noetherian ring by Hilbert's basis theorem as $\kappa\bg N_{\Delta,0}\jg$ is the localization of the noetherian ring $\kappa\bs N_{\Delta,0}\js$ therefore also noetherian. Finally, the statement follows from Prop.\ I.7.1.2 in \cite{LvO}.
\end{proof}

\begin{lem}\label{jacobson}
The Jacobson radical of $A\bg N_{\Delta,\infty}\jg$ is the ideal generated by $\varpi$ and $(H_{\Delta,0}-1)A\bg N_{\Delta,\infty}\jg$.
\end{lem}
\begin{proof}
$A\bg N_{\Delta,\infty}\jg$ is complete with respect to the filtration induced by the powers of $\varpi A\bg N_{\Delta,\infty}\jg +(H_{\Delta,0}-1)A\bg N_{\Delta,\infty}\jg$. Therefore this ideal is contained in the Jacobson radical of $A\bg N_{\Delta,\infty}\jg$. The other direction follows from noting that the Jacobson radical of $\kappa\bg N_{\Delta,0}\jg\cong \kappa\bs X_\alpha\mid \alpha\in \Delta\js [X_\alpha^{-1}\mid \alpha\in \Delta]$ is $0$ by Prop.\ \ref{noideal} since it is $T_0$-invariant.
\end{proof}

\subsection{The equivalence of categories}

By Lemma \ref{decAbgNDeltainftyjg} we may consider \'etale $T_+$-modules over $A \bg N_{\Delta,\infty}\jg$ the usual way: an \'etale $T_+$-module $D_\infty$ over $A \bg N_{\Delta,\infty}\jg$ is a finitely generated $A \bg N_{\Delta,\infty}\jg$-module with a semilinear action of the monoid $T_+$ such that for all $t\in T_+$ the map
\begin{eqnarray*}
1\otimes\varphi_t\colon A \bg N_{\Delta,\infty}\jg\otimes_{\varphi_t, A \bg N_{\Delta,\infty}\jg}D_\infty&\to& D_\infty\\
\lambda\otimes x&\mapsto& \lambda\varphi_t(x)
\end{eqnarray*}
is an isomorphism. Similarly, we may consider \'etale $T_\ast$-modules over $A \bg N_{\Delta,\infty}\jg$ for any submonoid $T_\ast\leq T_+$ and also over the rings $A \bg N_{\Delta,k}\jg$ for each $k\geq 0$.

Let $T_\ast\leq T_+$ be a submonoid containing some power $s^{r_\ast}$ of $s=\xi(p)$. Our goal is to prove an equivalence between the category $\mathcal{D}^{et}(T_\ast,A \bg N_{\Delta,0}\jg)$ of (finitely generated) \'etale $T_\ast$-modules over $A \bg N_{\Delta,0}\jg$ and the category $\mathcal{D}^{et}(T_\ast,A \bg N_{\Delta,\infty}\jg)$ of (finitely generated) \'etale $T_\ast$-modules over $A \bg N_{\Delta,\infty}\jg$. The reason why this is not a formal consequence of Thm.\ 8.20 in \cite{SVZ} or of Prop.\ 3.1 in \cite{Z14} or even of those proofs is that there is no section of the group homomorphism $N_0\twoheadrightarrow N_{\Delta,0}$ in general. Therefore it is not obvious a priori how to construct a functor from $\mathcal{D}^{et}(T_\ast,A \bg N_{\Delta,0}\jg)$ to $\mathcal{D}^{et}(T_\ast,A \bg N_{\Delta,\infty}\jg)$. The idea is to prove an equivalence over each level $k$ and to take the projective limit. Even though there is no section of the group homomorphism $N_{\Delta,k}\twoheadrightarrow N_{\Delta,0}$ either, we do have a section from a finite index subgroup $s^{k}N_{\Delta,0}s^{-k}$. 

For an \'etale $T_\ast$-module  $D$ over $A \bg N_{\Delta,0}\jg$ and integer $r_\ast\mid k$ note that we have a ring homomorphism $\varphi_{s^{kn_0}}\colon A\bg N_{\Delta,0}\jg\hookrightarrow A\bg s^{kn_0}N_{\Delta,0}s^{-kn_0}\jg\leq A\bg N_{\Delta,k}\jg$. So we define $$D_k:=\mathbb{M}_{k,0}(D):=A\bg N_{\Delta,k}\jg\otimes_{\varphi_{s^{kn_0}},A\bg N_{\Delta,0}\jg}D\ .$$ This is an \'etale $T_\ast$-module over $A\bg N_{\Delta,k}\jg$ as we compute
\begin{align*}
A\bg N_{\Delta,k}\jg\otimes_{\varphi_t,A\bg N_{\Delta,k}\jg}D_k= A\bg N_{\Delta,k}\jg\otimes_{\varphi_t,A\bg N_{\Delta,k}\jg}A\bg N_{\Delta,k}\jg\otimes_{\varphi_{s^{kn_0}},A\bg N_{\Delta,0}\jg}D\cong\\
\cong A\bg N_{\Delta,k}\jg\otimes_{\varphi_{ts^{kn_0}},A\bg N_{\Delta,0}\jg}D\cong A\bg N_{\Delta,k}\jg\otimes_{\varphi_{s^{kn_0}},A\bg N_{\Delta,0}\jg}A\bg N_{\Delta,0}\jg\otimes_{\varphi_{t},A\bg N_{\Delta,0}\jg}D\cong\\
\cong A\bg N_{\Delta,k}\jg\otimes_{\varphi_{s^{kn_0}},A\bg N_{\Delta,0}\jg}D\cong D_k
\end{align*}
for each $t\in T_\ast$. The above identification is given by the map $1\otimes\varphi_t\colon A\bg N_{\Delta,k}\jg\otimes_{\varphi_t,A\bg N_{\Delta,k}\jg}D_k\to D_k$. On the other hand, we compute
\begin{align*}
H_0(H_{\Delta,0}/H_{\Delta,k},D_k)=H_0(H_{\Delta,0}/H_{\Delta,k},A\bg N_{\Delta,k}\jg\otimes_{\varphi_{s^{kn_0}},A\bg N_{\Delta,0}\jg}D)\cong\\
\cong H_0(H_{\Delta,0}/H_{\Delta,k},A\bg N_{\Delta,k}\jg)\otimes_{\varphi_{s^{kn_0}},A\bg N_{\Delta,0}\jg}D\cong A\bg N_{\Delta,0}\jg \otimes_{\varphi_{s^{kn_0}},A\bg N_{\Delta,0}\jg}D\cong D
\end{align*}
by Lemma \ref{HhomologAbgjg}. So we obtained a natural isomorphism between the identity functor on $\mathcal{D}^{et}(T_\ast,A \bg N_{\Delta,0}\jg)$ and the functor $H_0(H_{\Delta,0}/H_{\Delta,k},(\cdot)_k)$.

\begin{lem}\label{equivcatk}
The functors $\mathbb{M}_{k,0}\colon D\mapsto D_k$ and $\mathbb{D}_{0,k}\colon D_k\mapsto H_0(H_{\Delta,0}/H_{\Delta,k},D_k)$ are quasi-inverse equivalences of categories between $\mathcal{D}^{et}(T_\ast,A \bg N_{\Delta,0}\jg)$ and $\mathcal{D}^{et}(T_\ast,A \bg N_{\Delta,k}\jg)$.
\end{lem}
\begin{proof}
We have already seen that $\mathbb{D}_{0,k}\circ\mathbb{M}_{k,0}\cong\id$. For the other direction let $D_k$ be an object in $\mathcal{D}^{et}(T_\ast,A \bg N_{\Delta,k}\jg)$. Note that by Lemma \ref{HhomologAbgjg} we have $\mathbb{D}_{0,k}=A \bg N_{\Delta,0}\jg\otimes_{A \bg N_{\Delta,k}\jg}\cdot$. Moreover, $(H_{\Delta,0}/H_{\Delta,k}-1)A \bg N_{\Delta,k}\jg$ lies in the kernel of the ring homomorphism $\varphi_{s^{kn_0}}\colon A \bg N_{\Delta,k}\jg\to A \bg N_{\Delta,k}\jg$. So we may factor $\varphi_{s^{kn_0}}$ as
\begin{equation*}
\xymatrix{
A \bg N_{\Delta,k}\jg\ar@{->>}[r]\ar@/_2pc/[rr]_{\varphi_{s^{kn_0}}} & A \bg N_{\Delta,0}\jg \ar[r]^{\varphi_{s^{kn_0}}\ \ } &  A \bg N_{\Delta,k}\jg\ .\\
&&
}
\end{equation*}
Therefore we compute
\begin{align*}
\mathbb{M}_{k,0}\circ\mathbb{D}_{0,k}(D_k)=A \bg N_{\Delta,k}\jg \otimes_{\varphi_{s^{kn_0}},A\bg N_{\Delta,0}\jg} A\bg N_{\Delta,0}\jg \otimes_{A\bg N_{\Delta,k}\jg} D_k\cong\\
\cong A \bg N_{\Delta,k}\jg \otimes_{\varphi_{s^{kn_0}},A\bg N_{\Delta,k}\jg} D_k\cong D_k
\end{align*}
where the last isomorphism follows from the \'etaleness of $D_k$.
\end{proof}

For an object $D$ in $\mathcal{D}^{et}(T_\ast,A \bg N_{\Delta,0}\jg)$ and integers $k_1\leq k_2$ both divisible by $r_\ast$ we have $$\mathbb{M}_{k_1,0}(D)=A\bg N_{\Delta,k_1}\jg \otimes_{A\bg N_{\Delta,k_2}\jg}\mathbb{M}_{k_2,0}(D)\ .$$ In particular, $(\mathbb{M}_{k,0}(D))_{k\geq 0}$ forms a projective system. So we define $\mathbb{M}_{\infty,0}:=\varprojlim_k\mathbb{M}_{k,0}$ and $\mathbb{D}_{0,\infty}:=H_0(H_{\Delta,0},\cdot)=A \bg N_{\Delta,0}\jg \otimes_{A \bg N_{\Delta,\infty}\jg}\cdot$. $\mathbb{D}_{0,\infty}$ is a functor from $\mathcal{D}^{et}(T_\ast,A \bg N_{\Delta,\infty}\jg)$ to $\mathcal{D}^{et}(T_\ast,A \bg N_{\Delta,0}\jg)$. It is not so trivial (but as we shall see, true) that $\mathbb{M}_{\infty,0}$ is also a functor in the reverse direction. Note that any module over $A\bg N_{\Delta,k}\jg$ can be regarded as a module over $A\bg N_{\Delta,\infty}\jg$ via the quotient map $A\bg N_{\Delta,\infty}\jg\twoheadrightarrow A\bg N_{\Delta,k}\jg$ for any $k\geq 0$ divisible by $r_\ast$. Moreover, we have a semilinear action of $T_\ast$ on each $\mathbb{M}_{k,0}(D)$ whence also on the projective limit as the connecting maps are $T_\ast$-equivariant by construction. In particular, $\mathbb{M}_{\infty,0}(D)$ is a module over $A\bg N_{\Delta,\infty}\jg$ with a semilinear action of $T_\ast$. Our key result is the following

\begin{pro}\label{Minftyimage}
For each object $D$ in $\mathcal{D}^{et}(T_\ast,A \bg N_{\Delta,0}\jg)$ the $A\bg N_{\Delta,\infty}\jg$-module $\mathbb{M}_{\infty,0}(D)$ with  the $T_\ast$-action induced by the $T_\ast$-action on $D$ is an object in $\mathcal{D}^{et}(T_\ast,A \bg N_{\Delta,\infty}\jg)$.
\end{pro}
\begin{proof}
We proceed in $3$ steps.

\emph{Step 1: We show that $\mathbb{M}_{\infty,0}(D)$ is finitely generated over $A\bg N_{\Delta,\infty}\jg$.}  Note that the kernel of the ring homomorphism $A \bg N_{\Delta,k}\jg \twoheadrightarrow A \bg N_{\Delta,0}\jg $ is a nilpotent ideal therefore contained in the Jacobson radical of $A \bg N_{\Delta,k}\jg $. So if $D$ is generated by the elements $d_1,\dots,d_r$ then any lifts $d_{1,k},\dots,d_{r,k}$ of $d_1,\dots,d_r$ to $\mathbb{M}_{k,0}(D)$ generate $\mathbb{M}_{k,0}(D)$ by Nakayama's Lemma ($r_\ast\mid k$). Since the natural quotient maps $\mathbb{M}_{k_2,0}(D)\to \mathbb{M}_{k_1,0}(D)$ are surjective for each pair $k_1\leq k_2$ of integers (both divisible by $r_\ast$), we can choose the lifts $d_{1,k},\dots,d_{r,k}$ recursively in a compatible way so that $d_{i,k_2}$ maps to $d_{i,k_1}$ under the quotient map $\mathbb{M}_{k_2,0}(D)\to \mathbb{M}_{k_1,0}(D)$. Now for $k_1\leq k_2$ consider the commutative diagram 
\begin{equation}\label{Ck1Ck2}
\xymatrixcolsep{4pc}\xymatrix{
0\ar[r] & C_{k_2}\ar[r]\ar[d] & \bigoplus_{i=1}^r A \bg N_{\Delta,k_2}\jg e_{i,k_2} \ar[r]^-{e_{i,k_2}\mapsto d_{i,k_2}}\ar[d]^{e_{i,k_2}\mapsto e_{i,k_1}} &  \mathbb{M}_{k_2,0}(D)\ar[r]\ar[d] &0\\
0\ar[r] & C_{k_1}\ar[r] & \bigoplus_{i=1}^r A \bg N_{\Delta,k_1}\jg e_{i,k_1} \ar[r]^-{e_{i,k_1}\mapsto d_{i,k_1}} &  \mathbb{M}_{k_1,0}(D)\ar[r] &0
}
\end{equation}
in which $C_{k_j}$ ($j=1,2$) are defined so as to make the rows exact. Since the map $\mathbb{M}_{k_2,0}(D)\to \mathbb{M}_{k_1,0}(D)$ is induced by the identification $\mathbb{M}_{k_1,0}(D)\cong A\bg N_{\Delta,k_1}\jg\otimes_{A\bg N_{\Delta,k_2}\jg} \mathbb{M}_{k_2,0}(D)$ and the horizontal map in the middle is also induced by the natural quotient map $A\bg N_{\Delta,k_2}\jg \to A\bg N_{\Delta,k_1}\jg$, we obtain an exact sequence
\begin{equation}\label{notyetontheleft}
\xymatrixcolsep{3.7pc}\xymatrix{
A\bg N_{\Delta,k_1}\jg\otimes_{A\bg N_{\Delta,k_2}\jg}C_{k_2}\ar[r] & \bigoplus_{i=1}^r A \bg N_{\Delta,k_1}\jg e_{i,k_1} \ar[r]^-{e_{i,k_1}\mapsto d_{i,k_1}} &  \mathbb{M}_{k_1,0}(D)\ar[r] &0\ .
}
\end{equation}
We deduce that the composite map $C_{k_2}\twoheadrightarrow A\bg N_{\Delta,k_1}\jg\otimes_{A\bg N_{\Delta,k_2}\jg}C_{k_2}\to  C_{k_1}$ is also surjective. Therefore by the Mittag--Leffler condition we obtain an exact sequence
\begin{equation*}
\xymatrix{
0\ar[r] & \varprojlim_{r_\ast\mid k} C_{k}\ar[r] & \bigoplus_{i=1}^r A \bg N_{\Delta,\infty}\jg e_{i,\infty} \ar[r] &  \mathbb{M}_{\infty,0}(D)\ar[r] &0\ .
}
\end{equation*}
In particular, $\mathbb{M}_{\infty,0}(D)$ is finitely generated over $A\bg N_{\Delta,\infty}\jg$. 

\emph{Step 2: We show that the map \eqref{phitinfty} below is injective for all $t\in T_\ast$.}
\begin{lem}\label{H1vanish}
For all $k\geq 0$ divisible by $r_\ast$ We have $H_i(H,\mathbb{M}_{k,0}(D))=0$ for all $i>0$ and subgroup $H\leq H_{\Delta,0}/H_{\Delta,k}$.
\end{lem}
\begin{proof}
The group ring $A\bg s^{kn_0}N_{\Delta,0}s^{-kn_0}\jg[H_{\Delta,0}/H_{\Delta,k}]$ is a subring in $A\bg N_{\Delta,k}\jg$. Moreover, we have 
\begin{equation*}
A\bg N_{\Delta,k}\jg=\bigoplus_{u\in J(N_0/s^{kn_0}N_0s^{-kn_0}H_{\Delta,0})}A\bg s^{kn_0}N_{\Delta,0}s^{-kn_0}\jg[H_{\Delta,0}/H_{\Delta,k}]u\ .
\end{equation*}
In particular, $A\bg N_{\Delta,k}\jg$ is a free left module over $A\bg s^{kn_0}N_{\Delta,0}s^{-kn_0}\jg[H_{\Delta,0}/H_{\Delta,k}]$. Therefore $\mathbb{M}_{k,0}(D)=A\bg N_{\Delta,k}\jg\otimes_{\varphi_{s^{kn_0}},A\bg N_{\Delta,0}\jg}D$ is an induced module as a representation of any subgroup $H\leq H_{\Delta,0}/H_{\Delta,k}$. In particular, its higher homology groups vanish.
\end{proof}
Now we take $H_{\Delta,k_1}/H_{\Delta,k_2}$-homology of the upper row in \eqref{Ck1Ck2} and apply Lemma \ref{H1vanish} (with $H:=H_{\Delta,k_1}/H_{\Delta,k_2}\leq H_{\Delta,0}/H_{\Delta,k_2}$) to deduce the exactness of the sequence \eqref{notyetontheleft} on the left using Lemma \ref{HhomologAbgjg} and the long exact sequence of homology. By the lower row in \eqref{Ck1Ck2} we obtain the isomorphism $C_{k_1}\cong A\bg N_{\Delta,k_1}\jg\otimes_{A\bg N_{\Delta,k_2}\jg}C_{k_2}$ for each pair $k_1\leq k_2$ (divisible by $r_\ast$). Now for any fixed $k\geq 0$ we get a commutative diagram
\begin{equation}
\xymatrix{
&0\ar[d] \\ 
A\bg N_{\Delta,k}\jg\otimes_{A\bg N_{\Delta,\infty}\jg} \varprojlim_{k'} C_{k'}\ar[r]\ar[d] & C_{k}\ar[d]\\
\bigoplus_{i=1}^r A \bg N_{\Delta,k}\jg e_{i,\infty} \ar[r]\ar[d] &  \bigoplus_{i=1}^r A \bg N_{\Delta,k}\jg e_{i,k} \ar[d] \\
A\bg N_{\Delta,k}\jg\otimes_{A\bg N_{\Delta,\infty}\jg} \mathbb{M}_{\infty,0}(D)\ar[r]\ar[d]  &  \mathbb{M}_{k,0}(D)\ar[d]\\
0 &0
}\label{Dkinfty}
\end{equation}
with exact columns. The horizontal maps are onto since for example $C_k$ is quotient of $\varprojlim_{k'}C_{k'}$ by the surjectivity of the maps $C_{k'}\to C_k$ ($k\leq k'$) and this quotient map factors through the maximal quotient of $\varprojlim_{k'}C_{k'}$ on which $A\bg N_{\Delta,\infty}\jg$ acts via its quotient $A\bg N_{\Delta,k}\jg$. Therefore the lower horizontal map is an isomorphism as the middle map clearly is.

Now take an arbitrary $t$ in $T_\ast$. Since $T_\ast$ acts on $\mathbb{M}_{\infty,0}(D)$, we have a map
\begin{eqnarray}
1\otimes\varphi_t\colon A\bg N_{\Delta,\infty}\jg\otimes_{\varphi_t,A\bg N_{\Delta,\infty}\jg}\mathbb{M}_{\infty,0}(D)&\to& \mathbb{M}_{\infty,0}(D)\notag\\
\sum_{u\in J(N_0/tN_0t^{-1})}u\otimes m_u&\mapsto& \sum_{u\in J(N_0/tN_0t^{-1})}u\varphi_t(m_u)\ .\label{phitinfty}
\end{eqnarray}
Write each $m_u$ ($u\in J(N_0/tN_0t^{-1})$) as a sequence $m_u=(m_{u,k})_k$ with $m_{u,k}\in \mathbb{M}_{k,0}(D)$. Assume that $\sum_{u\in J(N_0/tN_0t^{-1})}u\otimes m_u$ lies in the kernel of the map \eqref{phitinfty}. Then for each $k\geq 0$ (divisible by $r_\ast$) we have 
\begin{equation*}
\sum_{u\in J(N_0/tN_0t^{-1})}u\varphi_t(m_{u,k})=0\ .
\end{equation*}
Since $\mathbb{M}_{k,0}(D)$ is an \'etale $T_\ast$-module by Lemma \ref{equivcatk}, we obtain that $\sum_{u\in J(N_0/tN_0t^{-1})}u\otimes m_{u,k}=0$ in $A\bg N_{\Delta,k}\jg\otimes_{\varphi_t,A\bg N_{\Delta,k}\jg}\mathbb{M}_{k,0}(D)$. Let now $k$ be big enough so that $H_{\Delta,k}$ is contained in $tN_0t^{-1}$. Then $A\bg N_{\Delta,k}\jg$ is a free right module over the image of $\varphi_t$ with generators $u\in J(N_0/tN_0t^{-1})$. Therefore we have $1\otimes m_{u,k}=0$ for each $u\in J(N_0/tN_0t^{-1})$ which shows that $m_{u,k}$ lies in $(t^{-1}H_{\Delta,k}t-1)\mathbb{M}_{k,0}(D)$. So for any $0\leq k'\leq k$ for which $H_{\Delta,k'}$ contains $t^{-1}H_{\Delta,k}t$ we deduce that $m_{u,k'}=0$. However, for any fixed $0\leq k'$ there exists a large enough $k\geq k'$ with $H_{\Delta,k'}\supseteq t^{-1}H_{\Delta,k}t$ so we have $m_{u,k'}=0$ for all $k'\geq 0$. Therefore \eqref{phitinfty} is injective.

\emph{Step 3: We show that \eqref{phitinfty} is surjective.} Let $m=(m_k)_k\in\mathbb{M}_{\infty,0}(D)$ be arbitrary. Since $\mathbb{M}_{k,0}(D)$ is an \'etale $T_\ast$-module over $A\bg N_{\Delta,k}\jg$ for all $k\geq 0$ divisible by $r_\ast$ (see Lemma \ref{equivcatk}), we may decompose $m_k$ as 
\begin{equation}\label{varphitmuk}
m_k=\sum_{u\in J(N_0/tN_0t^{-1})}u\varphi_t(m_{u,k}) 
\end{equation}
for some $m_{u,k}\in \mathbb{M}_{k,0}(D)$. Note that the elements $m_{u,k}$ are not unique since $\varphi_t$ is not injective on $\mathbb{M}_{k,0}(D)$. However, their images $\varphi_t(m_{u,k})$ under $\varphi_t$ are unique for each $u\in J(N_0/tN_0t^{-1})$ and $k\geq k(t)$ large enough so that $H_{\Delta,k}$ is contained in $tN_0t^{-1}$. Fix $m_{u,k}$ arbitrarily for each $u\in J(N_0/tN_0t^{-1})$ and $k\geq k(t)$ (divisible by $r_\ast$) so that they satisfy \eqref{varphitmuk} and put $$Y_{u,k}:=\{m'_{u,k}\in\mathbb{M}_{k,0}(D)\mid \varphi_t(m'_{u,k})=\varphi_t(m_{u,k})\}\ .$$
Let $k\geq k(t)$ be fixed now let $k'\geq k$ be another integer so that we have $t^{-1}H_{\Delta,k'}t\leq H_{\Delta,k}$. The kernel of $\varphi_t\colon \mathbb{M}_{k',0}(D)\to \mathbb{M}_{k',0}(D)$ equals $(t^{-1}H_{\Delta,k'}t-1)\mathbb{M}_{k',0}$. However, by our assumption that $t^{-1}H_{\Delta,k'}t$ is contained in $H_{\Delta,k}$, it follows that $(t^{-1}H_{\Delta,k'}t-1)\mathbb{M}_{k',0}$ maps to $0$ in $\mathbb{M}_{k,0}(D)$. Hence for each $u\in J(N_0/tN_0t^{-1})$ the image of $Y_{u,k'}$ in $Y_{u,k}$ is a single element $m^\ast_{u,k}\in Y_{u,k}$. Therefore the projective system $(Y_{u,k})_k$ satisfies the Mittag--Leffler condition so that $Y_{u,\infty}:=\varprojlim_k Y_{u,k}\subset \mathbb{M}_{\infty,0}(D)$ is a set having a single element $m_u:=(m^\ast_{u,k})_{k\geq 0}$. We clearly have $m=\sum_{u\in J(N_0/tN_0t^{-1})}u\varphi_t(m_u)$ as required.
\end{proof}

Now our main result in this section becomes a simple application of the above Proposition.

\begin{thm}\label{equivcat}
The functors $\mathbb{M}_{\infty,0}$ and $\mathbb{D}_{0,\infty}$ are quasi-inverse equivalences of categories between $\mathcal{D}^{et}(T_\ast,A \bg N_{\Delta,0}\jg)$ and $\mathcal{D}^{et}(T_\ast,A \bg N_{\Delta,\infty}\jg)$.
\end{thm}
\begin{proof}
We saw in the proof of Prop.\ \ref{Minftyimage} that the lower horizontal map in \eqref{Dkinfty} is an isomorphism for all $k\geq 0$ divisible by $r_\ast$. In the special case of $k=0$ this provides us with a natural isomorphism between the identity and $\mathbb{D}_{0,\infty}\circ\mathbb{M}_{\infty,0}$ on $\mathcal{D}^{et}(T_\ast,A \bg N_{\Delta,0}\jg)$.

Let $D_\infty$ be an object in $\mathcal{D}^{et}(T_\ast,A \bg N_{\Delta,\infty}\jg)$. For each $k\geq 0$ (divisible by $r_\ast$) we have a natural quotient map 
\begin{equation*}
D_\infty\to D_k:=A \bg N_{\Delta,k}\jg\otimes_{A \bg N_{\Delta,\infty}\jg}D_\infty\ .
\end{equation*}
Moreover, $D_k$ is an \'etale $T_\ast$-module over $A \bg N_{\Delta,k}\jg$ corresponding to $D_0=\mathbb{D}_{0,\infty}(D_\infty)$ via the equivalence of categories in Lemma \ref{equivcatk}. These reduction maps are compatible therefore we obtain a natural map $f\colon D_\infty\to\varprojlim_k D_k=\mathbb{M}_{\infty,0}\circ\mathbb{D}_{0,\infty}(D_\infty)$. This map $f$ is an isomorphism modulo the Jacobson radical of $A\bg N_{\Delta,\infty}\jg$ by Lemma \ref{jacobson} using that $\mathbb{D}_{0,\infty}\circ\mathbb{M}_{\infty,0}\circ\mathbb{D}_{0,\infty}\cong\mathbb{D}_{0,\infty}$. In particular $f$ is surjective. Moreover, $A \bg N_{\Delta,k}\jg\otimes_{A \bg N_{\Delta,\infty}\jg}f$ is also an isomorphism for all $k\geq 0$ divisible by $r_\ast$. Therefore we have $\Ker(f)\subseteq (H_{\Delta,k}-1)D_{\infty}$ for all $k$. Since the powers of the Jacobson radical $J$ of $A \bg N_{\Delta,\infty}\jg$ are cofinal with the ideals $(H_{\Delta,k}-1)A \bg N_{\Delta,\infty}\jg$ we deduce that $\Ker(f)\subseteq J^kD_{\infty}$ for all $k\geq 0$. Now note that $J$ is generated by a centralizing sequence by Lemma \ref{augker}. Therefore it satisfies the Artin--Rees property by Thm.\ 4.2.7 in \cite{McCR} and Prop.\ \ref{noeth}. This shows that $\bigcap_{k\geq 0} J^k D_\infty=\{0\}$ since $D_\infty$ is finitely generated. In particular, $f$ is an isomorphism.
\end{proof}
\begin{rem}
The above proof shows that we have $D_\infty\cong\varprojlim_k A \bg N_{\Delta,k}\jg\otimes_{A \bg N_{\Delta,\infty}\jg}D_\infty$ for any object $D_\infty$ in $\mathcal{D}^{et}(T_\ast,A \bg N_{\Delta,\infty}\jg)$.
\end{rem}

\subsection{A noncommutative variant of $D^\vee_{\Delta}$}

The goal in this section is to define a noncommutative variant $D^\vee_{\Delta,\infty}$ of the functor $D^\vee_\Delta$ from smooth representations $\pi$ of $G$ over $A$ to (projective limits of) \'etale $T_+$-modules over $A\bg N_{\Delta,\infty}\jg$. The motivation is that this allows us to construct a $G$-equivariant sheaf on $G/B$ (in the sense of \cite{SVZ}) attached to $\pi$ which is crucial in possibly reconstructing $\pi$ from $D^\vee_{\Delta}(\pi)$ in the case of extensions of principal series.

Let $\pi$ be a smooth representation of $B_0$ over $A$ together with an action of the monoid $B_+$ on $\pi$ extending the action of $B_0$. (For instance, $\pi$ could be a smooth representation of $B$.) Denote by $F_{\alpha,k}$ the operator $\Tr_{H_{\Delta,k}/t_\alpha H_{\Delta,k}t_\alpha^{-1}}\circ (t_\alpha\cdot)$ on $\pi^{H_{\Delta,k}}$ and consider the skew polynomial ring $$ A\bs N_{\Delta,k}\js [F_{\Delta,k}]:=A\bs N_{\Delta,k}\js[F_{\alpha,k}\mid \alpha\in\Delta]$$ in the variables $F_{\alpha,k}$ ($\alpha\in\Delta$) that commute with each other and satisfy $F_{\alpha,k}\lambda=\varphi_\alpha(\lambda)F_{\alpha,k}$ for any $\lambda\in A\bs N_{\Delta,k}\js$ and $\alpha\in\Delta$. Further, for any $t=\prod_{\alpha\in\Delta}t_\alpha^{k_\alpha}$ we put $F_{t,k}:=\prod_{\alpha\in\Delta}F_{t_\alpha}^{k_\alpha}$. We denote by $\mathcal{M}_{\Delta,k}(\pi^{H_{\Delta,k}})$ the set of finitely generated $A\bs N_{\Delta,k}\js[F_{\Delta,k}]$-submodules of $\pi^{H_{\Delta,k}}$ that are stable under the action of $T_0$ and admissible as a representation of $N_{\Delta,k}$. We proceed as in section 2 of \cite{EZ}.

\begin{lem}\label{formula}
For each $\alpha\in\Delta$ we have $F_{\alpha,0}=F_\alpha$ and $F_{\alpha,k}\circ \Tr_{H_{\Delta,k}/s^kH_{\Delta,0}s^{-k}}\circ(s^k \cdot)=\Tr_{H_{\Delta,k}/s^kH_{\Delta,0}s^{-k}}\circ(s^k \cdot)\circ F_{\alpha,0}$ as maps on $\pi^{H_{\Delta,0}}$.
\end{lem}
\begin{proof}
We compute
\begin{align*}
F_{\alpha,k}\circ \Tr_{H_{\Delta,k}/s^kH_{\Delta,0}s^{-k}}\circ(s^k \cdot)=\\
=\Tr_{H_{\Delta,k}/t_\alpha H_{\Delta,k}t_\alpha^{-1}}\circ(t_\alpha\cdot)\circ \Tr_{H_{\Delta,k}/s^kH_{\Delta,0}s^{-k}}\circ(s^k \cdot)=\\
=\Tr_{H_{\Delta,k}/t_\alpha H_{\Delta,k}t_\alpha^{-1}}\circ \Tr_{t_\alpha H_{\Delta,k}t_\alpha^{-1}/s^kt_\alpha H_{\Delta,0}t_\alpha^{-1}s^{-k}}\circ(s^{k}t_\alpha \cdot)=\\
=\Tr_{H_{\Delta,k}/s^{k}t_\alpha H_{\Delta,0}t_\alpha^{-1}s^{-k}}\circ(s^{k}t_\alpha \cdot)=\\
=\Tr_{H_{\Delta,k}/s^kH_{\Delta,0}s^{-k}}\circ \Tr_{s^kH_{\Delta,0}s^{-k}/s^{k}t_\alpha H_{\Delta,0}t_\alpha^{-1}s^{-k}}\circ(s^{k}t_\alpha \cdot)=\\
=\Tr_{H_{\Delta,k}/s^kH_{\Delta,0}s^{-k}}\circ(s^k\cdot)\circ \Tr_{H_{\Delta,0}/t_\alpha H_{\Delta,0}t_\alpha^{-1}}\circ(t_\alpha \cdot)=\\
=\Tr_{H_{\Delta,k}/s^kH_{\Delta,0}s^{-k}}\circ(s^k \cdot)\circ F_{\alpha,0}\ .
\end{align*}
\end{proof}

Let $M_0$ be any (finitely generated) $A\bs N_{\Delta,0}\js[F_\Delta]$-module. Then $A\bs N_{\Delta,k}\js \otimes_{A\bs N_{\Delta,0}\js,\varphi_{s^k}}M_0$ naturally has the structure of a module over $A\bs N_{\Delta,k}\js[F_{\Delta,k}]$ by putting $F_{\alpha,k}(\lambda\otimes m):=\varphi_{t_\alpha}(\lambda)\otimes F_\alpha(m)$. Let $M$ be in $\mathcal{M}_\Delta(\pi^{H_{\Delta,0}})$. In view of Lemma \ref{formula} we define $M_k$ to be the image of the  $A\bs N_{\Delta,k}\js[F_{\Delta,k}]$-module homomorphism
\begin{eqnarray*}
A\bs N_{\Delta,k}\js \otimes_{A\bs N_{\Delta,0}\js,\varphi_{s^k}}M&\to& \pi^{H_{\Delta,k}}\\
\lambda\otimes m &\mapsto& \lambda\Tr_{H_{\Delta,k}/s^k H_{\Delta,0}s^{-k}}(s^km)\ .
\end{eqnarray*}
In particular, $M_k$ is an $ A\bs N_{\Delta,k}\js [F_{\Delta,k}]$-submodule of $\pi^{H_{\Delta,k}}$. $ \Tr_{H_{\Delta,k}/s^kH_{\Delta,0}s^{-k}}\circ(s^kM)$ is a $s^kN_0s^{-k}H_{\Delta,k}$-subrepresentation of $\pi^{H_{\Delta,k}}$ and we have $M_k=N_0 \Tr_{H_{\Delta,k}/s^kH_{\Delta,0}s^{-k}}\circ(s^kM)$.

\begin{lem}\label{M_k}
For any $M\in\mathcal{M}_\Delta(\pi^{H_{\Delta,0}})$ the $N_0$-subrepresentation $M_k$ lies in $\mathcal{M}_{\Delta,k}(\pi^{H_{\Delta,k}})$.
\end{lem}
\begin{proof}
Let $\{m_1,\dots,m_r\}$ be a set of generators of $M$ as an $A\bs N_{\Delta,0}\js [F_\Delta]$-module. Then by Lemma \ref{formula} the elements $1\otimes m_1,\dots,1\otimes m_r$ generate $A\bs N_{\Delta,k}\js \otimes_{A\bs N_{\Delta,0}\js,\varphi_{s^k}}M$ as a module over $A\bs N_{\Delta,k}\js[F_{\Delta,k}]$. In particular, $M_k$ is finitely generated. 

For the stability under the action of $T_0$ note that $T_0$ normalizes both $H_{\Delta,k}$ and $s^kH_{\Delta,0}s^{-k}$ and the elements in $T_0$ commute with $s$.

Since $M$ is admissible as an $N_{\Delta,0}$-representation and $A\bs N_{\Delta,k}\js$ is finitely generated and free as a module over $A\bs s^k N_{\Delta,0}s^{-k}\js$, we obtain that $M_k$ is admissible, too.
\end{proof}

In order to simplify notation we write $M_k^\vee[1/X_\Delta]:=M_k^\vee[1/\varphi_{s^{kn_0}}(X_\Delta)]$ for $A\bg N_{\Delta,k}\jg\otimes_{A\bs N_{\Delta,k}\js}M_k^\vee$.

\begin{pro}
The map
\begin{eqnarray}
A\bs N_{\Delta,k}\js \otimes_{A\bs N_{\Delta,0}\js,\varphi_{s^k}}M&\twoheadrightarrow& M_k\notag\\
\lambda\otimes m &\mapsto& \lambda\Tr_{H_{\Delta,k}/s^k H_{\Delta,0}s^{-k}}(s^km)\ .\label{Mkdefine}
\end{eqnarray}
induces an isomorphism $M_k^\vee[1/X_\Delta]\cong A\bg N_{\Delta,k}\jg \otimes_{A\bg N_{\Delta,0}\jg,\varphi_{s^k}} M^\vee[1/X_\Delta]$.
\end{pro}
\begin{proof}
Since $A\bs N_{\Delta,k}\js$ is a finitely generated free module over $A\bs s^k N_{\Delta,0}s^{-k}\js$, we have identifications 
\begin{eqnarray*}
(A\bs N_{\Delta,k}\js \otimes_{A\bs N_{\Delta,0}\js,\varphi_{s^k}}M)^\vee&\cong& A\bs N_{\Delta,k}\js \otimes_{A\bs N_{\Delta,0}\js,\varphi_{s^k}}M^\vee\ ;\\
(A\bs N_{\Delta,k}\js \otimes_{A\bs N_{\Delta,0}\js,\varphi_{s^k}}M)^\vee[1/X_\Delta]&\cong& A\bg N_{\Delta,k}\jg \otimes_{A\bg N_{\Delta,0}\jg,\varphi_{s^k}}(M^\vee[1/X_\Delta])\ .
\end{eqnarray*}
Therefore we have an injective morphism $f\colon M_k^\vee[1/X_\Delta]\hookrightarrow A\bg N_{\Delta,k}\jg \otimes_{A\bg N_{\Delta,0}\jg,\varphi_{s^k}} M^\vee[1/X_\Delta]$ of $A\bg N_{\Delta,k}\jg$-modules. Moreover, we have a commutative diagram
\begin{equation*}
\xymatrix{
A\bs N_{\Delta,k}\js \otimes_{A\bs N_{\Delta,0}\js,\varphi_{s^k}}M\ar[d]_{\Tr_{H_{\Delta,0}/H_{\Delta,k}}}\ar@{->>}[r] & M_k\ar[d]^{\Tr_{H_{\Delta,0}/H_{\Delta,k}}} \\
A\bs N_{\Delta,k}\js \otimes_{A\bs N_{\Delta,0}\js,\varphi_{s^k}}M\ar@{->>}[r] & M_k
}
\end{equation*}
such that we have $\Tr_{H_{\Delta,0}/H_{\Delta,k}}(M_k)=N_0F_{s^k}(M)\subseteq M$ and the image of the left vertical map equals $(A\bs N_{\Delta,k}\js \otimes_{A\bs N_{\Delta,0}\js,\varphi_{s^k}}M)^{H_{\Delta,0}}$. Dualizing and inverting $\varphi_{s^{kn_0}}(X_\Delta)$ we obtain a commutative diagram
\begin{equation*}
\xymatrix{
A\bg N_{\Delta,k}\jg \otimes_{A\bg N_{\Delta,0}\jg,\varphi_{s^k}}M^\vee[1/X_\Delta]\ar[d]_{\Tr_{H_{\Delta,0}/H_{\Delta,k}}} & M_k^\vee[1/X_\Delta]\ar[d]^{\Tr_{H_{\Delta,0}/H_{\Delta,k}}}\ar@{_{(}->}[l]_-f \\
A\bg N_{\Delta,k}\jg \otimes_{A\bg N_{\Delta,0}\jg,\varphi_{s^k}}M^\vee[1/X_\Delta] & M_k^\vee[1/X_\Delta]\ar@{_{(}->}[l]_-f
}
\end{equation*}
By Prop.\ \ref{fingendualetale} the inclusion $N_0F_{s^k}(M)\subseteq M$ induces an isomorphism $$M^\vee[1/X_\Delta]\cong A\bg N_{\Delta,0}\jg \otimes_{A\bg N_{\Delta,0}\jg,\varphi_{s^k}}M^\vee[1/X_\Delta]=(N_0F_{s^k}(M))^\vee[1/X_\Delta]\ .$$ On the other hand, the (co)image of the left vertical map is $$H_0(H_{\Delta,0}/H_{\Delta,k},A\bg N_{\Delta,k}\jg \otimes_{A\bg N_{\Delta,0}\jg,\varphi_{s^k}}M^\vee[1/X_\Delta])\cong A\bg N_{\Delta,0}\jg \otimes_{A\bg N_{\Delta,0}\jg,\varphi_{s^k}}M^\vee[1/X_\Delta]\ .$$
So $f$ becomes onto after taking $H_{\Delta,0}/H_{\Delta,k}$-coinvariants. Hence by Nakayama's Lemma $f$ is an isomorphism.
\end{proof}

Since the map \eqref{Mkdefine} is a $A\bs N_{\Delta,k}\js[F_{\Delta,k}]$-module homomorphism, we obtain a commutative diagram
\begin{equation*}
\xymatrix{
A\bg N_{\Delta,k}\jg \otimes_{A\bg N_{\Delta,0}\jg,\varphi_{s^k}}M^\vee[1/X_\Delta]\ar[d]_-{(1\otimes F_t)^\vee} & M_k^\vee[1/X_\Delta]\ar[d]^-{(1\otimes F_t)^\vee}\ar[l]_-\sim \\
A\bg N_{\Delta,k}\jg\otimes_{A\bg N_{\Delta,k}\jg,\varphi_t}A\bg N_{\Delta,k}\jg \otimes_{A\bg N_{\Delta,0}\jg,\varphi_{s^k}}M^\vee[1/X_\Delta] & A\bg N_{\Delta,k}\jg\otimes_{A\bg N_{\Delta,k}\jg,\varphi_t}M_k^\vee[1/X_\Delta]\ar[l]^-\sim
}
\end{equation*}
for all $t\in \prod_{\alpha\in\Delta}t_\alpha^{\mathbb{N}}$ with both horizontal and the left vertical arrows being isomorphisms. Therefore the right vertical arrow is also an isomorphism. In particular, $M_k^\vee[1/X_\Delta]\cong\mathbb{M}_{k,0}(M^\vee[1/X_\Delta])$ is an \'etale $T_+$-module over $A\bg N_{\Delta,k}\jg$ since we have $T_+=T_0 \prod_{\alpha\in\Delta}t_\alpha^{\mathbb{N}}$ and $T_0$ acts on $M_k^\vee[1/X_\Delta]$ by conjugation. Taking the projective limit with respect to $k\geq 0$ we define
\begin{equation*}
D^\vee_{\Delta,\infty}(\pi):=\varprojlim_{k\geq 0,M\in\mathcal{M}_\Delta(\pi^{H_{\Delta,0}})}M_k^\vee[1/X_\Delta]\ .
\end{equation*}
By construction, $M_\infty^\vee[1/X_\Delta]:=\varprojlim_k M_k^\vee[1/X_\Delta]$ is an object in $\mathcal{D}^{et}(T_+,A \bg N_{\Delta,\infty}\jg)$ that corresponds to $M^\vee[1/X_\Delta]$ under the equivalence of categories in Thm.\ \ref{equivcat}.

We call two elements $M,M'\in\mathcal{M}_\Delta(\pi^{H_{\Delta,0}})$ equivalent ($M\sim M'$) if the inclusions $M\subseteq M+M'$ and $M'\subseteq M+M'$ induce isomorphisms $M^\vee[1/X_{\Delta}]\cong (M+M')^\vee[1/X_{\Delta}]\cong {M'}^\vee[1/X_{\Delta}]$. In particular, this is an equivalence relation on the set $\mathcal{M}(\pi^{H_{\Delta,0}})$. Similarly, we say that $M_k,M'_k\in\mathcal{M}_{\Delta,k}(\pi^{H_{\Delta,k}})$ are equivalent if the inclusions $M_k\subseteq M_k+M'_k$ and $M_k'\subseteq M_k+M'_k$ induce isomorphisms $M_k^\vee[1/X_{\Delta}]\cong (M_k+M'_k)^\vee[1/X_{\Delta}]\cong {M'_k}^\vee[1/X_{\Delta}]$.
\begin{rem}
The maps
\begin{eqnarray*}
M&\mapsto& N_0\Tr_{H_{\Delta,k}/s^kH_{\Delta,0}s^{-k}}(s^kM)\\
\Tr_{H_{\Delta,0}/H_{\Delta,k}}(M_k) & \mapsfrom & M_k
\end{eqnarray*}
induce a bijection between the sets $\mathcal{M}(\pi^{H_{\Delta,0}})/\sim$ and $\mathcal{M}_{\Delta,k}(\pi^{H_{\Delta,k}})/\sim$. In particular, we have 
\begin{equation*}
D^\vee_{\Delta,\infty}(\pi)=\varprojlim_{k\geq 0}\varprojlim_{M_k\in\mathcal{M}_{\Delta,k}(\pi^{H_{\Delta,k}})}M_k^\vee[1/X_{\Delta}]\ .
\end{equation*}
\end{rem}

On a finitely generated \'etale $T_+$-module $D$ over $A\bg N_{\Delta,\infty}\jg$ we define the weak topology as follows. We put the natural compact topology on any finitely generated $A\bs N_{\Delta,k}\js$-submodule of the $A\bg N_{\Delta,k}\jg$-module $D_{H_{\Delta,k}}=H_0(H_{\Delta,k},D)$ of coinvariants and the inductive limit topology of these topologies on $D_{H_{\Delta,k}}$ (we call this also the weak topology on $D_{H_{\Delta,k}}$). Now we equip $D$ with the projective limit topology of the weak topologies on $D_{H_{\Delta,k}}$. Finally, if $\varprojlim_{i\in I}D_i$ is a projective limit of finitely generated \'etale $T_+$-modules over $A\bg N_{\Delta,\infty}\jg$ then we equip it with the projective limit topology of the weak topologies of each $D_i$.

\subsection{A natural transformation from the Schneider--Vigneras $D$-functor to $D^\vee_{\Delta,\infty}$}\label{transf}

In order to avoid confusion we denote by $D_{SV}(\pi)$ the Schneider--Vigneras functor defined as $D(\pi)$ in \cite{SVig} (note that in \cite{SVig} the notation $V$ is used for the $o$-torsion representation that we denote by $\pi$). Recall that $D_{SV}(\pi)$ is defined as the inductive limit  
\begin{equation*}
D_{SV}(\pi):=\varinjlim_{W\in \mathcal{B}_+(\pi)}W^\vee
\end{equation*}
where $\mathcal{B}_+(\pi)$ is the set of generating $B_+$-subrepresentations in $\pi_{\mid B}$. This is a $\Lambda(N_0)$-module with an action of $B_+^{-1}$. For further details we refer the reader to \cite{SVig}. The inclusions $M_k\subseteq \pi$ ($k\geq 0, M\in\mathcal{M}_\Delta(\pi^{H_{\Delta,0}})$) define a $T_0$-equivariant $A\bs N_0\js$-module homomorphism $\beta\colon\pi^\vee\to D^\vee_{\Delta,\infty}(\pi)$. Our goal is to show that $\beta$ factors through the map $\pi^\vee\twoheadrightarrow D_{SV}(\pi)$ and it is $T_+^{-1}$-equivariant. Here $T_+^{-1}$ acts on $D^\vee_{\Delta,\infty}(\pi)$ via the $\psi$-action (see section 4 of \cite{EZ}).

Let $W$ be in $\mathcal{B}_+(\pi)$ and $M\in\mathcal{M}_\Delta(\pi^{H_{\Delta,0}})$. Then by Lemma 2.1 in \cite{SVig} (see also Lemma 3.1 in \cite{EZ}) there is an integer $k(M)\geq 0$ such that for all $k\geq k(M)$ we have $M_k\leq W$. Therefore the map $\pi^\vee\twoheadrightarrow M_k^\vee\to M_k^\vee[1/X_\Delta]$ factors through $\pi^\vee\twoheadrightarrow W^\vee$. Taking projective limits with respect to $k$ we obtain a $A\bs N_0\js$-homomorphism
\begin{equation*}
\pr_{W,M}\colon W^\vee\to M_\infty^\vee[1/X_\Delta]\ .
\end{equation*}

The following is the analogue of Lemma 3.2 in \cite{EZ} in our situation with essentially the same proof.

\begin{lem}\label{prpsigamma}
The map $\pr_{W,M}$ is $\psi_t$-equivariant for all $t\in T_+$.
\end{lem}
\begin{proof}
The $T_0$-equivariance is clear as it is given by the multiplication by elements of $T_0$ on both sides. By multiplicativity we are reduced to showing the $\psi_t$-equivariance for $t=t_\alpha$ for all $\alpha$ in $\Delta$. Let $k=k(\alpha,M)>0$ be large enough so that $H_{\Delta,k}$ is contained in $t_\alpha H_{\Delta,0}t_\alpha^{-1}\leq t_\alpha N_0t_\alpha^{-1}$ and $M_k$ is contained in $W$. Denote by $H_{k,-,\alpha}$ the kernel of the group homomorphism $t_\alpha(\cdot)t_\alpha^{-1}\colon N_{\Delta,k}\to N_{\Delta,k}$ and put $\Tr:=\Tr_{H_{k,-,\alpha}}=\sum_{u\in H_{k,-,\alpha}}u\in A\bs N_{\Delta,k}\js$. By our assumption on $k$ we have $H_{k,-,\alpha}=t_\alpha^{-1} H_{\Delta,k}t_\alpha/H_{\Delta,k}$. For any $f$ in $W^\vee=\Hom_A(W,A )$ and $w\in W$ we have $\psi_{t_\alpha}(f)(w)=f(t_\alpha w)$ by definition. On the other hand, the map $1\otimes F_{\alpha,k}\colon A\bs N_{\Delta,k}\js \otimes_{\varphi_{t_\alpha},A\bs N_{\Delta,k}\js }M_k\to M_k$ factors through $A\bs N_{\Delta,k}\js \otimes_{\varphi_{t_\alpha},A\bs N_{\Delta,k}/H_{k,-,\alpha}\js }(M_k/\Ker(\Tr_{\mid M_k}))$. Moreover, $A\bs N_{\Delta,k}\js$ is a finite free module over $A\bs N_{\Delta,k/H_{k,-,\alpha}}\js$ via $\varphi_{t_\alpha}$. Hence we have a series of maps
\begin{align*}
M_k^\vee\overset{(1\otimes F_{\alpha,k})^\vee}{\to}(A\bs N_{\Delta,k}\js \otimes_{\varphi_{t_\alpha},A\bs N_{\Delta,k}/H_{k,-,\alpha}\js } (M_k/\Ker(\Tr)))^\vee\overset{\sim}{\to}\\
\overset{\sim}{\to}A\bs N_{\Delta,k}\js \otimes_{\varphi_{t_\alpha},A\bs N_{\Delta,k}/H_{k,-,\alpha}\js } (M_k/\Ker(\Tr))^\vee\overset{\Tr}{\to}\\
\overset{\Tr}{\to}A\bs N_{\Delta,k}\js \otimes_{\varphi_{t_\alpha},A\bs N_{\Delta,k}/H_{k,-,\alpha}\js } \Tr(M_k)^\vee\overset{\sim}{\to}\\
\overset{\sim}{\to}A\bs N_{\Delta,k}\js \otimes_{\varphi_{t_\alpha},A\bs N_{\Delta,k}/H_{k,-,\alpha}\js } (M_k^\vee/\Ker(\Tr))
\end{align*}
under which the image of $f=f_{\mid M_k}$ is as follows:
\begin{align*}
f\overset{(1\otimes F_{\alpha,k})^\vee}{\mapsto} (f_1\colon u\otimes (m+\Ker(\Tr))\mapsto f(uF_{\alpha,k}(m)))\overset{\sim}{\mapsto}\\
\overset{\sim}{\mapsto} \sum_{u\in J(N_{\Delta,k}/t_\alpha N_{\Delta,k}t_\alpha^{-1})}u\otimes f_{3,u}\text{ where }f_{3,u}(m+\Ker(\Tr)):=f(uF_{\alpha,k}(m));\overset{\sim}{\mapsto}\\
\overset{\sim}{\mapsto} \sum_{u\in J(N_{\Delta,k}/t_\alpha N_{\Delta,k}t_\alpha^{-1})}u\otimes f_{4,u}\text{ where }f_{4,u}(\Tr(m)):=f(ut_\alpha\Tr(m));\overset{\sim}{\mapsto}\\
\overset{\sim}{\mapsto} \sum_{u\in J(N_{\Delta,k}/t_\alpha N_{\Delta,k}t_\alpha^{-1})}u\otimes (\pr_{W,M,k}(\psi_{t_\alpha}(u^{-1}f))+\Ker(\Tr))\ .
\end{align*}
Moreover, we have an identification $M_k^\vee[1/X_\Delta]\cong \mathbb{M}_{k,0}(M^\vee[1/X_\Delta])$. In particular, $M_k^\vee[1/X_\Delta]$ is induced as a representation of $H_{k,-,\alpha}\leq H_{\Delta,0}/H_{\Delta,k}$. Therefore $(M_k^\vee/\Ker(\Tr))[1/X_\Delta]$ can be identified with the coinvariants $M_k^\vee[1/X_\Delta]_{H_{k,-,\alpha}}$. So the image of $f$ under the composite map 
\begin{align*}
W^\vee\to M_k^\vee\to A\bg N_{\Delta,k}\jg \otimes_{\varphi_{t_\alpha},A\bg N_{\Delta,k}/H_{k,-,\alpha}\jg } M_k^\vee[1/X_\Delta]_{H_{k,-,\alpha}}\cong\\
\cong A\bg N_{\Delta,k}\jg \otimes_{\varphi_{t_\alpha},A\bg N_{\Delta,k}\jg } M_k^\vee[1/X_\Delta]
\end{align*}
equals $\sum_{u\in J(N_{\Delta,k}/t_\alpha N_{\Delta,k}t_\alpha^{-1})}u\otimes \pr_{W,M,k}(\psi_{t_\alpha}(u^{-1}f))$. However, since $M_\infty^\vee[1/X]$ is an \'etale $T_+$-module over $A\bg N_{\Delta,\infty}\jg$, we have a unique decomposition of $\pr_{W,M}(f)$ as
\begin{equation*}
\pr_{W,M}(f)=\sum_{u\in J(N_0/t_\alpha N_0t_\alpha^{-1})}u\varphi(\psi_{t_\alpha}(u^{-1}\pr_{W,M}(f)))
\end{equation*}
whence we must have $\psi_{t_\alpha}(\pr_{W,M}(f))=\pr_{W,M}(\psi_{t_\alpha}(f))$ as claimed.
\end{proof}

By taking the projective limit with respect to $M\in\mathcal{M}(\pi^{H_{\Delta,0}})$ and the injective limit with respect to $W\in\mathcal{B}_+(\pi)$ we obtain a $\psi_s$- and $\Gamma$-equivariant $\Lambda(N_0)$-homomorphism 
\begin{equation*}
\pr:=\varinjlim_{W}\varprojlim_{M}\pr_{W,M}\colon D_{SV}(\pi)\to D^\vee_{\Delta,\infty}(\pi)\ .
\end{equation*}

\begin{pro}\label{1otimestildeprinj}
Let $D$ be a finitely generated \'etale $T_+$-module over $A\bg N_{\Delta,\infty}\jg$, and $f\colon D_{SV}(\pi)\to D$ be a continuous $\psi_t$-equivariant $A\bs N_0\js$-homomorphism for all $t\in T_+$. Then $f$ factors uniquely through $\pr$, ie.\ there exists a unique $\psi_t$-equivariant $A\bs N_0\js$-homomorphism $\hat{f}\colon D^\vee_{\Delta,\infty}(\pi)\to D$ such that $f=\hat{f}\circ\pr$.
\end{pro}

\begin{proof}
We prove the uniqueness first. Let $\hat{f}$ and $\hat{f}'$ be two such maps. Then the image of $\pr$ lies in the kernel of $\hat{f}-\hat{f}'\colon D^\vee_{\Delta,\infty}(\pi)\to D$. By continuity, $\hat{f}-\hat{f}'$ factors through $M^\vee_\infty[1/X_\Delta]$ for some $M\in \mathcal{M}_{\Delta}(\pi^{H_{\Delta,0}})$. Applying $A\bg N_{\Delta,k}\jg\otimes_{A\bg N_{\Delta,\infty}\jg}\cdot$ we deduce that the composite map $$\pi^\vee\twoheadrightarrow M_k^\vee\to M_k^\vee[1/X_\Delta]=A\bg N_{\Delta,k}\jg\otimes_{A\bg N_{\Delta,\infty}\jg}M^\vee_\infty[1/X_\Delta]\overset{1\otimes(\hat{f}-\hat{f}')}{\to}  A\bg N_{\Delta,k}\jg\otimes_{A\bg N_{\Delta,\infty}\jg}D$$ is $0$. Since the image of $M_k^\vee$ in $M_k^\vee[1/X_\Delta]$ generates $M_k^\vee[1/X_\Delta]$ (as a module over $A\bg N_{\Delta,k}\jg$), we deduce that the map $ M_k^\vee[1/X_\Delta]\overset{1\otimes(\hat{f}-\hat{f}')}{\to} A\bg N_{\Delta,k}\jg\otimes_{A\bg N_{\Delta,\infty}\jg}D$ is $0$. Letting $k\to\infty$ we obtain $\hat{f}=\hat{f}'$ using the remark after Thm.\ \ref{equivcat}.

At first we construct a homomorphism $\hat{f}_{H_{\Delta,0}}:D^\vee_\Delta(\pi)=(D^\vee_{\Delta,\infty}(\pi))_{H_{\Delta,0}}\to D_{H_{\Delta,0}}$ such that the following diagram commutes:
\begin{equation*}
\xymatrix{
D_{SV}(\pi)\ar[rd]_f\ar[r]^{\pr} & D^\vee_{\Delta,\infty}(\pi)\ar[r]^{(\cdot)_{H_{\Delta,0}}} & D^\vee_\Delta(\pi)\ar[d]^{\hat{f}_{H_{\Delta,0}}}\\
 & D\ar[r]_{(\cdot)_{H_{\Delta,0}}} & D_{H_{\Delta,0}}
}
\end{equation*}

Consider the composite map $f':\pi^\vee\to D_{SV}(\pi)\stackrel{f}{\to}D\to D_{H_{\Delta,0}}$. Note that $f'$ is continuous and $D_{H_{\Delta,0}}$ is Hausdorff, so $\Ker(f')$ is closed in $\pi^\vee$. Therefore $M_0:=(\pi^\vee/\Ker(f'))^\vee$ is naturally a subspace in $\pi$. We claim that $M_0$ lies in $\mathcal{M}_{\Delta}(\pi^{H_{\Delta,0}})$. Indeed, $M_0^\vee$ is a quotient of $\pi^\vee_{H_{\Delta,0}}$, hence $M_0\leq\pi^{H_{\Delta,0}}$ and it is $T_0$-invariant since $f'$ is $T_0$-equivariant. $M_0$ is admissible because it is discrete, hence $M_0^\vee$ is compact, equivalently finitely generated over $A \bs N_{\Delta,0}\js$, because $M_0^\vee$ can be identified with a $A \bs N_{\Delta,0}\js$-submodule of $D_{H_{\Delta,0}}$ which is finitely generated over $A \bg N_{\Delta,0}\jg$. Finally, $M_0$ is finitely generated over $A \bs N_{\Delta,0}\js[F_\Delta]$ by Proposition \ref{dualcompactfingen}.

Now we have an injective morphism $f_0\colon M_0^ \vee[1/X_\Delta]\hookrightarrow D$ of \'etale $T_+$-modules over $A\bg N_{\Delta,0}\jg$. The map $\hat{f}_{H_{\Delta,0}}\colon D^\vee_{\Delta}(\pi)\to D_{H_{\Delta,0}}$ is the composite map $D^ \vee_{\Delta}(\pi)\twoheadrightarrow M_0^\vee[1/X_\Delta]\hookrightarrow D$. It is well defined and makes the above diagram commutative, because the map $$\pi^\vee\to D_{SV}(\pi)\stackrel{\pr}{\to} D^\vee_{\Delta,\infty}(\pi)\stackrel{(\cdot)_{H_{\Delta,0}}}{\to}D^\vee_\Delta(\pi)\to M_0^\vee[1/X_\Delta]$$
is the same as $\pi^\vee\to M_0^\vee\to M_0^\vee[1/X_\Delta]$.

Finally, $M^\vee[1/X]$ (resp.\ $D_{H_{\Delta,0}}$) corresponds to $M_\infty^\vee[1/X_\Delta]$ (resp.\ to $D$) via the equivalence of categories in Theorem \ref{equivcat} therefore $f_0$ can uniquely be lifted to a $T_+$-equivariant $A\bg N_{\Delta,\infty}\jg$-homomorphism $f_\infty\colon M_\infty^\vee[1/X_\Delta]\hookrightarrow D$. The map $\hat{f}$ is defined as the composite $D^\vee_{\Delta,\infty}(\pi)\twoheadrightarrow M_\infty^\vee[1/X_\Delta]\hookrightarrow D$. Now the image of $f-\hat{f}\circ\pr$ is a $\psi_t$-invariant ($t\in T_+$) $A\bs N_0\js$-submodule in $(H_{\Delta,0}-1)D$. For any $x\in D_{SV}(\pi)$ and $k\geq 0$ we may write $(f-\hat{f}\circ\pr)(x)$ in the form $\sum_{u\in J(N_0/s^kN_0s^{-k})}u\varphi^k((f-\hat{f}\circ\pr)(\psi^k(u^{-1}x)))$ that lies in $(H_{\Delta,k}-1)D$. Hence we obtain $(f-\hat{f}\circ\pr)(D_{SV}(\pi))\subseteq \bigcap_k (H_{\Delta,k}-1)D=\{0\}$ as we have $D\cong\mathbb{M}_{\infty,0}(D_{H_{\Delta,0}})= \varprojlim_k D_{H_{\Delta,k}}$.
\end{proof}

By Corollary 4.11 in \cite{EZ} there exists a homomorphism $\widetilde{\pr}:\widetilde{D_{SV}}(\pi)\to D^{\vee}_{\Delta,\infty}(\pi)$ of \'etale $T_+$-modules over $A\bs N_0\js$ such that $\pr=\widetilde{\pr}\circ\iota$ where $\iota\colon D_{SV}(\pi)\hookrightarrow \widetilde{D_{SV}}(\pi)$ is the \'etale hull of $D_{SV}(\pi)$ which is defined as $\widetilde{D_{SV}}(\pi):=\varinjlim_{t\in T_+}\varphi_t^\ast D_{SV}(\pi)$ where we put $\varphi_t^\ast D_{SV}(\pi):=A\bs N_0\js_{A\bs N_0\js,\varphi_t}D_{SV}(\pi)$. For more details on the \'etale hull, in particular its universal property, we refer the reader to \cite{EZ}. The following is the natural analogue of Thm.\ 2.27 in \cite{EZ} in our situation.

\begin{cor}\label{pscompdsv}
We have
\begin{equation*}
D^\vee_{\Delta,\infty}(\pi)\cong \varprojlim_{D}D
\end{equation*}
where $D$ runs through the finitely generated \'etale $T_+$-modules over $A\bg N_{\Delta,\infty}\jg$ arising as a quotient of $A\bg N_{\Delta,\infty}\jg\otimes_{A\bs N_0\js}\widetilde{D_{SV}}(\pi)$ such that the quotient map is continuous in the weak topology of $D$ and the final topology on $A\bg N_{\Delta,\infty}\jg\otimes_{A\bs N_0\js}\widetilde{D_{SV}}(\pi)$ of the map $1\otimes\iota\colon D_{SV}(\pi)\to A\bg N_{\Delta,\infty}\jg\otimes_{A\bs N_0\js}\widetilde{D_{SV}}(\pi)$.
\end{cor}
\begin{proof}
Let $D$ be a finitely generated \'etale $T_+$-module over $A\bg N_{\Delta,\infty}\jg$ with a continuous $\psi_t$-equivariant (for all $t\in T_+$) homomorphism of $A\bs N_0\js$-modules $D_{SV}(\pi)\to D$. Then by the universal property (Prop.\ 2.20 in \cite{EZ}) of the \'etale hull this factors uniquely through $\widetilde{D_{SV}}(\pi)$ and also through $A\bg N_{\Delta,\infty}\jg\otimes_{A\bs N_0\js}\widetilde{D_{SV}}(\pi)$ since $D$ is a module over $A\bg N_{\Delta,\infty}\jg$. Therefore by Lemma \ref{prpsigamma} and Proposition \ref{1otimestildeprinj} the topological quotients of $A\bg N_{\Delta,\infty}\jg\otimes_{A\bs N_0\js}\widetilde{D_{SV}}(\pi)$ in the category of finitely generated \'etale $T_+$-modules over $A\bg N_{\Delta,\infty}\jg$ are exactly the \'etale $T_+$ modules $M_\infty^\vee[1/X_\Delta]$ for $M\in \mathcal{M}_\Delta(\pi^{H_{\Delta,0}})$.
\end{proof}

\subsection{A $G$-equivariant sheaf on $G/B$ attached to $D^\vee_{\Delta}(\pi)$}

The goal of this section is to construct a $G$-equivariant sheaf on $G/B$ with sections on $\mathcal{C}_0:=N_0w_0B/B$ isomorphic to $\widetilde{D_{\Delta}}(\pi):=\widetilde{\pr}(\widetilde{D_{SV}}(\pi))\subseteq D^\vee_{\Delta,\infty}(\pi)$ as an \'etale $T_+$-module. (Here $w_0$ denotes a representative of the element with maximal length in the Weyl group $N_G(T)/T$.) The method of constructing a $G$-equivariant sheaf on $G/B$ does not work in our situation since the property $\mathfrak{T}(1)$ in Prop.\ 6.8 in \cite{SVZ} is not satisfied in general for \'etale $T_+$-modules over $A\bg N_{\Delta,\infty}\jg$. The idea is to use the $G$-action on $\pi^\vee$ in order to construct the operators $\mathcal{H}_g\colon \widetilde{D_{\Delta}}(\pi)\to \widetilde{D_{\Delta}}(\pi)$. So let $\pi$ be a smooth representation of $G$ over $A$ and choose an arbitrary object $M$ in $\mathcal{M}_{\Delta}(\pi^{H_{\Delta,0}})$. Denote by $M_\infty$ the Pontryagin dual of the image of the natural map $\pi^\vee\to M_\infty^\vee[1/X_\Delta]$.

\begin{lem}\label{MinftyB+}
$M_\infty$ is a $B_+$-subrepresentation of $\pi$. If $M$ is chosen so that $M^\vee$ is $X_\Delta$-torsion free then we have $M_\infty=\bigcup_k M_k=B_+M$.
\end{lem}
\begin{rem}
Let $D_0$ be the image of $\pi^\vee$ in $M^\vee[1/X_\Delta]$. By Proposition \ref{dualcompactfingen} $D_0^\vee\subseteq M$ lies in $\mathcal{M}_\Delta(\pi^{H_{\Delta,0}})$ and $D_0\cong (D_0^\vee)^\vee$ is $X_\Delta$-torsion free since it is contained in $M^\vee[X_\Delta]$. Moreover, we have $D_0[1/X_\Delta]=M^\vee[1/X_\Delta]$, so we may replace $M$ by $D_0^\vee$ so that the second conclusion holds.
\end{rem}
\begin{proof}
The $N_0$-invariance of the subspace $M_\infty \subseteq \pi$ is clear. Let $t\in T_+$ and $m\in M_\infty$ be arbitrary. Assume that $tm$ does not lie in $M_\infty\leq \pi$. Then there is an element $\mu\in \pi^\vee$ such that $\mu_{\mid M_\infty}=0$ but $(t^{-1}\mu)(m)=\mu(tm)\neq 0$. By Lemma \ref{prpsigamma} we compute $0\neq (t^{-1}\mu)_{\mid M_\infty}= \pr_{W,M}(t^{-1}\mu_{\mid W})=\psi_t(\pr_{W,M}(\mu))=\psi_t(0)=0$ which is a contradiction. So $M_\infty$ is $B_+$-invariant.

Assume now that $M^\vee$ is $X_\Delta$-torsion free, ie.\ the map $M^\vee\to M^\vee[1/X_\Delta]$ is injective. Therefore $M$ is contained in $M_\infty$ since $M^\vee$ is a quotient of $M^\vee_\infty$. Hence $M_k\subseteq B_+M$ is also contained in $M_\infty$. Now assume that $\bigcup_k M_k\subsetneq M_\infty$. Then there exists an element $\mu$ in $\pi^\vee$ such that $\mu_{\mid M_k}=0$ for all $k\geq 0$ but $\mu_{\mid M_\infty}\neq 0$. In particular, the image of $\mu$ in $M_k^\vee[1/X_\Delta]$ is zero for all $k\geq 0$ whence it is also zero in $\varprojlim M_k^\vee[1/X_\Delta]=M_\infty^\vee[1/X_\Delta]$ contradicting to $\mu_{\mid M_\infty}\neq 0$.
\end{proof}

Let us denote by $\mathcal{M}^0_\Delta(\pi^{H_{\Delta,0}})$ the set of those $M\in\mathcal{M}_\Delta(\pi^{H_{\Delta,0}})$ so that $M^\vee$ has no $X_\Delta$-torsion. We still have $D_\Delta^\vee(\pi)=\varprojlim_{M\in\mathcal{M}^0_\Delta(\pi^{H_{\Delta,0}})}M^\vee[1/X_\Delta]$.

\begin{lem}\label{1otimesFsurjective}
Suppose we choose $M\in\mathcal{M}^0_\Delta(\pi^{H_{\Delta,0}})$. Then the map
\begin{equation*}
1\otimes F_s\colon A\bs N_{\Delta,0}\js\otimes_{\varphi_s,A\bs N_{\Delta,0}\js}M\to M
\end{equation*}
is surjective.
\end{lem}
\begin{proof}
The assertion is equivalent to the injectivity of the dual map
\begin{equation*}
(1\otimes F_s)^\vee\colon M^\vee\to A\bs N_{\Delta,0}\js\otimes_{\varphi_s,A\bs N_{\Delta,0}\js}M^\vee\ .
\end{equation*}
This map is, however, injective by Lemma 4.2, Prop.\ 4.5, and Cor.\ 4.8 in \cite{EZ} noting that $M^\vee$ is contained in the \'etale $T_+$-module $M^\vee[1/X_\Delta]$.
\end{proof}

For an integer $n\geq 1$ we denote by $U^{(n)}$ the kernel of the group homomorphism $\mathbf{G}(\mathbb{Z}_p)\to \mathbf{G}(\mathbb{Z}/p^n\mathbb{Z})$.

\begin{lem}\label{MinftyU1}
We have $U^{(1)}M_\infty =M_\infty$ for all $M\in\mathcal{M}^0_\Delta(\pi^{H_{\Delta,0}})$.
\end{lem}
\begin{proof}
We may assume without loss of generality that $M^\vee$ is $X_\Delta$-torsion free so that we have $M_\infty=B_+M$ by Lemma \ref{MinftyB+}. Now since $M$ lies in $\mathcal{M}_\Delta(\pi^{H_{\Delta,0}})$, it is generated by finitely many elements $m_1,\dots,m_r\in M$ as a module over $A\bs N_{\Delta,0}\js [F_\Delta]$. Since the elements of  $A\bs N_{\Delta,0}\js [F_\Delta]$ are finite linear combinations of elements in $B_+$ we deduce that $m_1,\dots,m_r$ generates $M_\infty$ as a $B_+$-subrepresentation of $\pi$. Since $\pi$ is smooth, there exists an integer $n\geq 1$ such that $m_1,\dots,m_r$ are all fixed by $U^{(n)}$. Moreover, we may assume without loss of generality that the group $T_0$ permutes the elements of the set $\{m_1,\dots,m_r\}$. By Lemmata \ref{MinftyB+} and \ref{1otimesFsurjective} we may write any element $m\in M_\infty$ as a finite linear combination $m=\sum_{i=1}^r\lambda_i s^nm_i$ with $\lambda_i$ in the monoid-ring $A[B_+]$. So we are reduced to showing that $uvts^nm_i$ lies in $M_\infty$ for all $u\in U^{(1)}$, $v\in N_0$, $t\in T_+$, and $1\leq i\leq r$. Since $U^{(1)}$ is normal in $\mathbf{G}(\Zp)$ and $N_0\subseteq \mathbf{G}(\Zp)$, the element $u_1:=v^{-1}uv$ also lies in $U^{(1)}$. By the Iwahori factorization we may write $u_1$ as a product $u_1=n_1t_1\overline{n_1}$ with $n_1\in U^{(1)}\cap N\leq N_0$, $t_1\in  U^{(1)}\cap T$, and $\overline{n_1}\in U^{(1)}\cap \overline{N}$. Therefore $s^{-n}t^{-1}\overline{n_1}ts^n$ lies in $s^{-n}t^{-1}(U^{(1)}\cap\overline{N})ts^n\leq U^{(n)}\cap\overline{N}$ whence we have $s^{-n}t^{-1}\overline{n_1}ts^nm_i=m_i$ for all $1\leq i\leq r$ by our choice of $n$. So we compute
\begin{align*}
uvts^nm_i=vu_1ts^nm_i=vn_1t_1ts^n(s^{-n}t^{-1}\overline{n_1}ts^n)m_i=vn_1t_1ts^nm_i\in B_+M=M_\infty\ .
\end{align*}
\end{proof}

Now let $g$ be an arbitrary element in $G$ and put $\mathcal{U}_g:=\{u\in N_0\mid x_u\in g^{-1}\mathcal{C}_0\cap\mathcal{C}_0\}$ where $u\mapsto x_u=uw_0B$ is the homeomorphism $N_0\to\mathcal{C}_0=N_0w_0B/B\subset G/B$. For $u\in \mathcal{U}_g$ we may factorize $gu$ as $gu=n(g,u)t(g,u)\overline{n}(g,u)$ with $n(g,u)\in N_0$, $t(g,u)\in T$, and $\overline{n}(g,u)\in \overline{N}$. By Lemma 7.3 in \cite{EZ} there exists an integer $k_0=k_0(g)$ such that for all $k\geq k_0$ and $u\in \mathcal{U}_g$ we have $us^kN_0s^{-k}\subseteq \mathcal{U}_g$, $s^kt(g,u)\in T_+$, and $s^{-k}\overline{n}(g,u)s^k\in U^{(1)}\cap \overline{N}$. Moreover, since $\mathcal{U}_g$ is compact, it is a finite union of cosets of the form $us^kN_0s^{-k}$ ($u\in\mathcal{U}_g$). We denote by $\widetilde{M_\infty^\vee}$ the \'etale hull of the image $M_\infty^\vee\subset M_\infty^\vee[1/X_\Delta]$ of the natural map $\pr_M\colon \pi^\vee\to M_\infty^\vee[1/X_\Delta]$. As a special case of Cor.\ 4.8 in \cite{EZ} we may view $\widetilde{M_\infty^\vee}$ as an \'etale $T_+$-submodule of $M_\infty^\vee[1/X_\Delta]$. More explicitly, since the powers of $s$ are cofinal in $T_+$, we have 
\begin{align*}
\widetilde{M_\infty^\vee}=\bigcup_{k\geq 0} \bigoplus_{u\in J(N_0/s^kN_0s^{-k})}u\varphi_{s^k}(M_\infty^\vee)=\bigcup_{k\geq k_0} \bigoplus_{u\in J(N_0/s^kN_0s^{-k})}u\varphi_{s^k}(M_\infty^\vee)=\\
=\bigcup_{k\geq k_0} \bigoplus_{u\in J(N_0/s^kN_0s^{-k})}\res_{us^kN_0s^{-k}}^{N_0}(M_\infty^\vee) 
\end{align*}
by Cor.\ 4.11 in \cite{EZ} noting that $\psi_s\colon M_\infty^\vee\to M_\infty^\vee$ is surjective since $(s\cdot)\colon M_\infty\to M_\infty$ is injective. Here $\res_{us^kN_0s^{-k}}^{N_0}$ denotes the operator $(u\cdot)\circ\varphi_{s^k}\circ\psi_{s^k}\circ(u^{-1}\cdot)$. More generally, for an open compact subset $\mathcal{U}\subset N_0$ we may write $\mathcal{U}$ as a disjoint union of cosets $us^kN_0s^{-k}$ ($u$ running on a finite subset $U_k\subset \mathcal{U}$) for $k\geq 0$ large enough depending on $\mathcal{U}$ and put $\res_{\mathcal{U}}^{N_0}:=\sum_{u\in U_k}\res_{us^kN_0s^{-k}}^{N_0}$. For $k\geq k_0$ we choose a set $U_{k,g}:= \mathcal{U}_g\cap J(N_0/s^kN_0s^{-k})$ of representatives of the cosets of $\mathcal{U}_g/s^kN_0s^{-k}$. By Lemma 4.13 in \cite{EZ} the map 
\begin{equation*}
n(g,\cdot)\colon us^kN_0s^{-k}\to n(g,u)t(g,u)s^kN_0s^{-k}t(g,u)^{-1} 
\end{equation*}
is a bijection for each $u\in U_{k,g}$. Under the identification $N_0\overset{\sim}{\to}\mathcal{C}_0$ the above bijection corresponds to the bijection 
\begin{equation*}
(g\cdot)\colon us^kN_0w_0B/B\overset{\sim}{\to} n(g,u)t(g,u)s^kN_0w_0B/B\ .
\end{equation*}
So we define
\begin{eqnarray*}
\mathcal{H}_g\colon \widetilde{M_\infty^\vee}&\to& \widetilde{M_\infty^\vee}\\
\sum_{u\in J(N_0/s^kN_0s^{-k})}\res_{us^kN_0s^{-k}}^{N_0}(\pr_M(\mu_{u,k}))&\mapsto& \sum_{u\in U_{k,g}}\res_{n(g,u)t(g,u)s^kN_0s^{-k}t(g,u)^{-1}}^{N_0}(\pr_M(g\mu_{u,k}))
\end{eqnarray*}
for $\mu_{u,k}$ in $\pi^\vee$ ($u\in U_{k,g}$, $k\geq k_0$).
\begin{lem}\label{Hgwelldefined}
The map $\mathcal{H}_g$ above is well-defined.
\end{lem}
\begin{proof}
Since $\res_{\mathcal{U}}^{N_0}$ does not depend on the choice of the decomposition of the compact open subset $\mathcal{U}\subseteq N_0$ into a disjoint union of cosets of subgroups of $N_0$ of the form $tN_0t^{-1}$ ($t\in T_+$), it suffices to show that $\mathcal{H}_g$ does not depend on the choice of $\mu_{u,k}$ ($u\in U_{k,g}$, $k\geq k_0$). Let $\mu$ be in $\pi^\vee$ such that $\res_{us^kN_0s^{-k}}^{N_0}(\pr_M(\mu))=0$. Since $(u\cdot)\circ\varphi_{s^k}$ is injective on $\widetilde{M_\infty^\vee}$, we obtain $\pr_M(s^{-k}u^{-1}\mu)=\psi_{s^k}(u^{-1}\pr_M(\mu))=0$ by Lemma \ref{prpsigamma}. In other words the restriction of $s^{-k}u^{-1}\mu$ to $M_\infty\subset \pi$ is zero. By our assumption $k\geq k_0$ we have $s^{-k}\overline{n}(g,u)s^k$ lies in $U^{(1)}$. Hence by Lemma \ref{MinftyU1} we compute 
\begin{align*}
0=\pr_M(s^{-k}u^{-1}\mu)=\pr(s^{-k}\overline{n}(g,u)u^{-1}\mu)=\pr(s^{-k}t(g,u)^{-1}n(g,u)^{-1}g\mu)\ .
\end{align*}
In particular, the formula defining $\mathcal{H}_g$ does not depend on the choice of $\mu_{u,k}$.
\end{proof}

\begin{pro}\label{sheafM}
For any smooth $o$-torsion representation $\pi$ of $G$ and any $M\in\mathcal{M}^0_\Delta(\pi^{H_{\Delta,0}})$ there exists a $G$-equivariant sheaf $\mathfrak{Y}_{\pi,M}$ on $G/B$ with sections $\mathfrak{Y}_{\pi,M}(\mathcal{C}_0)$ on $\mathcal{C}_0$ isomorphic to $\widetilde{M_\infty^\vee}$ as an \'etale $T_+$-module over $A\bs N_0\js$ such that we have $\mathcal{H}_g=(g\cdot)\circ\res^{\mathcal{C}_0}_{g^{-1}\mathcal{C}_0\cap \mathcal{C}_0}$ as maps on $\mathfrak{Y}_{\pi,M}(\mathcal{C}_0)$.
\end{pro}
\begin{proof}
By Prop.\ 5.14 in \cite{SVZ} we are bound to check the relations H1, H2, H3 therein: for any $g,h\in N_0\overline{B}N_0$, $b\in B\cap N_0\overline{B}N_0$ and compact open subset $\mathcal{V}\subset\mathcal{C}_0$ we have
\begin{enumerate}[H$1$]
\item $\res_{\mathcal{V}}^{\mathcal{C}_0}\circ\mathcal{H}_g=\mathcal{H}_g\circ\res_{g^{-1}\mathcal{V}\cap\mathcal{C}_0}^{\mathcal{C}_0}$ ;
\item $\mathcal{H}_g\circ\mathcal{H}_h=\mathcal{H}_{gh}\circ\res_{(gh)^{-1}\mathcal{C}_0\cap h^{-1}\mathcal{C}_0\cap\mathcal{C}_0}^{\mathcal{C}_0}$ ;
\item $\mathcal{H}_b=b\circ \res_{b^{-1}\mathcal{C}_0\cap \mathcal{C}_0}^{\mathcal{C}_0}$ .
\end{enumerate}
Let $d$ be in $\widetilde{M_\infty^\vee}$ and let $k\geq \max(k_0(g),k_0(h),k_0(gh))$ be large enough so that all the sets $g^{-1}\mathcal{V}\cap \mathcal{C}_0$, $\mathcal{U}_g=g^{-1}\mathcal{C}_0\cap \mathcal{C}_0$, $(gh)^{-1}\mathcal{C}_0\cap h^{-1}\mathcal{C}_0\cap\mathcal{C}_0$, $\mathcal{U}_h$, and $\mathcal{U}_b$ can be written as the disjoint union of sets of the form $us^kN_0w_0B/B$ for $u\in N_0$ and so that $d$ is contained in $N_0\varphi^k(M_\infty^\vee)$. So we may write $d=\sum_{u\in J(N_0/s^kN_0s^{-k})}\res_{us^kN_0s^{-k}}^{N_0}(\pr_M(\mu_{u,k}))$ for some $\mu_{u,k}\in\pi^\vee$ ($u\in J(N_0/s^kN_0s^{-k})$).

H1:  We distinguish two cases: If $x_u$ does not lie in $g^{-1}\mathcal{V}\cap\mathcal{C}_0$ then we have $us^kN_0w_0B\cap g^{-1}\mathcal{V}=\emptyset$ whence the right hand side of H1 has value zero at $\res_{us^kN_0s^{-k}}^{N_0}(\pr_M(\mu_{u,k}))$. Moreover, we obtain $n(g,u)t(g,u)s^kN_0w_0B\cap \mathcal{V}=gus^kN_0w_0B\cap \mathcal{V}=\emptyset$ (Lemma 7.4 in \cite{EZ}), so we the left hand side of H1 also has value zero at $\res_{us^kN_0s^{-k}}^{N_0}(\pr_M(\mu_{u,k}))$. On the other hand, if $x_u$ lies in $g^{-1}\mathcal{V}\cap\mathcal{C}_0$ then we have $us^kN_0w_0B\subseteq g^{-1}\mathcal{V}\cap \mathcal{C}_0$ and $n(g,u)t(g,u)s^kN_0w_0B\subseteq \mathcal{V}\cap g\mathcal{C}_0$ whence both sides have value $\mathcal{H}_g(\res_{us^kN_0s^{-k}}^{N_0}(\pr_M(\mu_{u,k})))$ at $\res_{us^kN_0s^{-k}}^{N_0}(\pr_M(\mu_{u,k}))$.

H2: We distinguish two cases: If $x_u$ does not lie in $(gh)^{-1}\mathcal{C}_0\cap h^{-1}\mathcal{C}_0\cap\mathcal{C}_0$ then we clearly have $$\mathcal{H}_{gh}\circ\res_{(gh)^{-1}\mathcal{C}_0\cap h^{-1}\mathcal{C}_0\cap\mathcal{C}_0}^{\mathcal{C}_0}(\res_{us^kN_0s^{-k}}^{N_0}(\pr_M(\mu_{u,k})))=0\ .$$ Further, $\mathcal{H}_h(\res_{us^kN_0s^{-k}}^{N_0}(\pr_M(\mu_{u,k})))$ is supported on $hus^kN_0w_0B$ that is disjoint to $\mathcal{U}_g$ whence we have $$\mathcal{H}_{g}\circ\mathcal{H}_h(\res_{us^kN_0s^{-k}}^{N_0}(\pr_M(\mu_{u,k})))=0\ ,$$ too. On the other hand, if $x_u$ lies in $(gh)^{-1}\mathcal{C}_0\cap h^{-1}\mathcal{C}_0\cap\mathcal{C}_0$ then we compute
\begin{align*}
\mathcal{H}_{gh}\circ\res_{(gh)^{-1}\mathcal{C}_0\cap h^{-1}\mathcal{C}_0\cap\mathcal{C}_0}^{\mathcal{C}_0}(\res_{us^kN_0s^{-k}}^{N_0}(\pr_M(\mu_{u,k})))=\\
=\mathcal{H}_{gh}(\res_{us^kN_0s^{-k}}^{N_0}(\pr_M(\mu_{u,k})))=\res_{ghus^kN_0w_0B}^{\mathcal{C}_0}(\pr_M(gh\mu_{u,k}))\ .
\end{align*} 
Moreover, we write $hus^kN_0w_0B$ as the disjoint union of sets of the form $vs^lN_0w_0B$ ($v$ running on a subset $V\subseteq J(N_0/s^lN_0s^{-l})$) for some $l\geq k$. So we compute
\begin{align*}
\mathcal{H}_g\circ\mathcal{H}_h(\res_{us^kN_0s^{-k}}^{N_0}(\pr_M(\mu_{u,k})))=\mathcal{H}_g\circ\res_{hus^{k}N_0w_0B}^{\mathcal{C}_0}(\pr_M(h\mu_{u,k}))=\\
=\sum_{v\in V}\mathcal{H}_g\circ\res_{vs^{l}N_0w_0B}^{\mathcal{C}_0}(\pr_M(h\mu_{u,k}))=\sum_{v\in V}\res_{gvs^{l}N_0w_0B}^{\mathcal{C}_0}(\pr_M(gh\mu_{u,k}))=\\
=\res_{ghus^{k}N_0w_0B}^{\mathcal{C}_0}(\pr_M(gh\mu_{u,k}))\ .
\end{align*}

H3: Let $b=v_1tv_2$ be in $B\cap N_0\overline{B} N_0=N_0TN_0$ with $v_1,v_2\in N_0$ and $t\in T$. We may write $t=t_1^{-1}t_2$ for some $t_1,t_2\in T_+$. Recall from \cite{SVZ} that the action of $b$ on $\widetilde{M_\infty^\vee}$ is defined by the formula $b(d)=v_1\psi_{t_1}\circ\varphi_{t_2}(v_2d)$. This does not depend on the choice of $t_1$ and $t_2$ as we have $\psi_{t'}\circ\varphi_{t'}=\id$ for all $t'\in T_+$. If $x_u$ does not lie in $b^{-1}\mathcal{C}_0\cap \mathcal{C}_0$ then both sides of H3 vanish at $\res_{us^kN_0s^{-k}}^{N_0}(\pr_M(\mu_{u,k}))$, so assume we have $x_u\in b^{-1}\mathcal{C}_0\cap\mathcal{C}_0$. By the choice of $k$ and $u$, $ts^k$ lies in $T_+$ and $tv_2ut^{-1}$ lies in $N_0$. We compute
\begin{align*}
\mathcal{H}_b(\res_{us^kN_0s^{-k}}^{N_0}(\pr_M(\mu_{u,k})))=\res_{(v_1tv_2ut^{-1})ts^kN_0w_0B}^{\mathcal{C}_0}(\pr_M(v_1tv_2\mu_{u,k}))=\\
=v_1tv_2ut^{-1}\varphi_{ts^k}\circ\psi_{ts^k}(tu^{-1}v_2^{-1}t^{-1}v_1^{-1}\pr_M(v_1tv_2\mu_{u,k}))=\\
=v_1tv_2ut^{-1}\psi_{t_1}\circ\varphi_{t_2s^k}(\pr_M(s^{-k}t^{-1}tu^{-1}v_2^{-1}t^{-1}v_1^{-1}v_1tv_2\mu_{u,k}))=\\
=v_1\psi_{t_1}(t_1tv_2ut^{-1}t_1^{-1}\varphi_{t_2}\circ\varphi_{s^k}(\pr_M(s^{-k}u^{-1}\mu_{u,k})))=\\
=v_1\psi_{t_1}\circ\varphi_{t_2}(v_2u\varphi_{s^k}\circ\psi_{s^k}(u^{-1}\pr_M(\mu_{u,k})))=\\
=v_1\psi_{t_1}\circ\varphi_{t_2}(v_2\res_{us^kN_0s^{-k}}^{N_0}(\pr_M(\mu_{u,k})))=b\circ\res_{us^kN_0s^{-k}}^{N_0}(\pr_M(\mu_{u,k}))
\end{align*}
as desired.
\end{proof}

We equip the space $\mathfrak{Y}_{\pi,M}(G/B)$ of global sections with the coarsest topology such that the restriction maps $\mathfrak{Y}_{\pi,M}(G/B)\to \mathfrak{Y}_{\pi,M}(g\mathcal{C}_0)$ are continuous for all $g\in G$. This topology is Hausdorff by Lemma 4.9 in \cite{EZ}.

We denote by $\beta_{\mathcal{C}_0,M}$ the composite map $\pi^\vee\to M_\infty^\vee\to \widetilde{M_\infty^\vee}=\mathfrak{Y}_{\pi,M}(\mathcal{C}_0)$. For each $g\in G$ and $\mu\in \pi^\vee$ we define $\beta_{g\mathcal{C}_0,M}(\mu):=g\cdot\beta_{\mathcal{C}_0,M}(g^{-1}\mu)$. Writing $g^{-1}\mathcal{C}_0\cap\mathcal{C}_0$ as a disjoint union of sets of the form $us^kN_0w_0B$ with $u\in U$ for some $k\geq k_0(g)$ and $U:= J(N_0/s^kN_0s^{-k})\cap \mathcal{U}_g$, we compute
\begin{align*}
\res^{g\mathcal{C}_0}_{\mathcal{C}_0\cap g\mathcal{C}_0}(\beta_{g\mathcal{C}_0,M}(\mu))=\res^{g\mathcal{C}_0}_{\mathcal{C}_0\cap g\mathcal{C}_0}(g\beta_{\mathcal{C}_0,M}(g^{-1}\mu))=g\res^{\mathcal{C}_0}_{g^{-1}\mathcal{C}_0\cap \mathcal{C}_0}(\beta_{\mathcal{C}_0,M}(g^{-1}\mu))=\\
=\mathcal{H}_g(\pr_M(g^{-1}\mu))=\sum_{u\in U}\res_{gus^kN_0w_0B}^{\mathcal{C}_0}(\pr_M(\mu))=\\
=\res^{\mathcal{C}_0}_{\mathcal{C}_0\cap g\mathcal{C}_0}(\pr_M(\mu))=\res^{\mathcal{C}_0}_{\mathcal{C}_0\cap g\mathcal{C}_0}(\beta_{\mathcal{C}_0,M}(\mu))\ .
\end{align*}
Since the compact open subsets $g\mathcal{C}_0$ ($g\in G$) cover $G/B$, the maps $\beta_{g\mathcal{C}_0,M}\colon \pi^\vee\to \mathfrak{Y}_{\pi,M}(g\mathcal{C}_0)$ glue together into a continuous $G$-equivariant map $\beta_{G/B,M}\colon \pi^\vee\to \mathfrak{Y}_{\pi,M}(G/B)$ with $\beta_{g\mathcal{C}_0,M}=\res_{g\mathcal{C}_0}^{G/B}\circ\beta_{G/B,M}$ for all $g\in G$. 
\begin{cor}\label{sheaffunctorial}
The map $\beta_{G/B,M}\colon \pi^\vee\to \mathfrak{Y}_{\pi,M}(G/B)$ is natural in the pair $(\pi,M)$ in the following sense: for a morphism $f\colon \pi\to\pi'$ of smooth representations of $G$ and $M\in \mathcal{M}^0_\Delta(\pi^{H_{\Delta,0}})$, the image $f(M)$ lies in $\mathcal{M}^0_\Delta(\pi'^{H_{\Delta,0}})$. Moreover, for any $M'\in \mathcal{M}^0_\Delta(\pi'^{H_{\Delta,0}})$ containing $f(M)$ we have a commutative square
\begin{equation*}
\xymatrix{
\pi'^\vee\ar[rr]^-{\beta_{G/B,M'}}\ar[d] && \mathfrak{Y}_{\pi',M'}(G/B)\ar[d]\\
\pi'^\vee\ar[rr]^-{\beta_{G/B,M}} && \mathfrak{Y}_{\pi,M}(G/B)
}
\end{equation*}
\end{cor}
Now we fix $\pi$ and vary $M$ in $\mathcal{M}^0_\Delta(\pi^{H_{\Delta,0}})$ to obtain
\begin{cor}\label{deltasheaf}
There exists a $G$-equivariant sheaf $\mathfrak{Y}_{\pi,\Delta}$ on $G/B$ with sections $\mathfrak{Y}_{\pi,\Delta}(\mathcal{C}_0)$ isomorphic to $\widetilde{D_{\Delta}}(\pi)=\widetilde{\pr}(\widetilde{D_{SV}}(\pi))$ together with a natural $G$-equivariant continuous map $\beta_{G/B,\Delta}\colon \pi^\vee\to \mathfrak{Y}_{\pi,\Delta}(G/B)$.
\end{cor}
\begin{proof}
Note that the image of $\pi^\vee$ in $D^\vee_\Delta(\pi)$ is the Pontryagin dual of the union $\bigcup_{M\in \mathcal{M}^0_\Delta(\pi^{H_{\Delta,0}})}M_\infty$. Therefore by definition $\widetilde{D_{\Delta}}(\pi)$ is the \'etale hull of $(\bigcup_{M\in \mathcal{M}^0_\Delta(\pi^{H_{\Delta,0}})}M_\infty)^\vee=\varprojlim_{M\in \mathcal{M}^0_\Delta(\pi^{H_{\Delta,0}})}M_\infty^\vee$. So we may conclude the existence of the operators $\mathcal{H}_g$ as in Lemma \ref{Hgwelldefined}. The existence of the sheaf follows the same way as in Prop.\ \ref{sheafM}.
\end{proof}

\begin{rem}
Let $\pi$ be an irreducible admissible smooth representation of $G$ over $\kappa$ and assume that $D^\vee_\Delta(\pi)\neq 0$. Then we can realize $\pi^\vee$ as a $G$-invariant subspace in the global sections of a ``small'' $G$-equivariant sheaf $\mathfrak{Y}$ on $G/B$. By small we mean here that the section $\mathfrak{Y}(\mathcal{C}_0)$ is contained in a \emph{finitely generated} module over $\kappa\bg N_{\Delta,\infty}\jg$. Indeed, we may take any $0\neq M\in \mathcal{M}^0_\Delta(\pi^{H_{\Delta,0}})$ and put $\mathfrak{Y}:=\mathfrak{Y}_{\pi,M}$. It would be natural to expect that the irreducibility of $\pi$ would imply the irreducibility of $D^\vee_\Delta(\pi)$. In particular, this would mean that $D^\vee_\Delta(\pi)$ is in fact finitely generated. We end this section by presenting some very preliminary ideas towards proving this.
\end{rem}

We choose a total ordering of the Weyl group $N_G(T)/T$ refining the Bruhat order. This gives a decreasing filtration of $G/B$ by open $B$-invariant subsets $\Fil^{w}(G/B):=\bigcup_{w_1\geq w}Bw_1B/B\subseteq G/B$ for $w\in N_G(T)/T$. Its bottom term is $\Fil^{w_0}(G/B)$ for the element $w_0\in N_G(T)/T$ of maximal length and we have $\Fil^1(G/B)=G/B$. For each $w\in N_G(T)/T$ we define $$\Fil^w_M(\pi^\vee):=\beta_{G/B,M}^{-1}(\Ker(\res_{\Fil^w(G/B)}^{G/B}\colon \mathfrak{Y}_{\pi,M}(G/B)\to  \mathfrak{Y}_{\pi,M}(\Fil^w(G/B))))\ .$$
This is an increasing filtration of $\pi^\vee$ by closed $B$-subrepresentations. Taking Pontryagin duals we obtain a decreasing filtration $\Fil^w_M(\pi):=(\pi^\vee/\Fil^w_M(\pi^\vee))^\vee\leq \pi$ whose graded pieces we denote by $\mathrm{gr}^w_M(\pi)$.
\begin{lem}
We have $\mathrm{gr}_{M}^{w_0}(\pi)=\bigcup_{n\geq 0} s^{-n}M_\infty$ and  $D_{SV}(\mathrm{gr}_{M}^{w_0}(\pi))\cong M_\infty^\vee$.
\end{lem}
\begin{proof}
Note that $M_\infty^\vee$ can be identified with $\res_{\mathcal{C}_0}^{G/B}(\beta_{G/B,M}(\pi^\vee))$. On the other hand, we write $\mathcal{C}:=Nw_0B/B=\Fil^{w_0}(G/B)=\bigcup_{n\geq 0} s^{-n}\mathcal{C}_0$. So we compute
\begin{align*}
\mathrm{gr}_{M}^{w_0}(\pi)^\vee=\res^{G/B}_{\Fil^{w_0}}(\beta_{G/B,M}(\pi^\vee))=\varprojlim_n\res^{G/B}_{s^{-n}\mathcal{C}_0}\circ\beta_{G/B,M}(\pi^\vee)=\varprojlim_n\beta_{s^{-n}\mathcal{C}_0,M}(\pi^\vee)=\\
=\varprojlim_n s^{-n}\cdot\beta_{\mathcal{C}_0,M}(\pi^\vee)=\varprojlim_{\psi_s\colon M_\infty^{\vee}\to M_\infty^\vee}M_\infty^\vee=(\varinjlim_{s^n\colon M_\infty\to M_\infty}M_\infty)^\vee=(\bigcup_{n}s^{-n}M_\infty)^\vee\ .
\end{align*} 

This shows, in particular, that $M_\infty$ is a generating $B_+$-subrepresentation of $\mathrm{gr}_{M}^{w_0}(\pi)$. Let $W\subseteq M_\infty$ be another generating $B_+$subrepresentation. By Lemma \ref{MinftyB+} the $B_+$-subrepresentation $M_\infty$ is generated by a finite set $\{m_1,\dots,m_r\}$ of elements in $M$. By Lemma 2.1 in \cite{SVig} there exists an integer $k\geq 0$ such that $s^km_i$ lies in $W$ for all $1\leq i\leq r$. In particular, we have $M\subset M_k=\Tr_{H_{\Delta,k}/s^kH_{\Delta,0}s^{-k}}(s^kM)\subseteq B_+\{s^km_1,\dots,s^km_r\}\subseteq W$ whence we also deduce $M_\infty=B_+M\subseteq W$. So $M_\infty$ is a minimal generating $B_+$-subrepresentation of $\mathrm{gr}_{M}^{w_0}(\pi)$, hence we have $D_{SV}(\mathrm{gr}_{M}^{w_0}(\pi))=M_\infty^\vee$ by definition.
\end{proof}

\begin{que}
In what generality is it true that $D^\vee_\Delta(\mathrm{gr}_{M}^{w_0}(\pi))\cong M^\vee[1/X_{\Delta}]$?
\end{que}

\begin{rem}
The answer to the above question is affirmative if $\pi$ is a principal series.
\end{rem}

\subsection{The fully faithful property of $D^\vee_\Delta$ for principal series}

We denote by $SP_A^0$ the category of those finite length smooth representations of $G$ over $A$ whose irreducible subquotients are (necessarily irreducible) principal series representations.

\begin{lem}\label{splitDsplitpi}
Let $\chi$ and $\chi'$ be two (not necessarily distinct) characters $T\to \kappa^\times$ such that both $\Ind_B^G\chi$ and $\Ind_B^G\chi'$ are irreducible. Then the natural map
\begin{equation*}
\Ext^1_G(\Ind_B^G\chi',\Ind_B^G\chi)\to \Ext^1(D^\vee_\Delta(\Ind_B^G\chi), D^\vee_\Delta(\Ind_B^G\chi'))
\end{equation*}
is injective.
\end{lem}
\begin{proof}
We distinguish two cases whether $\chi=\chi'$ or not. Assume first $\chi\neq \chi'$. By Thm.\ 1.1 in \cite{H} if the left hand side is nonzero then we have $\chi'=s_\alpha(\chi)\cdot\varepsilon^{-1}\circ\alpha$ for some simple root $\alpha\in \Delta$ where $s_\alpha(\chi)$ denotes the conjugate of $\chi$ with respect to the reflection corresponding to $\alpha$ in the Weyl group $N_G(T)/T$ and $\varepsilon$ is the modulo $p$ cyclotomic character. By our assumption that $\Ind_B^G\chi$ is irreducible, we have $s_\alpha(\chi)\neq\chi$ so we can only have $\chi'=s_\alpha(\chi)\cdot\varepsilon^{-1}\circ\alpha$ for at most one element $\alpha\in\Delta$. Moreover, in this case the unique nonsplit extension $\pi$ comes from parabolic induction from $G_\alpha\cong TGL_2(\mathbb{Q}_p)$ generated by $T$, $N_\alpha$, and $s_\alpha$. Therefore even Breuil's functor $D^\vee_\Delta(\pi)\cong\kappa\bg X\jg\otimes_{\ell,\kappa\bg N_{\Delta,0}\jg}D^\vee_\Delta(\pi)$ (see Cor.\ \ref{reduceellSP}) is a non-split extension of $(\varphi,\Gamma)$-modules by Thm.\ 6.1 in \cite{B} and Thm.\ VII.5.2 in \cite{C}.

Now assume $\chi=\chi'$. Then by Thm.\ 1.2 in \cite{H} the natural map $$\Ext^1_T(\chi,\chi)\to \Ext^1_G(\Ind_B^G\chi,\Ind_B^G\chi)$$ is bijective. However, the composite map $\Ext^1_T(\chi,\chi)\to \Ext^1(D^\vee_\Delta(\Ind_B^G\chi), D^\vee_\Delta(\Ind_B^G\chi))$ is injective since whenever $\delta$ is an extension of $\chi$ by itself then we have the identification $D^\vee_\Delta(\Ind_B^G\delta)\cong \kappa\bg N_{\Delta,0}\jg\otimes \delta^\vee$ and we may recover $\delta$ from this by taking the Pontryagin dual of the intersection $$\bigcap_{n>0}\varphi_s^n(\kappa\bg N_{\Delta,0}\jg\otimes \delta^\vee)\cong \left(\bigcap_{n>0}\varphi_s^n(\kappa\bg N_{\Delta,0}\jg)\right)\otimes \delta^\vee\cong\delta^\vee\ .$$
\end{proof}

\begin{rem}
We do not know whether or not the above map between extension groups is also surjective. For this one should determine the dimension of the right hand side. The author plans to turn back to this question in a future work.
\end{rem}

\begin{lem}\label{notinsmaller}
Let $\pi$ be an object in the category $SP_A^0$ and $D$ be an object in $\mathcal{D}^{et}(T_+,A\bg N_{\Delta,0}\jg )$ that is a successive extension of rank $1$ objects. Assume that the length of $\pi$ is bigger than that of $D$. Then $\pi^\vee$ cannot be embedded into the global sections of a $G$-equivariant sheaf $\mathfrak{Y}$ with an injection $\mathfrak{Y}(\mathcal{C}_0)\hookrightarrow \mathbb{M}_{\infty,0}(D)$.
\end{lem}
\begin{proof}
Assume it can. By Prop.\ \ref{1otimestildeprinj} and Thm.\ \ref{equivcat} (or simply by Prop.\ \ref{dualcompactfingen}) the composite map $$\pi^\vee\hookrightarrow \mathfrak{Y}(G/B)\twoheadrightarrow\mathfrak{Y}(\mathcal{C}_0)\hookrightarrow \mathbb{M}_{\infty,0}(D)\twoheadrightarrow \mathbb{D}_{0,\infty}\circ\mathbb{M}_{\infty,0}(D)\cong D$$ factors uniquely through $D^\vee_\Delta(\pi)$. In particular, there exists an $M\in \mathcal{M}^0_\Delta(\pi^{H_{\Delta,0}})$ such that we have the factorization $\pi^\vee\twoheadrightarrow M^\vee\hookrightarrow M^\vee[1/X_\Delta]\hookrightarrow D$. We may assume without loss of generality that we have $M^\vee[1/X_\Delta]\cong D$. By assumption we have a filtration on $\pi$ whose subquotients are principal series. By Theorem \ref{princserexact} this induces a filtration on $D^\vee_\Delta(\pi)$ whose subquotients are $1$-dimensional objects. Further, this induces a filtration on $D$ via the surjective map $D^\vee_\Delta(\pi)\twoheadrightarrow D$ whence we also obtain a filtration on the sheaf $\mathfrak{Y}$ and its global sections $\mathfrak{Y}(G/B)$. Therefore passing to one particular subquotient we may assume without loss of generality that $\pi$ has length $2$ and $D$ is $1$ dimensional. So $\pi$ is an extension of $2$ principal series for characters $\chi$ and $\chi'$. In particular, we have two surjective maps $D^\vee_\Delta(\pi)\twoheadrightarrow D$ and $D^\vee_\Delta(\pi)\twoheadrightarrow D^\vee_\Delta(\Ind_B^G\chi)$. The kernels of these maps must be different since the map $\pi^\vee\to \mathfrak{Y}_{\pi,M_\chi}(G/B)$ has $(\Ind_B^G\chi')^\vee$ in its kernel but $\pi^\vee\to \mathfrak{Y}(G/B)$ is injective. (Here $M_\chi$ denotes the unique element in $\mathcal{M}_\Delta^0((\Ind_B^G\chi)^{H_{\Delta,0}})$.) In particular, $D^\vee_\Delta(\pi)$ splits as a direct sum of these two kernels whence by Lemma \ref{splitDsplitpi} $\pi\cong \Ind_B^G\chi\oplus \Ind_B^G\chi'$. Therefore the projection to $D$ has kernel $D^\vee_\Delta(\Ind_B^G\chi)$ that contradicts to the assumption that $\pi^\vee\to \mathfrak{Y}(G/B)$ is injective.
\end{proof}

Our main result in this section is the following

\begin{thm}
The restriction of $D^\vee_\Delta$ to the category $SP_A^0$ is fully faithful.
\end{thm}
\begin{proof}
The faithfulness is clear by the exactness (Thm.\ \ref{princserexact}) noting that $D^\vee_\Delta$ does not vanish on irreducible objects in $SP_A^0$ by Cor.\ \ref{princserDDelta}. Let $\pi$ and $\pi'$ be objects in the category $SP_A^0$ and assume we have a nonzero morphism $f\colon D^\vee_\Delta(\pi)\to D^\vee_\Delta(\pi')$. Let $D$ be the image of $f$ and $n$ be the (generic) length of $D$. Consider the sheaf $\mathfrak{Y}_{\pi,\Delta}$ on $G/B$ provided by Cor.\ \ref{deltasheaf}. By Prop.\ \ref{genlengthDveeSPA} $\mathfrak{Y}_{\pi,\Delta}=\mathfrak{Y}_{\pi,M}$ for some $M$ in $\mathcal{M}_\Delta(\pi^{H_{\Delta,0}})$ and we have $\genlength(M^\vee[1/X_\Delta])=\mathrm{length}(\pi)$. Further, the image of the composite map $M^\vee\to M^\vee[1/X_\Delta]=D^\vee_\Delta(\pi)\overset{f}{\to} D$ is the Pontryagin dual of another object $M_1\leq M$ in $\mathcal{M}_\Delta(\pi^{H_{\Delta,0}})$. Since $f$ is onto, $D=M_1^\vee[1/X_\Delta]$. By \ref{sheaffunctorial} we obtain a morphism $\mathfrak{Y}_{\pi,M}\to \mathfrak{Y}_{\pi,M_1}$ of $G$-equivariant sheaves whose kernel has sections
\begin{equation*}
\mathfrak{Y}_{\pi,\Ker(f)}(U):=\{x\in \mathfrak{Y}_{\pi,M}(U)\mid \res^{gU}_{gU\cap\mathcal{C}_0}(gx)\in\Ker(f)\subseteq M^\vee[1/X_\Delta]\text{ for all }g\in G\}
\end{equation*}
on an open subset $U\subseteq G/B$. Let $f_{G/B}\colon\pi^\vee\to\mathfrak{Y}_{\pi,M_1}(G/B)$ be the natural map and $\pi_1^\vee$ be its image and $\pi_2^\vee$ be its kernel. Then $\pi_1:=(\pi_1^\vee)^\vee$ is naturally a $G$-subrepresentation of $\pi$ with quotient $\pi_2=(\pi_2^\vee)^\vee$. By Lemma \ref{notinsmaller} applied to both $\pi_1^\vee\hookrightarrow\mathfrak{Y}_{\pi,M_1}(G/B)$ and $\pi_2^\vee\hookrightarrow \mathfrak{Y}_{\pi,\Ker(f)}(G/B)$ we obtain $\mathrm{length}(\pi_1)\leq n$ and $\mathrm{length}(\pi_2)\leq \genlength(\Ker(f))=\mathrm{length}(\pi)-n$. Hence $\mathrm{length}(\pi_1)=n$ and $D^\vee_\Delta(\pi_1)\cong D$. 

By a similar argument applied to the quotient map $D^\vee_\Delta(\pi')\twoheadrightarrow\Coker(f)$ we find a quotient $\pi'\twoheadrightarrow \pi'_2$ with an injective morphism $\pi_2'^\vee\hookrightarrow\mathfrak{Y}_{\pi,D}(G/B)$ into the global sections of a sheaf $\mathfrak{Y}_{\pi,D}$ whose sections on an open subset $U\subseteq G/B$ are given by
\begin{equation*}
\mathfrak{Y}_{\pi,D}(U):=\{x\in \mathfrak{Y}_{\pi',\Delta}(U)\mid \res^{gU}_{gU\cap\mathcal{C}_0}(gx)\in D\subseteq D^\vee_\Delta(\pi')\text{ for all }g\in G\}\ .
\end{equation*}
Further, we have $\mathrm{length}(\pi'_2)=n$ and $D^\vee_\Delta(\pi'_2)\cong D$.

Now we have maps $\pi_1^\vee\to D^\vee_\Delta(\pi_1)\cong D$ and $\pi_2'^\vee\to D^\vee_\Delta(\pi_2')\cong D$ so we may take the direct sum $\theta\colon \pi_1^\vee\oplus\pi_1'^\vee\to D$. This gives rise to an object $M_\ast\in \mathcal{M}_\Delta((\pi_1\oplus\pi_2')^{H_{\Delta,0}})$ and a $G$-equivariant continuous map $\overline{\theta}\colon \pi_1^\vee\oplus \pi_2'^\vee\to \mathfrak{Y}_{\pi_1\oplus\pi_1',M_\ast}(G/B)$ that is injective restricted to both $\pi_1^\vee$ and to $\pi_1'^\vee$. However, applying Lemma \ref{notinsmaller} to this situation again we deduce that the image of $\overline{\theta}$ has length at most $n$. Therefore $\overline{\theta}$ induces an isomorphism between $\pi_1^\vee$ and $\pi_1'^\vee$ (and the image). All in all we obtain a map $\pi'\twoheadrightarrow \pi_1'\cong \pi_1\hookrightarrow \pi$ that induces $f$ after taking $D^\vee_\Delta$.
\end{proof}

\begin{rems}\begin{enumerate}
\item In particular, the forgetful functor restricting $\pi$ to $B$ is also fully faithful on $SP_A^0$ as $D^\vee_\Delta$ factors through this.
\item $D^\vee_\Delta$ is not faithful on $SP_A$: for instance any finite dimensional representation lies in its kernel.
\end{enumerate}
\end{rems}


\begin{thebibliography}{33}





\bibitem{Be} Berger L., Repr\'esentations modulaires de $\GL_2(\Qp)$ et repr\'esentations galoisiennes de dimension $2$, \emph{Ast\'erisque} \textbf{330} (2010), 263--279.

\bibitem{Ber} Berger L., Multivariable Lubin-Tate $(\varphi,\Gamma)$-modules and filtered $\varphi$-modules, \emph{Math.\ Res.\ Lett.}\ \textbf{20} (2013), no. 3, 409--428.

\bibitem{Ber2} Berger L., Multivariable $(\varphi,\Gamma)$-modules and locally analytic vectors, \href{http://perso.ens-lyon.fr/laurent.berger/articles/PGMLAV.pdf}{preprint} (2016), to appear in \emph{Duke Math.\ Journal}

\bibitem{B03a} Breuil Ch., Sur quelques repr\'esentations modulaires et $p$-adiques de $\GL_2(\Qp)$ I, \emph{Compositio Math.} \textbf{138} (2003), 165--188.

\bibitem{B03b} Breuil Ch., Sur quelques repr\'esentations modulaires et $p$-adiques de $\GL_2(\Qp)$ II, \emph{J.\ Inst.\ Math.\ Jussieu} \textbf{2} (2003), 1--36.

\bibitem{BE} Breuil Ch., Emerton M., Repr\'esentations $p$-adiques ordinaires de $\GL_2(\Qp)$ et compatibilit\'e local-global, \emph{Ast\'erisque} \textbf{331} (2010), 255--315.

\bibitem{B1} Breuil Ch., The emerging $p$-adic Langlands programme, in:\ \emph{Proceedings of the International Congress of Mathematicians} Volume II, Hindustan Book Agency, New Delhi (2010), 203--230.

\bibitem{B} Breuil Ch., Induction parabolique et $(\varphi,\Gamma)$-modules, \emph{Algebra \& Number Theory} \textbf{9}(10) (2015), 2241--2291.

\bibitem{B16p} Breuil Ch., private communication, 11/03/2016

\bibitem{BH} Breuil Ch., Herzig F., Ordinary representations of $G(\mathbb{Q}_p)$ and fundamental algebraic representations, \emph{Duke Math.\ Journal} \textbf{164} (2015), 1271--1352.

\bibitem{BP} Breuil Ch., V.\ Pa\v{s}k\=unas, Towards a modulo $p$ Langlands correspondence for $\GL_2$, \emph{Memoirs of the AMS}, \textbf{216} (2012).


\bibitem{Mira}  Colmez P., $(\varphi, \Gamma)$-modules et repr\'esentations du mirabolique de $GL_{2}(\mathbb{Q}_{p})$, \emph{Ast\'erisque} \textbf{330} (2010),  61--153.

\bibitem{C}  Colmez P., Repr\'esentations de $GL_{2}(\mathbb{Q}_{p})$ et $(\varphi, \Gamma)$-modules, \emph{Ast\'erisque} \textbf{330} (2010),  281--509.



\bibitem{E1} Emerton M., On a class of coherent rings, with applications to the smooth representation theory of $GL_2(\mathbb{Q}_p)$ in characteristic $p$, \href{http://www.math.uchicago.edu/~emerton/pdffiles/frob.pdf}{preprint} (2008). 

\bibitem{E} Emerton M., Ordinary parts of admissible representations of $p$-adic reductive groups I.\ Definition and first properties, \emph{Ast\'erisque} \textbf{331} (2010), 355--402.

\bibitem{Er} Erd\'elyi M., On the Schneider--Vigneras functor for principal series, \emph{J.\ of Number Theory} \textbf{162} (2016), 68--85.

\bibitem{EZ} Erd\'elyi M., Z\'abr\'adi G., Links between generalized Montr\'eal functors, preprint (2015), \href{http://arxiv.org/abs/1412.5778}{arXiv:1412.5778}

\bibitem{F} Fontaine J.-M., Repr\'esentations $p$-adiques des corps locaux, in ``The Grothendieck
Festschrift'', vol 2, Prog.\ in Math. 87, 249--309, Birkh\"auser 1991.


\bibitem{H} Hauseux J., Complements sur les extensions entre series principales $p$-adiques et modulo $p$ de $G(F)$, preprint (2014), \href{http://arxiv.org/abs/1407.4630}{arXiv:1407.4630}

\bibitem{Ked} Kedlaya K., Some slope theory for multivariate Robba rings, preprint (2013), \href{http://arxiv.org/abs/1311.7468}{arXiv:1311.7468}

\bibitem{K1} Kisin M., Deformations of $G_{\Qp}$ and $\GL_2(\Qp)$ representations, \emph{Ast\'erisque} \textbf{330} (2010), 511--528.

\bibitem{LvO} Li H., van Oystaeyen F., \emph{Zariskian Filtrations}, $K$-monographs in Mathematics, vol.\ 2, Kluwer Academic Publishers, Dordrecht, 1996.

\bibitem{McCR} McConnell J.\ C., Robson J.\ C., \emph{Noncommutative Noetherian Rings}, Graduate Studies in Mathematics, vol.\ \textbf{30}, AMS, Providence, Rhode Island, 1987.

\bibitem{P} Pa\v{s}k\={u}nas V., The image of Colmez's Montreal functor, \emph{Publ.\ Math.}\ IHES \textbf{118}(1) (2013), 1--191.


\bibitem{SVig}   Schneider P., Vigneras M.-F., A functor from smooth $o$-torsion representations to $(\varphi, \Gamma)$-modules, Volume in honour of F.\ Shahidi, Clay Mathematics Proceedings Volume \textbf{13} (2011), 525--601.

\bibitem{SVZ} Schneider P., Vigneras M.-F., Z\'abr\'adi G., From \'etale $P_+$-representations to $G$-equivariant sheaves on $G/P$, in:\ \emph{Automorphic forms and Galois representations} (Volume 2), LMS Lecture Note Series \textbf{415} (2014), Cambridge Univ.\ Press, 248--366.


\bibitem{V} Vigneras M.-F., S\'erie principale modulo $p$ de groupes r\'eductifs $p$-adiques, \emph{Geom.\ and Funct.\ Analysis} \textbf{17} (2008), 2090--2112.




\bibitem{Z14} Z\'abr\'adi G., $(\varphi,\Gamma)$-modules over noncommutative overconvergent and Robba rings, \emph{Algebra \& Number Theory} \textbf{8}(1) (2014), 191--242.

\bibitem{Z16} Z\'abr\'adi G., Multivariable $(\varphi,\Gamma)$-modules and products of Galois groups, \href{http://www.cs.elte.hu/~zger/MultVarGal.pdf}{preprint} (2016)
\end{thebibliography}
\end{document}